%% file: Main.tex
\documentclass[11pt,reqno]{amsart}
\usepackage{a4wide}

\input{preamble}

\begin{document}
	\title[Cohomogeneity one actions on symmetric spaces of mixed type]{Cohomogeneity one actions on symmetric spaces \\ of mixed type}
	
	\author{Tom\'as Otero}
	\address{Mathematisches Institut, Universität Münster, Einsteinstr. 62, 48149 Münster, Germany}
	\email{tomas.otero@uni-muenster.de}
	
	\author{Ivan Solonenko}
	\address{Institut für Geometrie und Topologie, Universität Stuttgart, Stuttgart, Germany}
	\email{ivan.solonenko@mathematik.uni-stuttgart.de}
	
	\author{Hiroshi Tamaru}
	\address{Department of Mathematics, Osaka Metropolitan University, Sugimoto, Sumiyoshi-ku, Osaka 585-8585,
		Japan}
	\email{tamaru@omu.ac.jp}
	
	\input{0-Abstract} 
	
	\keywords{Isometric actions, cohomogeneity one actions, symmetric spaces}

	\subjclass[2020]{53C35, 57S20}
	
	\maketitle 
	
	\section{Introduction and main results}
	\input{1-Introduction}
    \section{Preliminaries}\label{section:structure}
	\input{2-Preliminaries}
	
	\section{Isometric actions on products}\label{section:products}
	\input{3-Isometric_actions_on_products}
	
	\section{Codimension-one homogeneous foliations on symmetric spaces}\label{section:HFSS}
	\input{4-C1_homogeneous_foliations}
	
	\section{Actions with singular orbits on symmetric spaces of nonpositive curvature}\label{section:singorb}
	\input{5-Actions_with_singular_orbits}

    \section{The non-simply connected case}\label{section:nonsimplyconnected}
    \input{6-The_non-simply_connected_case}

\input{bibliocustom}
\end{document}

%% file: Preamble.tex
\usepackage{amsmath,amsfonts,amssymb,amscd,amsbsy}
\usepackage{enumerate}
\usepackage{xcolor}
\usepackage{hyperref}
\usepackage{tikz}
\usepackage{array,multirow}
\usepackage{thmtools}
\usepackage{mathtools}
\usepackage{cite}
\usepackage{tikz-cd}
\usepackage{pict2e}
\usepackage{stackengine}

\makeatletter
\def\thmhead@plain#1#2#3{%
	\thmname{#1}\thmnumber{\@ifnotempty{#1}{ }\@upn{#2}}%
	\thmnote{ {\the\thm@notefont#3}}}
\let\thmhead\thmhead@plain
\makeatother

\newtheorem{theorem}{Theorem}[section]
\newtheorem{lemma}[theorem]{Lemma}
\newtheorem{proposition}[theorem]{Proposition}
\newtheorem{corollary}[theorem]{Corollary}
\newtheorem*{corollary*}{Corollary} 

\newtheorem{thmx}{Theorem}

\theoremstyle{definition}
\newtheorem{remark}[theorem]{Remark}
\newtheorem{example}[theorem]{Example}

\newtheorem*{ackn*}{Acknowledgments}

\theoremstyle{remark}

\newcommand{\R}{\ensuremath{\mathbb{R}}}
\newcommand{\Z}{\ensuremath{\mathbb{Z}}}

\newcommand{\g}[1]{\ensuremath{\mathfrak{#1}}}
\newcommand{\pr}{\operatorname{pr}}
\newcommand{\s}[1]{\ensuremath{\mathsf{#1}}}

\DeclareMathOperator{\Ad}{Ad}

\DeclareMathOperator{\ad}{ad}
\DeclareMathOperator{\Id}{Id}
\newcommand{\restr}[2]{{\left.\kern-\nulldelimiterspace #1 \vphantom{\big|} \right|_{#2}}} % restriction of a map symbol
\DeclareMathOperator{\Aut}{Aut}
\DeclareMathOperator{\tr}{tr}

\newcommand{\bilin}[2]{\langle \hspace{1pt} {#1} \hspace{0.5pt} , {#2} \hspace{1pt} \rangle \hspace{0.02em}}

\makeatletter
\DeclareRobustCommand{\loplus}{\mathbin{\mathpalette\dog@lsemi{+}}}
\DeclareRobustCommand{\roplus}{\mathbin{\mathpalette\dog@rsemi{+}}}

\newcommand{\dog@rsemi}[2]{\dog@semi{#1}{#2}{-90,90}}
\newcommand{\dog@lsemi}[2]{\dog@semi{#1}{#2}{270,90}}
\newcommand{\dog@semi}[3]{%
	\begingroup
	\sbox\z@{$\m@th#1#2$}%
	\setlength{\unitlength}{\dimexpr\ht\z@+\dp\z@\relax}%
	\makebox[\wd\z@]{\raisebox{-\dp\z@}{%
			\begin{picture}(1,1)
				\linethickness{\variable@rule{#1}}
				\roundcap
				\put(0.5,0.5){\makebox(0,0){\raisebox{\dp\z@}{$\m@th#1#2$}}}
				\put(0.5,0.5){\arc[#3]{0.5}}
			\end{picture}%
	}}%
	\endgroup
}
\newcommand{\variable@rule}[1]{%
	\fontdimen8  
	\ifx#1\displaystyle\textfont3\else
	\ifx#1\textstyle\textfont3\else
	\ifx#1\scriptstyle\scriptfont3\else
	\scriptscriptfont3\relax
	\fi\fi\fi
}
\makeatother

\newcommand\altxrightarrow[2][0pt]{\mathrel{\ensurestackMath{\stackengine%
  {\dimexpr#1-7.5pt}{\xrightarrow{\phantom{#2}}}{\scriptstyle\!#2\,}%
  {O}{c}{F}{F}{S}}}}
\newcommand{\isoto}{\altxrightarrow[1pt]{\sim}}

\DeclareFontFamily{U}{mathx}{}
\DeclareFontShape{U}{mathx}{m}{n}{<-> mathx10}{}
\DeclareSymbolFont{mathx}{U}{mathx}{m}{n}
\DeclareMathAccent{\widecheck}{0}{mathx}{"71}

%% file: 0-Abstract.tex
\begin{abstract}
        In this article, we study isometric cohomogeneity-one actions on symmetric spaces of mixed type, i.e., those whose universal cover splits as a nontrivial product of symmetric spaces of compact, noncompact, and Euclidean types. We provide a new family of ``diagonal'' cohomogeneity-one actions on symmetric spaces of the form $\R^n \times M_-$, where $M_-$ is of noncompact type. We show that, with the exception of this family, any cohomogeneity-one action on a symmetric space decomposes as a product of isometric actions on its compact, Euclidean, and noncompact factors. This fully reduces the classification problem for cohomogeneity-one actions to symmetric spaces of a single type.
\end{abstract}

%% file: 1-Introduction.tex
A smooth isometric action of a Lie group on a Riemannian manifold is of cohomogeneity one if its regular orbits are hypersurfaces. Cohomogeneity-one actions have long been a topic of great interest in Riemannian geometry. For instance, they can be used to construct Riemannian metrics with various special properties, such as Einstein metrics \cite{Boehm}, metrics with special holonomy \cite{BrySal,FH}, or metrics of positive curvature \cite{GWZ}: the usual procedure consists of modifying the given metric on a space $M$ along the orbits of a cohomogeneity-one action $\s{H} \curvearrowright M$ (e.g., by means of rescaling or a Cheeger deformation) to produce a new $\s{H}$-invariant metric with the desired properties. Cohomogeneity-one actions are also of significance from the viewpoint of submanifold geometry: their orbits form an isoparametric family, and the principal orbits are homogeneous hypersurfaces and thus have constant principal curvatures.

A natural question to ask in this regard is whether one can list all possible cohomogeneity-one actions on a given Riemannian manifold. By design, this kind of investigation makes particular sense for homogeneous spaces, since the existence of such actions calls for a large isometry group. To make this into a classification problem, one needs a suitable notion of equivalence. One such notion---natural from the standpoint of submanifold geometry---is that of \textit{orbit equivalence}: two actions are \textit{orbit-equivalent} if the corresponding orbit foliations of the ambient space are isometrically congruent.

Historically, the first spaces to receive a classification of cohomogeneity-one actions were the Euclidean and real hyperbolic spaces. In these manifolds, every isoparametric family is given by the orbits of a cohomogeneity-one action, so the classification follows from the classical works of Segre \cite{Segre} and Cartan \cite{Cartan} on isoparametric hypersurfaces. The same approach does not work for other spaces: for example, round spheres admit isoparametric hypersurfaces that are not homogeneous~\cite{FKM}. A classification of cohomogeneity-one actions on the spheres was achieved with different methods by Hsiang and Lawson in \cite{Hsiang-Lawson}.

Obtaining a complete classification for general symmetric spaces is still an ongoing effort, with many contributions over the last twenty five years.
For instance, a classification of cohomogeneity-one actions on irreducible symmetric spaces of compact type was achieved by Kollross \cite{Kollross:tams}. 
For reducible symmetric spaces of compact type, the problem is still open (see \cite{Kollross:reducible}). In the noncompact setting, this problem turned out to be more involved. In fact, it was not until very recently that a complete classification of cohomogeneity-one actions on symmetric spaces of noncompact type was achieved by the second author and Sanmart\'in-L\'opez in \cite{SLS:arXiv}. This result was a culmination of a long list of contributions by several authors \cite{BB:crelle,BDT:jdg,BD:tg,BT:jdg,BT:tohoku,BT:tams,BT:crelle,DDO:adv,DDR:crelle,Solonenko,Solonenko:reducible} (see \cite{DDO:reag} for a recent survey on the topic). It should be noticed that, unlike in the compact case, the classification for symmetric spaces of noncompact type was accomplished even in the case when the ambient space is reducible (see \cite{DDO:adv}).

So far, cohomogeneity-one actions on symmetric spaces have been studied on a type-by-type basis. Recall that the universal Riemannian covering space of a symmetric space splits as a Riemannian product $M_+ \times M_0 \times M_-$, where $M_0$ is a Euclidean space and $M_+$ and $M_-$ are symmetric spaces of compact and noncompact types, respectively. In this article, we consider symmetric spaces for which the above decomposition is nontrivial, which we call \textit{of mixed type}. We study cohomogeneity-one actions on such spaces via the behavior of the action on each individual factor.

Our main result asserts that with the only exception of a certain family of “diagonal” actions (which we parameterize explicitly), any other cohomogeneity-one action on a symmetric space of mixed type $M_+\times M_0 \times M_-$ decomposes: that is, it has the
same orbits as some product action $\s{H}_+ \times \s{H}_0 \times \s{H}_- \curvearrowright M_+\times M_0 \times M_-$ (see Section~\ref{section:products} for a more rigorous treatment of decomposability of actions in the general context of Riemannian products). This fully reduces the problem of classifying cohomogeneity-one actions on any symmetric space to that on its compact, Euclidean, and noncompact factors.
	
	\begin{thmx}\label{mainth}
		Let $M = M_+ \times M_0 \times M_-$ be a simply connected symmetric space, where $M_0 \cong \R^n$ for some $n \in \mathbb{N}$ and $M_+$ and $M_-$ are symmetric spaces of compact and noncompact type, respectively.
		Let $\s{H}$ be a connected Lie group acting on $M$ properly and isometrically with cohomogeneity one.
		Then, either the action of $\s{H}$ decomposes, or it is orbit-equivalent to the action of $I(M_+) \times \s{H}_{E,X}$, where $\s{H}_{E,X} \subseteq I^0(M_0) \times I^0(M_-)$ is a Lie subgroup inducing a homogeneous codimension-one foliation of diagonal type on $M_0 \times M_-$ as described below.
	\end{thmx}
	
    \begin{remark}
    Note that the theorem is formulated only for simply connected symmetric spaces. In Section \ref{section:nonsimplyconnected}, we explain how to dispense with this restriction and give a version of Theorem~\ref{mainth} for non-simply connected symmetric spaces (see Theorem \ref{th:nonsimplyconnected}).
    \end{remark}
    
	\begin{remark}
        By construction---which we lay out below---the action $\s{H}_{E,X} \curvearrowright M_0 \times M_-$ can only be defined when both $M_0$ and $M_-$ are nontrivial. If either of them is trivial, any cohomogeneity-one action on $M$ decomposes.
	\end{remark}
	
Let us introduce some context and notation and define the action of $\s{H}_{E,X}$ appearing in Theorem~\ref{mainth} (more details about the structure of symmetric spaces can be found in Section~\ref{section:structure}). Any Riemannian symmetric space can be written as a quotient $M = \s{G}/\s{K}$, where $\s{G} = I^0(M)$ and $\s{K}$ is the isotropy subgroup of $\s{G}$ at some fixed base point $o \in M$. The isometry Lie algebra $\g{g}$ admits a natural reductive decomposition $\g{g} = \g{k} \oplus \g{p}$ with $\g{k} = \mathrm{Lie}(K)$ and $\g{p} \cong T_oM$. If $M$ is simply connected, it splits into a Riemannian product $M = M_+ \times M_0 \times M_-$ as described above, and we have the corresponding splittings $\s{G} = \s{G}_+ \times \s{G}_0 \times \s{G}_-$, $\s{K} = \s{K}_+ \times \s{K}_0 \times \s{K}_-$, $\g{g} = \g{g}_+ \oplus \g{g}_0 \oplus \g{g}_-$, etc. Here $\g{g}_-$ is a noncompact real semisimple Lie algebra and $\g{g}_-=\g{k}_-\oplus\g{p}_-$ is a Cartan decomposition. We fix a restricted root space decomposition $\g{g}_-=\bigoplus\g{g}_\alpha$ of $\g{g}_-$ and consider the corresponding Iwasawa decomposition $\s{G}_- = \s{K}_-\s{A}\s{N}$. The solvable subgroup $\s{A} \ltimes \s{N}$ of $\s{G}_-$ acts on $M_-$ simply transitively and provides the so-called \emph{solvable model} of $M_-$. Note that the action of the translation group $\R^n \trianglelefteq I^0(M_0)$ on $M_0 \cong \R^n$ is also simply transitive. Let $\g{r}^n, \g{a},$ and $\g{n}$ stand for the Lie algebras of $\R^n, \s{A},$ and $\s{N}$, respectively, and pick any nonzero vectors $E \in \g{r}^n$ and $X \in \g{a}$. By using the properties of the root space decomposition, it can be easily checked that the subspace
	\[
		\mathfrak{h}_{E,X} = (\mathfrak{r}^{n} \ominus E) \oplus \mathbb{R}(E + X) \oplus (\mathfrak{a} \ominus X) \oplus \mathfrak{n},
	\]
is a codimension-one subalgebra of $\g{r}^n\oplus(\g{a}\loplus\g{n})$, where by $\g{v}\ominus Y$ we mean the orthogonal complement of a vector $Y$ in a subspace $\g{v}$ of $\g{r}^n\oplus(\g{a}\loplus\g{n})\cong T_o(M_0\times M_-)$ with respect to the inner product induced by the metric on $M_0 \times M_-$. The connected Lie subgroup $\s{H}_{E,X}$ of $I(M_0 \times M_-)$ with Lie algebra $\g{h}_{E,X}$ acts on $M_0 \times M_-$ with cohomogeneity one and without singular orbits. As we will show in Section \ref{section:HFSS}, the orbits of $\s{H}_{E,X}$ form a family of mutually isometrically congruent hypersurfaces, all of them of constant mean curvature $\frac{||E||}{||X|| \sqrt{||E||^2 + ||X||^2}} \sum_{\alpha \in \Sigma^+} \dim(\g{g}_\alpha) \alpha(X)$. The orbit foliation of $\s{H}_{E,X}$ will be called a \emph{homogeneous codimension-one foliation of diagonal type} on $M_0 \times M_-$. Figure~\ref{fig:hex} shows what such a foliation looks like in the case of $M_0 = \R$ and $M_- = \R H^2$.
	
    \begin{figure}[h!]\label{fig:hex}
		\includegraphics[width=0.3\textwidth]{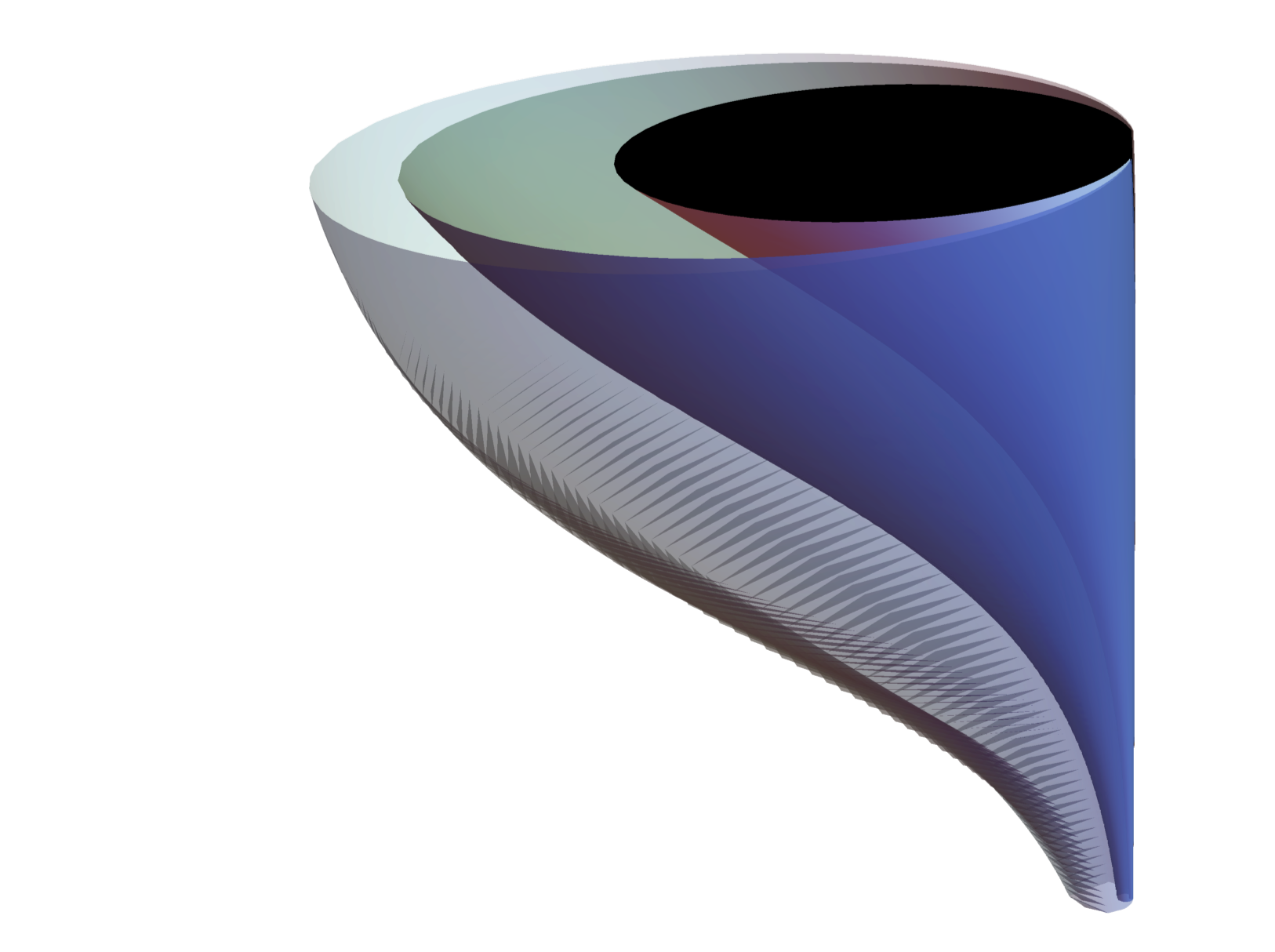}
        \caption{A homogeneous foliation of diagonal type on $\R \times \R H^2$.
        Here $\R H^2$ is represented by the Poincaré disk model in the horizontal plane.}
	\end{figure}
	
The road towards Theorem~\ref{mainth} consists of three main steps. First, we study whether a given isometric action on a Riemannian product decomposes based on its behavior on the factors. As an application, we will obtain the following:
	\begin{thmx}\label{th:main:decompcompactxhadamard}
		Let $M$ be a simply connected compact Riemannian manifold, $N$ a Hadamard manifold, and $\s{H}$ a connected Lie group acting on $M \times N$ properly and isometrically. If the action has cohomogeneity one, then it decomposes.
	\end{thmx}
Consider now a simply connected symmetric space $M = M_+ \times M_0 \times M_-$. Since $M_+$ is compact and $M_0 \times M_-$ is Hadamard, Theorem~\ref{th:main:decompcompactxhadamard} implies that we can reduce the study of cohomogeneity-one actions on $M$ to that on $M_+$ and $M_0 \times M_-$. This motivates us to study cohomogeneity-one actions on $M_0 \times M_-$ and investigate which of them decompose. For our second step, we work under the assumption that such an action has no singular orbits and thus induces a homogeneous Riemannian foliation. We are able to show that if the action does not decompose, it must be orbit-equivalent to the action of $\s{H}_{E,X}$ for some choices of $E$ and $X$ as defined above. This allows us to give a complete classification of homogeneous codimension-one foliations on any (simply connected) symmetric space:

	\begin{thmx}\label{theorem:foliations}
		Let $M = M_+ \times M_0 \times M_-$ be a simply connected symmetric space, where  $M_+ = \s{G}_+/\s{K}_+$ is of compact type, $M_0 \cong \R^n$ for some $n \in \mathbb{N}$, and $M_- = \s{G}_-/\s{K}_-$ is of noncompact type. 
		Let $\s{H}$ be a connected Lie group acting on $M$ properly, isometrically, with cohomogeneity one and no singular orbits.
		Then, the action of $\s{H}$ is orbit-equivalent to one of the following:
		\begin{enumerate}[\normalfont (a)]
			\item The action of $\s{G}_+\times\R^{n-1}\times \s{G}_-$.
			\item The action of $\s{G}_+\times\mathbb{R}^n\times \s{H}_{X}$, 
			where $\s{H}_{X}$ induces a codimension-one foliation of horospherical type on $M_-$. 
			\item The action of $\s{G}_+\times\mathbb{R}^n\times \s{H}_{i}$, 
			where $\s{H}_{i}$ induces a codimension-one foliation of solvable type on $M_-$
			\item The action of $\s{G}_+\times \s{H}_{E,X}$, 
			where $\s{H}_{E,X}$ induces a codimension-one foliation of diagonal type on $M_0 \times M_-$. 
		\end{enumerate}
	\end{thmx}
    
	\begin{remark}
        If any of the factors in the decomposition $M = M_+ \times M_0 \times M_-$ are trivial, it may not be possible to define some of the actions appearing above, and one should read Theorem~\ref{theorem:foliations} by removing such items from the list.
        For instance, if $M_0$ is nontrivial and $M_- = \{*\}$, then the action of $\s{H}$ is orbit-equivalent to the action described in (a).
		If $M_0 = \{*\}$ and $M_-$ is nontrivial, then the action of $\s{H}$ is orbit-equivalent to one of the actions described in (b) and (c).
		Finally, if both $M_0$ and $M_-$ are trivial, none of the actions in (a)--(d) can be defined and $M = M_+$ admits no homogeneous codimension-one foliations.
	\end{remark}
    
Note that the action in (a) comes from the action of $\R^{n-1}$ on $M_0 \cong \R^{n}$ by translations, whose orbits are parallel hyperplanes. This is the the only homogeneous codimension-one foliation on the Euclidean space up to congruence. Homogeneous codimension-one foliations on symmetric spaces of noncompact type have been studied extensively in \cite{BDT:jdg,BT:jdg,Solonenko:reducible} and correspond to two families of codimension-one subgroups of the solvable model of $M_-$. More precisely, the actions giving rise to so-called foliations of \emph{horospherical type} in (b) are given by connected Lie subgroups $\s{H}_X \subseteq \s{A} \ltimes \s{N}$ with Lie algebras of the form
	\[
	\g{h}_X = (\g{a} \ominus X) \loplus \g{n}
	\]
for some nonzero $X \in \g{a}$. The orbits of such an action are mutually congruent hypersurfaces. For certain choices of $X$, the orbits of $\s{H}_X$ are horospheres (see \cite[Rem.\ 5.4]{DVSLT}). The actions producing foliations of \emph{solvable type} in (c) are characterized by the fact that they have exactly one minimal orbit. They are given by connected Lie subgroups $\s{H}_i \subseteq \s{A} \ltimes \s{N}$ with Lie algebras
	\[
	\g{h}_{i} = \g{a} \loplus (\g{n} \ominus X)
	\]
for some  nonzero vector $X$ in a simple root space $\g{g}_{\alpha_i}$. Finally, the actions in (d) correspond to the new family of examples $\s{H}_{E,X}$ described above.

\begin{remark}\label{remark:hyperpolar}
    The actions of $\R^{n-1} \times \s{G}_-, \, \R^n \times \s{H}_X$, and $\s{H}_{E,X}$ on $M_0\times M_-$ can be viewed as part of a broader family of homogeneous foliations on $M_0\times M_-$:
    namely, if $\g{u}$ is any subspace of $\g{r}^n\oplus\g{a}$, then $\g{h}_\g{u}=\left((\g{r}^n\oplus\g{a})\ominus \g{u}\right)\loplus\g{n}$ is a subalgebra of $\g{r}^n\oplus(\g{a}\loplus\g{n})$.
    The action of the corresponding connected Lie subgroup $\s{H}_\g{u}$ of $\R^n\times (\s{A}\ltimes\s{N})$ on $M_0\times M_-$ gives rise to a (regular) homogeneous hyperpolar foliation with $\s{U}\cdot o$ as a section (\emph{hyperpolar} meaning that the action is polar and its sections are flat), where $\s{U}$ is the connected Lie subgroup of $\R^n \times \s{A}$ with Lie algebra $\g{u}$.
    Moreover, since $\g{h}_{\g{u}}$ is an ideal in $\g{r}^n\oplus(\g{a}\loplus\g{n})$, the orbits of $\s{H}_{\g{u}}$ are all mutually congruent (see \cite{KT:GeomDed}).
    The action of $\s{H}_{\g{u}}$ decomposes if and only if $\g{u} = (\g{u} \cap \g{r}^n) \oplus (\g{u} \cap \g{a})$. In this case, the orbits of $\s{H}_\g{u}$ are of the form $W \times F$, where $W$ is an affine subspace of $M_0 \cong \R^n$ parallel to $\R^n \ominus (\s{U} \cap \R^n)$, and $F$ is one of the leaves of the foliation $\mathcal{F}_{\Phi,V}$ described in \cite{BDT:jdg} with $\Phi = \emptyset$ and $V = \g{a}\ominus(\g{u}\cap\g{a})$.
\end{remark}
	
For our third and final step, we show that a cohomogeneity-one action on $M_0 \times M_-$ that has a singular orbit must necessarily decompose. This, together with Theorems~\ref{th:main:decompcompactxhadamard} and~\ref{theorem:foliations}, completes the proof of Theorem~\ref{mainth}.

\begin{thmx}\label{th:singorbdecomp}
	Let $M_0 \cong \R^n$ and let $M_- = \s{G}_-/\s{K}_-$ be a symmetric space of noncompact type.
	Let $\s{H}$ be a connected Lie group acting on $M_0 \times M_-$ properly and isometrically with cohomogeneity one.
	If $\s{H}$ has a singular orbit, then its action decomposes.
\end{thmx}

Theorem~\ref{mainth} can be interpreted as a classification result for homogeneous hypersurfaces, as these are precisely the regular orbits of isometric cohomogeneity-one actions. In fact, one can use it to derive explicit classifications of homogeneous hypersurfaces in a vast number of symmetric spaces. The only obstacle on the way to a complete classification for general symmetric spaces arises from the yet mysterious behavior of cohomogeneity-one actions on products of symmetric spaces of compact type (see~\cite{Kollross:reducible}). Homogeneous hypersurfaces provide the simplest example of hypersurfaces with constant principal curvatures.
Such hypersurfaces have recently been studied in \cite{ChavesSantos,MSSV} for the ambient spaces $\mathbb{S}^n \times \R$ and $\R \times \R H^n$ (which are symmetric spaces of mixed type). In turn, hypersurfaces with constant principal curvatures belong to the much larger class of hypersurfaces of constant mean curvature, whose study in products of space forms has become a lively area of research after the seminal work of Abresch and Rosenberg \cite{AR:Acta}. Homogeneous hypersurfaces have another remarkable geometric property: they are \emph{isoparametric} (that is, their locally defined nearby parallel hypersurfaces have constant mean curvature). The study and classification of isoparametric hypersurfaces is also a topic of great interest in Riemannian geometry. In recent years, there has been a particular surge of interest towards such hypersurfaces in products of space forms (see \cite{DVManzano,LimaPipoli,LimaPipoli2,SantosSantos,TXY:arxiv}). It should be noted that the only examples obtained in \cite{ChavesSantos,DVManzano,LimaPipoli,LimaPipoli2,MSSV,SantosSantos,TXY:arxiv} that do not split as products of the form $P_M \times N$ or $M \times P_N$ for some hypersurfaces $P_M \subseteq M$ and $P_N \subseteq N$ are precisely the orbits of $\s{H}_{E,X} \curvearrowright \R^n \times \R H^m$ for some $E$ and $X$.
	
\begin{example}[(\textit{Homogeneous hypersurfaces in products of space forms})]
	Consider a connected properly embedded homogeneous hypersurface $P \subseteq \mathbb{S}^n \times \R^m \times \R H^k$, where $n, k \ne 1$. A straightforward application of Theorem~\ref{mainth} yields that either $P$ splits as a product $P_{\mathbb{S}^n} \times \R^m \times \R H^k$, \, $\mathbb{S}^n \times P_{\R^m} \times \R H^k$, or $\mathbb{S}^n \times \R^m \times P_{\R H^k}$ (where $P_{\mathbb{S}^n}$, $P_{\R^m}$, $P_{\R H^k}$ are homogeneous hypersurfaces in $\mathbb{S}^n$, $\R^m$ and $\R H^k$, respectively), or $P$ is congruent to $\mathbb{S}^n \times P_{\Delta}$, where $P_\Delta$ is a leaf of a homogeneous codimension-one foliation of diagonal type on $\R^m \times \R H^k$.
		
	As a special case, if $P$ is a homogeneous hypersurface in $\mathbb{S}^n \times \R^m$ (resp., $\mathbb{S}^n \times \R H^k$), Theorem~\ref{mainth} implies that $P = P_{\mathbb{S}^n} \times \R^m$ or $P = \mathbb{S}^n \times P_{\R^m}$ (resp., $P = P_{\mathbb{S}^n} \times \R H^k$ or $P = \mathbb{S}^n \times P_{\R H^k}$), with $P_{\mathbb{S}^n}$, $P_{\R^m}$, and $P_{\R H^k}$ as above.
	Similarly, a homogeneous hypersurface $P \subseteq \R^m \times \R H^k$ either is a product $P_{\R^m} \times \R H^k$ or $\R^m \times P_{\R H^k}$, or is congruent to a leaf $P_{\Delta}$ of a homogeneous codimension-one foliation of diagonal type. It follows from \cite{Hsiang-Lawson} that $P_{\mathbb{S}^n}$ is congruent to a principal orbit of the action obtained by restricting the isotropy representation of a symmetric space of rank two to the unit sphere. By \cite{Segre}, $P_{\R^m}$ is a hyperplane, a sphere, or a (generalized) cylinder. Lastly, we know from \cite{Cartan} that $P_{\R H^k}$ is a horosphere, a totally geodesic $\R H^{k-1}$ or an equidistant hypersurface to it, a tube around a totally geodesic $\R H^l$, or a geodesic sphere.	
\end{example}

\begingroup
\linespread{1.05}\selectfont

Having dealt with Theorem \ref{mainth}, we turn our attention to cohomogeneity-one actions on non-simply connected symmetric spaces. For an isometric action of a connected Lie group on a Riemannian manifold $M$, we introduce a simple way to \textit{lift} the action to the universal Riemannian cover $\widetilde{M}$ of $M$. Conversely, we investigate when an action on $\widetilde{M}$ can \textit{descend} to one on $M$. We show that these two procedures are mutually inverse and respect the cohomogeneity. We also look at how they interact with the orbits and orbit equivalence and inspect when they preserve the properness of an action. This machinery allows us to reduce the study and classification of proper isometric actions on $M$ to those on $\widetilde{M}$, provided that we understand the group $\Aut(\pi)$ of deck transformations of $\pi \colon \widetilde{M} \to M$ and how it embeds into $I(\widetilde{M})$. In the special case when $M$ is a symmetric space, this group is well understood and looks particularly simple. This allows us to prove Theorem \ref{th:nonsimplyconnected}, which can be regarded as an analog of Theorem~\ref{mainth} for non-simply connected symmetric spaces.

One of the main technical difficulties we encounter throughout the article is having to deal with the properness assumption on an action. For example, investigating when the descent of a proper action is proper comprises a big part of Section \ref{section:nonsimplyconnected}. There are also several moments when it becomes unavoidable to consider actions that may fail to be proper. One such example is Lemma~\ref{lemma:SplittingC1}, another is our proof of Theorem \ref{th:singorbdecomp} in Subsection \ref{subsection:singorb:singorb}. In fact, our approach to Theorem \ref{th:singorbdecomp} requires us to consider arbitrary (not necessarily proper) transitive isometric actions on Euclidean spaces. Combining the ideas from ~\cite{Alekseevski,AW:homogeneous,GW:isomsolv} with Alekseevskii's description of flat metric Lie algebras \cite[Prop.\ 3.1]{Alekseevski}, we obtain a general classification of transitive groups of isometries of Euclidean spaces in Subsection \ref{subsection:singorb:transitivern}.

This paper is organized as follows. In Section~\ref{section:structure}, we give a brief exposition of some aspects of the theories of symmetric spaces, real semisimple Lie algebras, and isometric actions. In Section~\ref{section:products}, we develop some structural theory concerning isometric actions on Riemannian products. In particular, we prove some decomposability results that lead to Theorem~\ref{th:main:decompcompactxhadamard}. In Section~\ref{section:HFSS} we study cohomogeneity-one actions on products of symmetric spaces under the assumption that they do not have singular orbits, where we conclude by proving Theorem~\ref{theorem:foliations} and computing some geometric properties of the orbits of $\s{H}_{E,X}$. In Section~\ref{section:singorb} we prove Theorem~\ref{th:singorbdecomp}. Theorem~\ref{mainth} then follows as an immediate consequence of Theorems~\ref{th:main:decompcompactxhadamard},~\ref{theorem:foliations}, and~\ref{th:singorbdecomp}. Finally, in Section~\ref{section:nonsimplyconnected}, we explain how to relate isometric actions on a connected Riemannian manifold to those on its universal cover, which lets us generalize Theorem~\ref{mainth} to non-simply connected symmetric spaces.

\begin{ackn*}
	The authors would like to thank J. C. D\'iaz-Ramos, M. Dom\'inguez-V\'azquez, and J. M. Lorenzo-Naveiro for helpful comments and discussions.
    We also thank A. J. di Scala and C. Gorodski for pointing us towards the results of \cite{Alekseevski}.
	The first author acknowledges support from the Deutsche Forschungsgemeinschaft (DFG, German Research Foundation) Project ID~427320536--SFB 1442, and Germany's Excellence Strategy EXC 2044/2 - 390685587, Mathematics Münster: Dynamics-Geometry-Structure, and PID2022-138988NB-I00 funded by MICIU/AEI/10.13039/501100011033 and by ERDF, EU.
	The third author was supported by JSPS KAKENHI Grant Numbers 23K22395 and 24K21193.  
    This work was partially supported by the MEXT Promotion of Distinctive Joint Research Center Program (JPMXP0723833165)  
    and the Osaka Metropolitan University Strategic Research Promotion Project (Development of International Research Hubs).
\end{ackn*}

\par
\endgroup

%% file: 2-Preliminaries.tex
In this section, we introduce some basic concepts, terminology, and notation related to isometric actions, Riemannian symmetric spaces, and real semisimple Lie algebras, which will be used throughout the article. We also mention some aspects of the theory of parabolic subgroups of real semisimple Lie groups. For a more detailed exposition, we refer to \cite{Helgason} for the theory of symmetric spaces, \cite{Knapp} and \cite{OnishchikVinberg} for semisimple Lie algebras, and \cite{AleBet,Michor} for isometric actions (see also \cite{BCO:book}).

\flushbottom

\subsection{Isometric actions and their orbits}\label{subsection:prelim:isometric}
If $\s{H}$ is a Lie group acting smoothly on a smooth manifold $M$, then each orbit $\s{H} \cdot p \cong \s{H}/\s{H}_p$ is a injectively immersed submanifold of $M$. The \textit{cohomogeneity} of an action $\s{H} \curvearrowright M$ is the lowest codimension of its orbits. Orbits of the lowest codimension (as well as their points) are called \textit{regular} and the other ones are said to be \textit{singular}. The action is said to be \textit{proper} if the map $\s{H} \times M \to M \times M$ given by $(h,p) \mapsto (h \cdot p, p)$ is proper.  Proper actions enjoy a variety of favorable properties: for instance, their orbits are properly embedded, their isotropy subgroups are compact, and the orbit space $M/\s{H}$ is Hausdorff. If $M$ is a (complete and connected) Riemannian manifold, an action $\s{H} \curvearrowright M$ is called \textit{isometric} if it preserves the metric of $M$. In this case (after possibly quotienting $\s{H}$ by its ineffectiveness kernel) we can assume that $\s{H}$ is a Lie subgroup of the isometry group $I(M)$. Then, the action of $\s{H}$ on $M$ is proper if and only if $\s{H}$ is a closed subgroup of $I(M)$. If all orbits of $\s{H} \subseteq I(M)$ are closed, then $\s{H}$ has the same orbits as its closure $\overline{\s{H}}$ in $I(M)$. For these reasons, it is rather common to restrict oneself to closed subgroups of $I(M)$ when studying isometric actions. 

Throughout the article, we will occasionally have to work with actions that may fail to be proper. We will make use of the following two facts. First, given a smooth action $\s{H} \curvearrowright M$ of a connected Lie group $\s{H}$, the orbits of $\s{H}$ are not merely immersed but in fact weakly embedded, also known as initial submanifolds (see \cite[Th.\ 6.4]{Michor}): any smooth map $f \colon N \to M$ that takes values in an orbit $\s{H} \cdot p$ is actually smooth as a map to $\s{H} \cdot p$. Second, any isometric action of cohomogeneity zero on a connected Riemannian manifold must be transitive. Indeed, such an action must have an open orbit. Being a Riemannian homogeneous space, that orbit is complete and thus must coincide with the ambient manifold.

If $\s{H} \curvearrowright M$ is a proper isometric action with both $\s{H}$ and $M$ connected, the orbits are equidistant and hence one can define distance in the orbit space $M/\s{H}$, turning it into a metric space. If the action has cohomogeneity one and $M$ is complete, the orbit space is a 1-dimensional manifold with boundary and thus has to be homeomorphic to one of the following: the real line $\R$, the circle $\mathbb{S}^1$, the half-open interval $[0,1)$, or the closed interval $[0,1]$. In this case, either the orbits of $\s{H}$ form a regular codimension-one foliation, or else there exist at most two singular orbits, and the regular orbits are tubes around any of the singular ones. There are various relations between the geometry and topology of $M$ and $M/\s{H}$. For example, if $M$ is compact, then so is the orbit space, and hence the latter is homeomorphic to either $\mathbb{S}^1$ or $[0,1]$. If $M$ is simply connected, the orbit space cannot be $\mathbb{S}^1$, and the boundary points of $M/\s{H}$ correspond precisely to the singular orbits. 
If $M$ is a Hadamard manifold, then $M/\s{H}$ is either $\R$ or $[0,1)$. See \cite[Sect.~6.3]{AleBet} and \cite[Sect.~2]{BB:crelle} for further details.

\subsection{Symmetric spaces and their type.}\label{subsection:symmetricsptype}
Let $M = \s{G}/\s{K}$ be a symmetric space, where $\s{G} = I^0(M)$ is the identity component of the isometry group of $M$ and $\s{K}$ is the isotropy group of some fixed base point $o \in M$. As usual, we will use gothic letters to denote the corresponding Lie algebras. If $s_o \in I(M)$ is the geodesic symmetry at $o$, then $\theta = \Ad(s_o)$ is an involutive automorphism of $\g{g}$. The $(+1)$-eigenspace of $\theta$ coincides with the Lie algebra $\g{k}$ of $\s{K}$, whereas the $(-1)$-eigenspace $\g{p}$ is naturally identified with the tangent space $T_oM$, and we have a reductive decomposition $\g{g} = \g{k} \oplus \g{p}$. Let $\mathcal{B}$ be the Killing form of $\g{g}$. A symmetric space is said to be of \emph{compact}, \emph{Euclidean}, or \emph{noncompact} type if the restriction of $\mathcal{B}$ to $\g{p} \times \g{p}$ is negative definite, zero, or positive definite, respectively. If $M$ is of compact type, then $M$ is compact and has nonnegative curvature; if $M$ is of noncompact type, then $M$ is a Hadamard manifold, i.e., it is of nonpositive curvature and diffeomorphic to a Euclidean space; lastly, if $M$ is of Euclidean type, then it is flat and thus its universal cover is a Euclidean space.
 
Assume that $M$ is simply connected and let $M = M_0 \times M_1 \times \cdots \times M_k$ be its de Rham decomposition, where $M_0$ is the Euclidean factor. We can write $M_i = \s{G}_i/\s{K}_i$ for each $i = 0, \ldots, k$ just like we did for $M$, which leads to the splittings $\s{G} = \prod_i \s{G}_i$,  $\s{K} = \prod_i \s{K}_i$, and so on (cf.\ Lemma~\ref{lemma:solonenko}). Moreover, it follows from  Schur's lemma that for $i \neq 0$, $M_i$ must be of either compact or noncompact type. By grouping together the $M_i$'s of the same type, we get a Riemannian product decomposition $M = M_+ \times M_0 \times M_-$, where $M_+$ and $M_-$ are symmetric spaces of compact and noncompact types, respectively, and $M_0 \cong \R^n$ for some $n \in \mathbb{N}$.

Throughout Sections~\ref{section:HFSS} and~\ref{section:singorb}, it will come in very handy to fix an $\Ad_{\s{K}}$-invariant inner product on $\g{g}$ that makes $\g{k}$ and $\g{p}$ orthogonal. We do it as follows: first of all, proclaim $\g{g}_i$ and $\g{g}_j$ to be orthogonal for $i\neq j$. Just as we did before, fix an isometry $M_0\cong\R^n$. The (normal abelian) subgroup of $\s{G}_0$ consisting of translations can be naturally identified with $\R^n$ itself. We then have $\s{G}_0 \cong \s{SO}_n \ltimes \R^n$ and $\g{g}_0 \cong \g{so}_n \loplus \g{r}^n$, where $\g{r}^n = \g{p}_0$ is an abelian ideal. We write $\bilin{-}{-}_{M_0}$ for the scalar product on $\g{r}^n$ induced via $\g{r}^n \cong T_o \R^n \simeq \R^n$. For $X,Y\in \g{so}_n$ and $v,w\in\g{r}^n$, the formula $\bilin{X+v}{Y+w} = -\tr(XY) + \bilin{v}{w}_{M_0}$ defines an $\Ad_{\s{SO}_n}$-invariant inner product on $\g{g}_0$. Next, take any $1 \le i \le k$. The Killing form $\mathcal{B}_{\g{g}_i}$ is negative definite on $\g{k}_i$, and we also have $\g{k}_i \perp \g{p}_i$. If $M_i$ is of compact type, then $\mathcal{B}_{\g{g}_i}$ is also negative definite on $\g{p}_i$, and we take $\bilin{-}{-}$ to be $-\mathcal{B}_{\g{g}_i}$ on $\g{g}_i$. If $M_i$ is of noncompact type, then $\mathcal{B}_{\g{g}_i}$ is positive definite on $\g{p}_i$, and we take $\bilin{X}{Y} = -\mathcal{B}_{\g{g}_i} (X, \theta Y)$. This inner product satisfies $\langle \ad(X)Y, Z \rangle = -\langle Y, \ad(\theta X), Z \rangle$ for any $X, Y, Z \in \g{g}_i$ and, in particular, is $\ad_{\g{k}_i}$-invariant. Putting these together, we get an $\Ad_{\s{K}}$-invariant inner product $\bilin{-}{-}$ on the whole $\g{g}$.

In order to simplify some of the computations, unless mentioned otherwise, we will assume throughout the rest of the article that the metric on $M$ comes from the restriction of $\bilin{-}{-}$ to $\g{p}$. Although this might seem restrictive a priori, this assumptions turns out harmless when classifying isometric actions on $M$. Indeed, it is a consequence of Schur's lemma that any $\s{G}$-invariant Riemannian metric $g$ on $M$ can be written at the base point $o$ as
\[
g_o = \lambda_0 \restr{\bilin{-}{-}}{\g{p}_0 \times \g{p}_0} + \lambda_1 \restr{\bilin{-}{-}}{\g{p}_1 \times \g{p}_1} + \cdots + \lambda_k \restr{\bilin{-}{-}}{\g{p}_k \times \g{p}_k}
\]
for some constants $\lambda_i > 0$. Consequently, for any two $\s{G}$-invariant metrics $g$ and $g'$ on $M$, we have $I^0(M,g) = I^0(M,g')$ (see Lemma~\ref{lemma:solonenko}). This implies that an action of a connected Lie group on $M$ is isometric with respect to $g$ if and only if it is such with respect to $g'$ (see also \cite[Rem.\ 2.5]{DDO:adv}).
	
\begin{remark}
	Note that the subalgebras $\g{h}_X$, $\g{h}_i$, and $\g{h}_{E,X}$ in Theorems~\ref{mainth}~and~\ref{theorem:foliations} have been previously described in terms of the induced inner product on $\g{r}^n \oplus (\g{a} \loplus \g{n})$ under the identification $\g{r}^n \oplus (\g{a} \loplus \g{n}) \cong T_o(M_0 \times M_-)$. If we define these subalgebras in terms of our fixed inner product $\bilin{-}{-}$, the result might be different.
\end{remark}

\subsection{Structure of symmetric spaces of noncompact type}\label{subsection:structurenoncompact}
If $M_-=\s{G}_-/\s{K}_-$ is a symmetric space of noncompact type, then $\s{G}_-$ is a noncompact real semisimple Lie group and $\s{K}_-$ is a maximal compact subgroup of $\s{G}_-$. In this case, $\g{g}_- = \g{k}_- \oplus \g{p}_-$ is a Cartan decomposition and $\theta$ is the corresponding Cartan involution. Pick a maximal abelian subspace $\g{a}$ in $\g{p}_-$. For each $\lambda \in \g{a}^*$, define $\g{g}_\lambda = \{ X \in \g{g}_- \mid [H,X] = \lambda(H) X \; \text{for all} \; H \in \g{a} \}$. If $\lambda \ne 0$ and $\g{g}_{\lambda} \ne \{0\}$, $\lambda$ is called a (restricted) root of $\g{g}_-$ with respect to $\g{a}$, and $\g{g}_\lambda$ is the root space corresponding to $\lambda$. The set of roots of $\g{g}_-$ with respect to $\g{a}$ will be denoted by $\Sigma$. The subspace\footnote{In the presence of a Euclidean factor $M_0$, one has to be careful not to confuse this subspace with the isometry Lie algebra of $M_0$.} $\g{g}_0$ coincides with $Z_{\g{k}}(\g{a}) \oplus \g{a}$, where $Z_{\g{k}}(\g{a})$ is the centralizer of $\g{a}$ in $\g{k}$. The direct sum decomposition
\[
\g{g}_- = \g{g}_0 \oplus \left(\bigoplus_{\lambda \in \Sigma} \g{g}_{\lambda} \right)
\]
is called the restricted root space decomposition of $\g{g}_-$ with respect to $\g{a}$. Moreover, if $\lambda, \, \mu \in \Sigma \cup \{0\}$, we have $\theta \g{g}_\lambda = \g{g}_{-\lambda}$ and $[\g{g}_{\lambda},\g{g}_{\mu}] \subseteq \g{g}_{\lambda + \mu}$. The rank $r$ of $M_-$ coincides with $\dim(\g{a})$. The restriction of the Killing form $\mathcal{B}_{\g{g}_-}$ to $\g{a}$ is positive definite, so it gives rise to an isomorphism $\g{a} \isoto \g{a}^*$ and thus induces an inner product on $\g{a}^*$. The set $\Sigma$ constitutes a (possibly nonreduced) root system in $\g{a}^*$. Choose a system $\Sigma^+$ of positive roots in $\Sigma$ and let $\Lambda = \{ \alpha_1, \ldots, \alpha_r \} \subseteq \Sigma^+$ be the corresponding set of simple roots. The sum $\g{n} = \bigoplus_{\lambda \in \Sigma^+} \g{g}_\lambda$ is a nilpotent subalgebra of $\g{g}_-$. What is more, $\g{g}_-$ splits into a vector space direct sum $\g{g}_- = \g{k}_- \oplus \g{a} \oplus \g{n}$, known as an Iwasawa decomposition. If we write $\s{A}$ and $\s{N}$ for the connected Lie subgroups of $\s{G}_-$ with Lie algebras $\g{a}$ and $\g{n}$, respectively, then we get a global Iwasawa decomposition $\s{G}_- = \s{K}_-\s{A}\s{N}$ (meaning, the multiplication map yields a diffeomorphism $\s{K}_- \times \s{A} \times \s{N} \isoto \s{G}_-$). The subgroups $\s{A}$ and $\s{N}$ are closed and simply connected. Moreover, $\s{A}$ normalizes $\s{N}$, so the product $\s{A}\s{N} = \s{A} \ltimes \s{N}$ is a solvable Lie subgroup of $\s{G}_-$. The global Iwasawa decomposition implies that $\s{A}\ltimes\s{N}$ acts on $M_-$ simply transitively. The pullback of the metric of $M_-$ under the diffeomorphism $\s{A}\ltimes\s{N} \isoto M_-$ is left-invariant, which allows us to think of $M_-$ as a simply connected solvable Lie group endowed with a left-invariant metric. This is called the solvable model of $M_-$.

\subsection{Parabolic subalgebras of real semisimple Lie algebras}\label{subsection:prelim:parabolic}
Let $\g{g}_-$ be a noncompact real semisimple Lie algebra. A subalgebra $\g{q} \subseteq \g{g}_-$ is called \emph{parabolic} if it contains a subalgebra conjugate to $Z_{\g{k}}(\g{a}) \oplus \g{a} \oplus \g{n}$. Geometrically speaking, if $\g{g}_-$ is the isometry Lie algebra of a symmetric space of noncompact type $M_-$, then proper parabolic subalgebras of $\g{g}_-$ are the isotropy Lie algebras of points at the ideal boundary $M_-(\infty)$ under the action of $\s{G}_-$.
    
The conjugacy classes of parabolic subalgebras of $\g{g}_-$ can be parameterized by the subsets of the set of simple roots: given $\Phi \subseteq \Lambda$, we can consider the root subsystem of $\Sigma$ generated by $\Phi$, namely $\Sigma_\Phi = \Sigma \cap \operatorname{span}\{\Phi\}$. We choose $\Sigma_\Phi^+ = \Sigma^+ \cap \Sigma_\Phi$ as a system of positive roots for $\Sigma_\Phi$ and define the following subalgebras of $\g{g}_-$:
	\[
	\g{l}_{\Phi} = \g{g}_0 \oplus \Bigg( \bigoplus_{\lambda \in \Sigma_\Phi} \g{g}_\lambda \Bigg), \quad \g{n}_\Phi = \hspace{0.6em} \smashoperator{\bigoplus_{\lambda \in \Sigma^+ \setminus \Sigma_\Phi^+}} \; \g{g}_\lambda,    
	\]
the first of which is $\theta$-invariant and thus reductive, whereas the second one is nilpotent. It follows from the properties of root spaces that $\g{l}_{\Phi}$ normalizes $\g{n}_{\Phi}$ and therefore
	\begin{equation}\label{eq:Chevalley}
		\g{q}_\Phi = \g{l}_{\Phi} \loplus \g{n}_\Phi
	\end{equation}
is a subalgebra of $\g{g}_-$ containing $Z_{\g{k}}(\g{a}) \oplus \g{a} \oplus \g{n}$. It is called the parabolic subalgebra of $\g{g}_-$ associated with the subset $\Phi$, and decomposition~\eqref{eq:Chevalley} is known as its Chevalley decomposition. Any parabolic subalgebra of $\g{g}_-$ is conjugate to $\g{q}_\Phi$ for some choice of $\Phi \subseteq \Lambda$.
	
We define $\g{a}_\Phi = \bigcap_{\alpha \in \Phi} \ker(\alpha)$ and $\g{a}^\Phi = \g{a} \ominus \g{a}_\Phi$. The abelian subalgebra $\g{a}_\Phi$ normalizes $\g{n}_\Phi$ and its centralizer in $\g{g}_-$ is $\g{l}_\Phi$. The complement $\g{m}_\Phi = \g{l}_\Phi \ominus \g{a}_\Phi$ is also a $\theta$-invariant subalgebra, and we have a Lie algebra direct sum decomposition $\g{l}_\Phi = \g{m}_\Phi \oplus \g{a}_\Phi$. The induced decomposition 
	\[
	\g{q}_\Phi = (\g{m}_\Phi \oplus \g{a}_\Phi) \loplus \g{n}_\Phi
	\]
is called the Langlands decomposition of $\g{q}_\Phi$.
	
The subalgebra $\g{n}_\Phi$ admits a natural grading. Namely, consider $H^\Phi = \sum_{\alpha_i \in \Lambda \setminus \Phi} H^{i}$, where $\{H^i\}_{i=1}^r$ is the basis for $\g{a}$ dual to $\Lambda$, and write $m_\Phi = \delta(H^\Phi)$, where $\delta$ is the highest root in $\Sigma^+$. Note that $\lambda(H^\Phi) \in \{0, 1, \ldots, m_\Phi\}$ for all $\lambda \in \Sigma^+$, and $\lambda(H^\Phi) = 0$ if and only if $\lambda \in \Sigma_\Phi^+$. Thus, we have a splitting
    \[
    \g{n}_\Phi = \bigoplus_{\nu = 1}^{m_\Phi} \g{n}_\Phi^\nu, \quad \text{where} \quad \g{n}_\Phi^\nu = \hspace{0.4em} \smashoperator{\bigoplus_{\lambda(H^\Phi) = \nu}} \; \g{g}_\lambda.
    \]
Given $\lambda \in \Sigma$, we define $\g{k}_\lambda = \pr_{\g{k}_-}(\g{g}_{\lambda}) = \g{k}_- \cap (\g{g}_\lambda \oplus \g{g}_{-\lambda})$ and $ \g{p}_\lambda = \pr_{\g{p}_-}(\g{g}_{\lambda}) = {\g{p}_- \cap (\g{g}_\lambda \oplus \g{g}_{-\lambda})}$, where $\pr_{\g{k}_-}$ and $\pr_{\g{p}_-}$ are the orthogonal projections to $\g{k}_-$ and $\g{p}_-$, respectively. Then,
	\[
	\g{k}_\Phi = \g{q}_\Phi \cap \g{k} = \g{m}_\Phi \cap \g{k} = \g{k}_0 \oplus \biggl( \bigoplus_{\lambda \in \Sigma^+_\Phi} \g{k}_\lambda \biggr) \quad \text{and} \quad \g{b}_\Phi = \g{m}_\Phi \cap \g{p} = \g{a}^\Phi \oplus \biggl(\bigoplus_{\lambda \in \Sigma^+_\Phi} \g{p}_\lambda \biggr)
	\]
are a maximal compactly embedded subalgebra of $\g{q}_\Phi$ and a Lie triple system in $\g{p}_-$, respectively. The latter means that $\g{b}_\Phi$ corresponds to the tangent space at $o_{M_-}$ of a connected complete totally geodesic submanifold $B_\Phi$ of $M_-$, called a \emph{boundary component}. The submanifold $B_\Phi$ is itself a symmetric space of noncompact type whose rank is equal to the cardinality of $\Phi$. We write\footnote{The definition of $\g{g}_\Phi$ given here is different from the one used in~\cite{BDT:jdg,BT:crelle}, where $\g{g}_\Phi$ was defined to be the derived subalgebra of $\g{m}_\Phi$---which is in general larger and thus fails to be the isometry Lie algebra of $B_\Phi$ (see \cite[Rem.\ 2.4]{Solonenko}).} $\g{g}_\Phi$ for the subalgebra of $\g{g}_-$ generated by $\g{b}_\Phi$. This subalgebra is $\theta$-stable and semisimple and it has a natural Cartan decomposition $\g{g}_\Phi = [\g{b}_\Phi,\g{b}_\Phi] \oplus \g{b}_\Phi$. Note that $\g{a}^\Phi$ is a maximal abelian subspace of $\g{b}_\Phi$. The corresponding root system of $\g{g}_\Phi$ is given by $\Sigma_\Phi$ and it has $\Phi$ as a system of simple roots.
	
Write $\s{A}_\Phi$, $\s{N}_\Phi$, and $\s{G}_\Phi$ for the connected (closed) Lie subgroups of $\s{G}_-$ with Lie algebras $\g{a}_\Phi$, $\g{n}_\Phi$, and $\g{g}_\Phi$, respectively. The subgroup $\s{L}_\Phi = Z_{\s{G}_-}(\g{a}_\Phi)$ normalizes $\s{N}_\Phi$ and has $\g{l}_\Phi$ as its Lie algebra, and the semidirect product $\s{Q}_\Phi = \s{L}_\Phi \ltimes \s{N}_\Phi$ coincides with the normalizer $N_{\s{G}_-}(\g{q}_\Phi)$. We call $\s{Q}_\Phi$ the parabolic subgroup of $\s{G}_-$ associated with $\Phi$. We further define $\s{K}_\Phi = \s{L}_\Phi \cap \s{K}_-$ and $\s{M}_\Phi = \s{K}_\Phi \s{G}_\Phi$. These are closed subgroups of $\s{G}_-$ with Lie algebras $\g{k}_\Phi$ and $\g{m}_\Phi$, respectively. Note that $\s{M}_\Phi$ and $\s{L}_\Phi$ are (possibly disconnected) reductive Lie groups and $\s{K}_\Phi$ is a maximal compact subgroup in either of them (as well as in $\s{Q}_\Phi$). The group $\s{L}_\Phi$ splits as a direct product $\s{L}_\Phi = \s{M}_\Phi \times \s{A}_\Phi$. The decompositions 
	\[
	\s{Q}_\Phi = \s{L}_\Phi \ltimes \s{N}_\Phi \quad \text{and} \quad \s{Q}_\Phi = (\s{M}_\Phi \times \s{A}_\Phi) \ltimes \s{N}_\Phi
	\]
are called the Chevalley and Langlands decompositions of $\s{Q}_\Phi$, respectively.
	
The orbit of $\s{G}_\Phi$ (which coincides with that of $\s{M}_\Phi)$ through $o$ in $M_-$ is precisely the boundary component $B_\Phi$. Moreover, $(\s{G}_\Phi, \s{G}_\Phi \cap \s{K}_-) = (\s{G}_\Phi, \s{G}_\Phi \cap \s{K}_\Phi)$ is an almost effective Riemannian symmetric pair representing $B_\Phi$, meaning, the subgroup of elements of $\s{G}_\Phi$ that act on $B_\Phi$ trivially is finite. In particular, we have a finite covering $\s{G}_\Phi \rightarrow I^0(B_\Phi)$ and $\g{g}_\Phi$ is the isometry Lie algebra of $B_\Phi$. The Langlands decomposition of $\s{Q}_\Phi$ induces a diffeomorphism
	\[
	B_\Phi \times \s{A}_\Phi \times \s{N}_\Phi \isoto M_-,
	\]
known as the horospherical decomposition of $M_-$. This allows us to identify $T_{o_{M_-}}(M_-) \cong \g{b}_\Phi \oplus \g{a}_\Phi \oplus \g{n}_\Phi$.
	
	\begin{remark}
		Throughout the paper (and especially in Subsection~\ref{subsection:singorb:canonical}), we will be particularly interested in maximal proper parabolic subalgebras of $\g{g}_-$. Note that if $\Psi \subseteq \Phi \subseteq \Lambda$, then $\g{q}_\Psi \subseteq \g{q}_\Phi$.
		Thus, maximal proper parabolic subalgebras of $\g{g}_-$ correspond to subsets of simple roots of the form $\Lambda \setminus \{\alpha_j\}$ for some simple root $\alpha_j \in \Lambda$.
		For the sake of brevity, we will often drop the index $\Phi$ and substitute it by $j$ when dealing with maximal proper parabolic subalgebras (e.g., the parabolic subalgebra associated with $\Phi = \Lambda \setminus \{\alpha_j\}$ will be denoted by $\g{q}_j$, and so on).
	\end{remark}

%% file: 3-Isometric_actions_on_products.tex
Given an isometric action on a Riemannian product $M_1 \times \cdots \times M_n$, it makes sense to ask how much information about the action can be extracted from its behavior on each of the individual factors $M_i$. In this section, we introduce a notion of decomposability for isometric actions following the approach used by Kollross for products of symmetric spaces of compact type \cite{Kollross:reducible}. It should be noted that, when trying to carry out this kind of study for general Riemannian manifolds, one encounters some technical difficulties. Namely, a proper isometric action on a Riemannian product may not induce actions on its factors and, even if it does, these induced actions may fail to be proper. We deal with these obstacles and prove some decomposability results for cohomogeneity-one actions.

Given a collection of Lie groups $\s{H}_1, \ldots, \s{H}_k$ acting isometrically on (connected) Riemannian manifolds $M_1, \ldots, M_k$, one can define the \emph{product action} of $\s{H}_1 \times \cdots \times \s{H}_k$ on $M_1 \times \cdots \times M_k$ by the formula
\[
(h_1, \ldots,h_k) \cdot (p_1, \ldots, p_k) = (h_1 \cdot p_1, \ldots, h_k \cdot p_k),
\]
for $h_i \in \s{H}_i$ and  $p_i \in M_i$. Note that the orbits of this action are precisely the products of the orbits of the actions $\s{H}_i \curvearrowright M_i$. We will say that an isometric action of a Lie group $\s{H}$ on a product of Riemannian manifolds $M_1 \times \ldots \times M_k$ \emph{decomposes} if there exist isometric Lie group actions $\s{H}_i \curvearrowright M_i, \, i = 1, \ldots, k$, such that the action of $\s{H}$ has the same orbits as the product action of $\s{H}_1 \times \cdots \times \s{H}_k$ on $M_1 \times \cdots \times M_k$. (If $k = 1$, we say that every action decomposes.) It should be noted that even when the action of $\s{H}$ is proper, the actions $\s{H}_i \curvearrowright M_i$ do not have to be. However, in that case, since the orbits of $\s{H}$ are properly embedded, the same is true for the orbits of each of the actions $\s{H_i} \curvearrowright M_i$. As we mentioned in Subsection \ref{subsection:prelim:isometric}, for each $i$, the closure $\overline{\s{H}}_i$ in $I(M_i)$ must have then same orbits as $\s{H}_i$. Consequently, we can simply replace each $\s{H}_i$ with $\overline{\s{H}}_i$, whose action on $M_i$ is now proper.

\begin{remark}
	This notion of decomposability naturally depends on the expression of $M$ as a Riemannian product. In order to avoid this, in \cite{Kollross:reducible} an isometric action of a Lie group $\s{H}$ on a Riemannian manifold $M$ was defined to be \emph{decomposable} if there exists a decomposition of $M$ as a nontrivial Riemannian product $M = M_1 \times \cdots \times M_k$ such that the action of $\s{H}$ is orbit-equivalent to an action on $M_1 \times \cdots \times M_k$ that decomposes in the above sense.
    Although our notion is slightly more restrictive, it is tailored to the study of actions in relation to a fixed decomposition of $M$, such as the decomposition $M = M_+ \times M_0 \times M_-$ of a simply connected symmetric space.
\end{remark}

Let $M = M_1 \times \cdots \times M_k$ be a product of connected Riemannian manifolds and let $\s{H}$ be a Lie subgroup of $I(M_1) \times \cdots \times I(M_k) \subseteq I(M)$. Then, the action of $\s{H}$ on $M$ naturally induces an isometric action on each of the factors by taking
\[
(h_1, \ldots, h_k) \cdot p_i = h_i \cdot p_i
\]
for $(h_1, \ldots, h_k) \in \s{H}$ and $p_i \in M_i$. This is known as the \emph{projection action} of $\s{H}$ on $M_i$. The orbits of this action are precisely the projections of the orbits of $\s{H} \curvearrowright M$ in $M_i$. If we write $\pr_{I(M_i)} \colon I(M_1) \times \cdots \times I(M_k) \rightarrow I(M_i)$ for the projection onto $I(M_i)$, then the effectivization of the projection action of $\s{H}$ on $M_i$ is the action of $\pr_{I(M_i)}(\s{H})$ on $M_i$, so both actions have the same orbits. If the action of $\s{H}$ decomposes as $\prod_{i=1}^k \s{H}_i \curvearrowright \prod_{i=1}^k M_i$, the action $\s{H}_i \curvearrowright M_i$ must have the same orbits as the projection action of $\s{H}$ on $M_i$ (or, equivalently, as $\pr_{I(M_i)}(\s{H})$). Consequently, the action of $\s{H}$ decomposes if and only if it has the same orbits as the product of its projection actions on the factors $M_1, \ldots, M_k$.

\begin{remark}
	In general, the isometry group of a Riemannian product $M_1 \times \cdots \times M_k$ might be larger than $I(M_1)\times\dots\times I(M_n)$.
	For instance, for any positive integers $p$ and $q$, one has $I(\R^p) \times I(\R^q) = (\s{O}_p \times \s{O}_q) \ltimes \R^{p+q} \subsetneq \s{O}_{p+q} \ltimes \R^{p+q} = I(\R^p \times \R^q)$. Thus, it is not always possible to define the projection actions of an arbitrary Lie subgroup $\s{H}\subseteq I(M_1 \times \cdots \times M_k)$.
\end{remark}

If $M_1 \times \cdots \times M_k$ is a connected Riemannian product and $\s{H} \subseteq I(M_1) \times \cdots \times I(M_k)$ is a Lie subgroup, the orbit of $\s{H}$ through a point $(p_1, \ldots, p_k)$ is contained in the product of the orbits of the projection actions of $\s{H}$ on each $M_i$ through $p_i$:
$\s{H} \cdot (p_1, \ldots, p_k) \subseteq \s{H} \cdot p_1 \times \cdots \times \s{H} \cdot p_k$. Thus, if we denote the cohomogeneity by $\operatorname{cohom}(\s{H} \curvearrowright M)$, one has
\begin{equation}\label{codimension comparison}
	\operatorname{cohom}(\s{H} \curvearrowright M_1 \times \dots \times M_k) \ge \sum_{i=1}^k \operatorname{cohom}(\s{H} \curvearrowright M_i).
\end{equation}
As we observed above, if the action of $\s{H}$ decomposes, then it has the same orbits as the product of its projection actions. In that case, equality holds in \eqref{codimension comparison}. One could ask if the converse implication is also true: does equality in \eqref{codimension comparison} imply that the action decomposes? In an attempt to prove this, one might run into the following complication: even when the action of $\s{H}$ is proper, its projection actions do not have to be. Fortunately, this does not prove to be an obstacle---at least when the action has cohomogeneity one:

\begin{lemma}\label{lemma:SplittingC1}
	Let $M = M_1 \times \cdots \times M_k$ be a complete connected Riemannian product and $\s{H} \subseteq I(M_1) \times \dots \times I(M_k)$ a closed connected subgroup acting on $M$ with cohomogeneity one. Then, the action decomposes if and only if $\sum_{i=1}^{k} {\operatorname{cohom}(\s{H}\curvearrowright M_i)} = 1$.
\end{lemma}

\begin{proof}
	We only need to prove that equality in \eqref{codimension comparison} implies that the action of $\s{H}$ decomposes. The equality means that there is precisely one projection action that is of cohomogeneity one and the rest are transitive. Therefore, the product $\s{H} \times \cdots \times \s{H} \curvearrowright M_1 \times \cdots \times M_k$ of the projection actions is also of cohomogeneity one. Let $p = (p_1, \ldots, p_k)$ be a regular point of $M$ with respect to the action of $\s{H}$. It must then also be regular with respect to the product of the projection actions. In other words, both orbits $\s{H} \cdot p$ and $\s{H} \cdot p_1 \times \cdots \times \s{H} \cdot p_k$ are hypersurfaces. Since the former orbit is contained in the latter and the latter is  weakly embedded, the inclusion $\s{H} \cdot p \hookrightarrow \s{H} \cdot p_1 \times \cdots \times \s{H} \cdot p_k$ is smooth. By dimension reasons, this inclusion is open. This means that $\s{H}$ acts on the latter orbit isometrically and with cohomogeneity zero, and hence transitively. We have shown that every regular orbit of $\s{H}$ is also an orbit of $\s{H} \times \cdots \times \s{H}$. Since the action of $\s{H}$ on $M$ is proper and has cohomogeneity one, it can have at most two singular orbits. Each of those is contained in an orbit of $\s{H} \times \cdots \times \s{H}$ and the latter is connected. This implies that every singular orbit of $\s{H}$ is an orbit of $\s{H} \times \cdots \times \s{H}$ as well. Overall, the two actions share the same orbits and thus the action of $\s{H}$ decomposes.
\end{proof}

Next, we address the question of existence of projection actions. Let $M$ be a complete simply connected Riemannian manifold, and let $M = M_0 \times M_1 \times \cdots \times M_k$ be its de Rham decomposition. The following result---formulated in a slightly more general setting---asserts that the isometry group of $M$ behaves well with respect to the factors:
\begin{lemma}[{\cite[Prop.\ 2.1.60]{solonenko_thesis}}]\label{lemma:solonenko}
	Let $M = M_0 \times M_1 \times \cdots \times M_k$ be a Riemannian product, where $M_0$ is connected and flat and $M_i$'s are (non-flat) connected irreducible Riemannian manifolds for $1 \le i \le k$. Then, the group $I(M)$ decomposes as a semidirect product
	\[
	I(M) = [I(M_0) \times I(M_1) \times \cdots \times I(M_k)] \rtimes S_{M_1 \times \cdots \times M_k},
	\]
    where $S_{M_1 \times \cdots \times M_k}$ is the subgroup of the symmetric group $S_k$ consisting of elements permuting the mutually isometric factors $M_i, \, 1 \le i \le k$.
\end{lemma}

When applied to a complete simply connected Riemannian manifold, Lemma~\ref{lemma:solonenko} essentially reads as the uniqueness property of the de Rham decomposition. It follows immediately that, for $M = M_0 \times M_1 \times \cdots \times M_k$ as above, we have $I^0(M) = I^0(M_0) \times I^0(M_1) \times \cdots \times I^0(M_k)$. Thus, if $\s{H}$ is a connected Lie group acting on $M$ by isometries, we can always assume $\s{H} \subseteq I^0(M_0) \times I^0(M_1) \times \cdots \times I^0(M_k)$ without loss of generality, and thus the projection actions of $\s{H}$ on the factors are well-defined. Here is another application of Lemma~\ref{lemma:solonenko} repackaged in a slightly different way:
\begin{lemma}\label{lemma:projections}
	Let $M$ and $N$ be simply connected complete Riemannian manifolds and suppose that $N$ has no Euclidean de Rham factor.
	If $\s{H}$ is a connected Lie subgroup of $I(M \times N)$, then $\s{H} \subseteq I^0(M) \times I^0(N)$. In particular, the projection actions of $\s{H}$ on $M$ and $N$ are well-defined.
\end{lemma}

\begin{proof}
	Let $M = M_0 \times M_1 \times \cdots \times M_k$ and $N = N_1 \times \cdots \times N_l$ be the de Rham decompositions of $M$ and $N$, respectively. Then, the de Rham decomposition of $M \times N$ can be written as
	\[
	M \times N = M_0 \times M_1 \times \cdots \times M_k \times N_1 \times \cdots \times N_l.
	\]
	We then have
    \[
    I^0(M \times N) = \prod_{i=0}^k I^0(M_i) \times \prod_{j=1}^l I^0(N_j) = I^0(M) \times I^0(N).
    \]
    In particular, any connected Lie subgroup of $I(M \times N)$ is contained in $I^0(M) \times I^0(N)$.
\end{proof}

An iterative application of Lemma \ref{lemma:projections} implies $I^0(M_1 \times \cdots \times M_k) = I^0(M_1) \times \cdots \times I^0(M_k)$ whenever $M_i$'s are simply connected and complete and at most one of them has a nontrivial Euclidean de Rham factor, so the projection actions  are well-defined for any connected Lie subgroup of $I(M_1 \times \cdots \times M_k)$ in that case as well.

We now turn our attention to Theorem \ref{th:main:decompcompactxhadamard} and consider the special situation when $M$ is a compact simply connected Riemannian manifold and $N$ is a Hadamard manifold.
In this case, $M$ does not have a Euclidean factor, so any isometric action $\s{H} \curvearrowright M \times N$ by a connected Lie group has well-defined projection actions.
The next result guarantees that only transitive actions on $M \times N$ can project to transitive actions on both $M$ and $N$.
\begin{lemma}\label{proposition:reductioncompactxhadamard}
	Let $M \times N$ be a simply connected Riemannian product, where $M$ is compact and $N$ Hadamard, and let $\s{H}$ be a connected Lie group acting on $M \times N$ by isometries. If both projection actions of $\s{H}$ on $M$ and $N$ are transitive, then so is the action of $\s{H}$ on $M \times N$.
\end{lemma}

\begin{proof}
	Although $\s{H}$ does not have to be compact, a result of Montgomery \cite[Th.\ A]{Montgomery} ensures that there exists a compact subgroup $\s{K}$ of $\s{H}$ whose action on $M$ is still transitive. By Cartan's fixed point theorem, the projection action of $\s{K}$ on $N$ must have a fixed point, say $o_N$. Pick any point $o_M \in M$ and consider an arbitrary point $(p,q) \in M \times N$. Since the projection action of $\s{H}$ on $N$ is transitive, there exists an element $h \in \s{H}$ such that $h \cdot q = o_N$. Let $k\in \s{K}$ be such that $k \cdot (h \cdot p) = o_M$. We have
	\[
	(kh) \cdot (p,q) = (k \cdot (h \cdot p), k \cdot o_N) = (o_M, o_N).
	\]
	Thus, $\s{H}$ acts on $M \times N$ transitively.
\end{proof}

By combining the results of this section, we can now prove Theorem~\ref{th:main:decompcompactxhadamard}:

\begin{proof}[Proof of Theorem~\ref{th:main:decompcompactxhadamard}]
	Recall that we are given a product $M \times N$, where $M$ is simply connected and compact and $N$ is Hadamard. Suppose that $\s{H}$ is a connected Lie group acting on $M \times N$ properly, isometrically, and with cohomogeneity one. According to Lemma~\ref{lemma:projections}, we may assume that $\s{H}$ is a connected closed subgroup of $I^0(M) \times I^0(N)$.	Now, according to Lemma~\ref{proposition:reductioncompactxhadamard}, only one of the projection actions of $\s{H}$ on $M$ and $N$ can be transitive, and the other one has to be of cohomogeneity one. Lemma~\ref{lemma:SplittingC1} then implies that the action of $\s{H}$ decomposes, which completes the proof.
\end{proof}

As discussed in the introduction, this result brings us one step closer to Theorem~\ref{mainth}. Indeed, let $M = M_+ \times M_0 \times M_-$ be a simply connected symmetric space decomposed into three types as usual. Let $\s{H}$ be a closed connected subgroup of $I(M)$ acting on $M$ with cohomogeneity one. Since $M_+$ is compact and $N = M_0 \times M_-$ is Hadamard, Theorem~\ref{th:main:decompcompactxhadamard} ensures that the action of $\s{H}$ on $M$ decomposes with respect to the splitting $M = M_+ \times N$. Suppose now that $\s{L}$ is a connected Lie group acting on $M_0 \times M_-$ properly, isometrically, and with cohomogeneity one. Then, by Lemma~\ref{lemma:projections}, its projection actions on $M_0$ and $M_-$ are well-defined. If any of these projection actions is of cohomogeneity one, then the action of $\s{L}$ decomposes by virtue of Lemma \ref{lemma:SplittingC1}. Otherwise, $\s{L}$ has to act transitively on both $M_0$ and $M_-$, which leads us to:
\begin{corollary}\label{corollary:decompC1SymmetricSpaces}
	Let $M = M_+ \times M_0 \times M_-$ be a simply connected symmetric space and let $\s{H}$ be a closed connected subgroup of $I(M)$ acting on $M$ with cohomogeneity one. Then, either the action of $\s{H}$ decomposes, or it is orbit-equivalent to the action of $I(M_+) \times \s{H}_{\Delta}$, where $\s{H}_{\Delta} \subseteq I(M_0 \times M_-)$ is a closed connected subgroup  that acts on $M_0 \times M_-$ with cohomogeneity one and whose projection actions on $M_0$ and $M_-$ are transitive.
\end{corollary}
In what follows, we will pay special attention to cohomogeneity-one actions on $M_0 \times M_-$ that do not decompose and thus induce transitive actions in each of the two factors.

%% file: 4-C1_homogeneous_foliations.tex
The aim of this section is to prove Theorem~\ref{theorem:foliations}, which describes all possible cohomogeneity-one actions on symmetric spaces without singular orbits.
We begin with two lemmas that will prove useful later on. The first of them follows directly from our discussion in Subsection~\ref{subsection:prelim:isometric} about the orbit spaces of cohomogeneity one actions, and it guarantees the non-existence of codimension-one homogeneous foliations on symmetric spaces of compact type:

\begin{lemma}\label{lemma:foliationscompact}
	Let $\s{H}$ be a connected Lie group acting properly, isometrically, and with cohomogeneity one on a simply connected compact Riemannian manifold $M$. Then, the action of $\s{H}$ has two singular orbits. 
\end{lemma}

The next lemma will be of help when dealing with actions on $M_0 \cong \R^n$ without singular orbits. We formulate it in a more general context of actions preserving a fixed totally geodesic submanifold.

\begin{lemma}\label{lem:preserving_t.g.}
    Let $M = \s{G}/\s{K}$ be a symmetric space, and let $S \subseteq M$ be a connected properly embedded totally geodesic submanifold passing through the base point $o$. Let $\s{H} \subseteq \s{G}$ be a Lie subgroup preserving $S$. Then, the projection $\pr_{\g{k}}(\g{h})$ preserves $T_o S \subseteq T_o M$.
\end{lemma}

\begin{proof}
    Let $s_o$ be the geodesic symmetry of $M$ at $o$, and let $\Theta \in \operatorname{Aut}(\s{G})$ be the conjugation by $s_o$. Recall that $\s{K} \subseteq \s{G}^\Theta$ is an open subgroup. The subgroup $\s{F}$ of $\s{G}$ consisting of all elements preserving $S$ is a closed subgroup, and it is invariant under $\Theta$ because $S$ is invariant under $s_o$. Therefore, the Lie algebra $\g{f}$ splits as a vector space direct sum $\g{f} = (\g{f} \cap \g{k}) \oplus (\g{f} \cap \g{p})$, where $\g{f} \cap \g{k}$ is the Lie algebra of $\s{F} \cap \s{K}$ and thus can be described as the set of elements in $\g{k}$ preserving $T_o S$ under the isotropy representation. Since $\s{H} \subseteq \s{F}$, we have $\pr_\g{k}(\g{h}) \subseteq \g{f} \cap \g{k}$.
\end{proof}

\begin{remark}\label{lemma:foliationseuclidan}
    From the classification of cohomogeneity-one actions on Euclidean spaces, we know that a closed connected subgroup $\s{H} \subseteq \s{SO}_n \ltimes \R^n$ acting on $M_0$ with cohomogeneity one and without singular orbits must have parallel hyperplanes as its orbits. Therefore, we can apply Lemma \ref{lem:preserving_t.g.} to deduce that $\pr_{\g{so}_n}(\g{h})$ preserves $T_o(\s{H} \cdot o)$ and thus $\pr_{\s{SO}_n}(\s{H})$ preserves $\s{H} \cdot o$.
\end{remark}

\subsection{The proof of Theorem~\ref{theorem:foliations}} We now proceed to give a classification of codimension-one homogeneous foliations of symmetric spaces.
For convenience, we split the proof into a series of smaller results that we prove independently.

Let $M = M_+ \times M_0 \times M_-$ be a simply connected symmetric space, where $M_0 \cong \R^n$ and $M_+$ and $M_-$ are of compact and noncompact type, respectively, and let $\s{H}$ be a connected Lie group acting on $M$ properly, isometrically, with cohomogeneity one and no singular orbits. As mentioned in Section~\ref{section:products}, we may assume without loss of generality that $\s{H}$ is a closed connected subgroup of $I^0(M_+) \times I^0(M_0) \times I^0(M_-)$. Moreover, Theorem~\ref{th:main:decompcompactxhadamard} (or more precisely, Corollary~\ref{corollary:decompC1SymmetricSpaces}) guarantees that either the action of $\s{H}$ decomposes or it is orbit-equivalent to the action of $I(M_+) \times \s{H}_{\Delta}$, where $\s{H}_{\Delta}$ acts on $M_0 \times M_-$ with cohomogeneity one and its projection actions on $M_0$ and $M_-$ are both transitive. Note that, in case the action of $\s{H}$ decomposes, the fact that it has no singular orbits is equivalent to saying that its non-transitive projection action has no singular orbits.

First, suppose that the action of $\s{H}$ decomposes. Lemma \ref{lemma:foliationscompact} guarantees that $M_+$ does not admit cohomogeneity-one actions without singular orbits. Furthermore, the only codimension-one homogeneous foliations of $M_0 \cong \R^n$ are those given by parallel affine hyperplanes. This corresponds to item (a) in Theorem~\ref{theorem:foliations}. Next, assume that the projection action of $\s{H}$ on $M_-$ is non-transitive and thus induces a homogeneous codimension-one foliation. Such foliations have been extensively studied in \cite{BT:jdg}, \cite{BDT:jdg} and \cite{Solonenko:reducible}. Every cohomogeneity-one action on $M_-$ without singular orbits is orbit-equivalent to either the action of the connected Lie subgroup $\s{H}_X$ of $\s{G}_-$ with Lie algebra
\[
\g{h}_X = (\g{a} \ominus X) \loplus \g{n},
\]
for some unit vector $X \in \g{a}$, or else to the action of the connected Lie subgroup $\s{H}_i \subseteq \s{G}_-$ with Lie algebra
\[
\g{h}_i = \g{a} \loplus (\g{n} \ominus X)
\]
for a unit vector $X$ in a simple root space $\g{g}_{\alpha_i}$. These cases correspond to items (b) and (c) in Theorem~\ref{theorem:foliations}, respectively.

From now on, we take $\s{H} \subseteq I^0(M_0) \times I^0(M_-) = (\s{SO}_n \ltimes \R^n) \times \s{G}_-$ to be a closed connected subgroup whose action on $M_0 \times M_-$ has cohomogeneity one and no singular orbits and does not decompose. In particular, both $M_0$ and $M_-$ are nontrivial. Thanks to Lemma~\ref{lemma:SplittingC1}, the projection actions of $\s{H}$ on both $M_0$ and $M_-$ must be transitive.
As our first step, we show that we can replace $\s{H}$ with a possibly smaller subgroup that satisfies a number of extra properties.

\begin{lemma}\label{lemma:solvablesubgroupfoliations}
	There exists a closed connected solvable subgroup $\s{L}$ of $\s{H}$ acting on $M_0 \times M_-$ freely such that the orbits of $\s{L}$ and $\s{H}$ coincide. For any such $\s{L}$, we have $\g{l} \cap (\g{so}_n \oplus \g{k}_-) = \{0\}$. Moreover, there exists a choice of $\g{a} \subsetneq \g{p}$ and $\Sigma^+\subsetneq \Sigma$ such that $\g{l} \subseteq (\g{t}_0 \loplus \g{r}^n) \oplus (\g{t}_- \loplus (\g{a} \loplus \g{n}))$, where $\g{t}_0 = \pr_{\g{so}_n}(\g{l})$ and $\g{t}_- \subseteq Z_{\g{k}_-}(\g{a})$ are abelian subalgebras.
\end{lemma}

\begin{proof}
	Since $M_0 \times M_-$ is a Hadamard manifold, there exists a closed connected solvable subgroup $\s{L}$ of $\s{H}$ that acts on $M_0 \times M_-$ freely and has the same orbits as $\s{H}$ (see \cite[Th.\ 3.9]{LN:thesis}). The second statement in the lemma follows immediately from the fact that the action of $\s{L}$ is free. Let $\pr_{\g{so}_n}$ and $\pr_{\g{g}_-}$ denote the projections of $ (\g{so}_n \loplus \g{r}^n) \oplus \g{g}_-$ onto $\g{so}_n$ and $\g{g}_-$, respectively. These are both Lie algebra homomorphisms because $\g{r}^n \oplus \g{g}_-$ and $\g{so}_n \oplus \g{r}^n$ are ideals. Consequently, $\g{t}_0 = \pr_{\g{so}_n}(\g{l})$ is a solvable subalgebra of $\g{so}_n$. It is also a compact Lie algebra because so is $\g{so}_n$. This implies that $\g{t}_0$ must be abelian. Similarly, $\pr_{\g{g}_-}(\g{l})$ is a solvable subalgebra of $\g{g}$, hence it is contained in some maximal solvable subalgebra $\g{b} \subsetneq \g{g}_-$, also known as a Borel subalgebra. Since the projection action of $\s{H}$ on $M_-$ is transitive, so is the projection action of $\s{L}$ on $M_-$, which implies that the connected Lie subgroup of $\s{G}_-$ corresponding to $\g{b}$ also acts on $M_-$ transitively. It follows from the theory of Borel subalgebras of real semisimple Lie algebras (see, e.g., \cite{Mostow:AnnalsMaximal}) that there must exist a choice of $\g{a}\subsetneq\g{p}$ and a system of positive roots $\Sigma^+\subsetneq \Sigma$ such that $\g{b} = \g{t}_- \loplus (\g{a} \loplus \g{n})$, where $\g{t}_-$ is a maximal abelian subalgebra of $Z_{\g{k}_-}({\g{a}})$. This concludes the proof.
\end{proof}

To simplify the notation, we assume from now on that $\s{H}$ itself is solvable and acts on $M_0 \times M_-$ freely. We fix $\g{a} \subsetneq \g{p}$ and $\Sigma^+ \subsetneq \Sigma$ so that $\g{h} \subseteq (\g{t}_0 \loplus \g{r}^n) \oplus (\g{t}_- \loplus (\g{a} \loplus \g{n}))$ for some abelian $\g{t}_0 \subseteq \g{so}_n$ and $\g{t}_- \subseteq Z_{\g{k}_-}(\g{a})$. Consider the projection
\[
(\g{t}_0 \loplus \g{r}^n) \oplus (\g{t}_- \loplus (\g{a} \loplus \g{n})) \to \g{r}^n \oplus (\g{a} \loplus \g{n})
\]
along $\g{t}_0 \oplus \g{t}_-$. It follows from Lemma \ref{lemma:solvablesubgroupfoliations} that the restriction of this projection to $\g{h}$, which we denote by $\pi$, yields a linear isomorphism between $\g{h}$ and $\pi(\g{h}) \cong T_o(\s{H} \cdot o)$.

As mentioned above, the projection actions of $\s{H}$ on $M_0$ and $M_-$ are both transitive, so we may decompose the tangent space to $\s{H} \cdot o$ at $o$ as
\[
\pi(\g{h}) = (\g{r}^n \ominus E) \oplus \R(E+X) \oplus ((\g{a} \loplus \g{n}) \ominus X)
\]
for some nonzero vectors $E \in \g{r}^n$ and $X \in \g{a} \loplus \g{n}$. We now define three subspaces of $\g{h}$ by pulling the summands in the above decomposition back to $\g{h}$. We denote them by $\g{h}_0 = \pi^{-1}(\g{r}^n \ominus E)$, \, $\g{h}_\mathrm{mid} = \pi^{-1}(\R(E+X))$, and $\g{h}_{-} = \pi^{-1}((\g{a} \loplus \g{n}) \ominus X)$. Given $V \in \pi(\g{h})$, we denote by $T_V$ the unique vector in $\g{t}_0 \oplus \g{t}_-$ such that $V + T_V \in \g{h}$. In other words, the map $V \mapsto V + T_V$ is precisely the inverse of $\pi \colon \g{h} \isoto \pi(\g{h})$.

\begin{lemma}
	$\g{h}_0$, $\g{h}_\mathrm{mid}$, and $\g{h}_-$ are subalgebras of $\g{h}$.
\end{lemma}

\begin{proof}
	Let $V + T_V$ and $W + T_W$ be arbitrary vectors in $\g{h}_0$, where $V, W \in \g{r}^n \ominus E$. We compute
	\[
	[V+T_V, W+T_W] = [V,T_W] + [T_V,W] \in \g{h},
	\]
	which lies in $\g{r}^n$ because $\g{t}_0 \oplus \g{t}_-$ normalizes $\g{r}^n$. It is straightforward to see that $\g{r}^n \cap \g{h}$ is contained in $\g{h}_0$. This shows that $\g{h}_0$ is a subalgebra of $\g{h}$. Similarly, if $Y + T_Y$, $Z + T_Y$ are arbitrary vectors in $\g{h}_-$, one has
	\begin{equation}\label{h_-_subalgebra}
	[Y+T_Y, Z+T_Z] = [Y,T_Z] + [T_Y,Z] + [Y,Z] \in \g{h}.
	\end{equation}
	The first two summands in \eqref{h_-_subalgebra} are contained in $\g{a} \loplus \g{n}$ because $\g{t}_- \subseteq Z_{\g{k}_-}(\g{a})$ and the latter normalizes $\g{n}$, whereas the last summand lies in $\g{n}$. Overall, we see that $[Y+T_Y, Z+T_Z] \in (\g{a} \loplus \g{n}) \cap \g{h}$. Similarly to the above, we have $(\g{a} \loplus \g{n}) \cap \g{h} \subseteq \g{h}_-$, which implies that $\g{h}_-$ is also a subalgebra. Lastly, $\g{h}_\mathrm{mid}$ is one-dimensional, so it is an abelian subalgebra.
\end{proof}	

The next lemma puts a strong restriction on what $\g{h}$ can look like.

\begin{lemma}\label{lem:structure_of_h}
	In the above notation, we have $X \in \g{a}$ and $\g{n} \subseteq \g{h}$.
\end{lemma}

\begin{proof}
	Consider $\g{s} = \pr_{\g{g}_-}(\g{h}_-)$. This is a subalgebra of $\g{t}_- \oplus \g{a} \oplus \g{n}$ whose projection to $\g{a} \oplus \g{n}$ along $\g{t}_-$ is a linear isomorphism onto $(\g{a} \loplus \g{n}) \ominus X$. According to \cite[Prop.\ 5.4]{BT:jdg}, there are two options for $X$: either $X \in \g{a}$ or $X \in \R H_\alpha \oplus \g{g}_\alpha$ for some simple root $\alpha \in \Lambda$. We consider the latter case first and assume that $X \not\in \g{a}$, i.e., $X = a H_\alpha + X_\alpha$ for some $a \in \R$ and nonzero $X_\alpha \in \g{g}_\alpha$.
    
    We first prove that we may assume without loss of generality that $X \in \g{g}_\alpha$ or, in other words, $a = 0$. As was shown in \cite[Prop.\ 5.11]{BDT:jdg}, there exists $g \in \s{N}$ such that $\Ad(g) ((\g{a} \oplus \g{n}) \ominus X) = (\g{a} \oplus \g{n}) \ominus X_\alpha$. We consider the conjugate subgroup $\s{H}' = g \s{H} g^{-1}$ with Lie algebra $\g{h}' = \Ad_g(\g{h})$. Clearly, $\s{H}'$ is also a closed connected solvable subgroup of $(\s{SO}_n \ltimes \R^n) \times \s{G}_-$ that acts freely on $M_0 \times M_-$. Moreover, the actions of $\s{H}$ and $\s{H}'$ are orbit-equivalent via $\operatorname{Id}_{M_0}\times g$. It follows that both projection actions of $\s{H}'$ are transitive and the action of $\s{H}'$ does not decompose. Next, notice that $\Ad_g$ acts on $\g{so}_n \loplus \g{r}^n$ trivially and preserves the subalgebras $\g{n}, \, \g{a} \oplus \g{n},$ and $\g{t}_- \oplus \g{a} \oplus \g{n}$, even though it may not preserve $\g{a}$ and $\g{t}_-$. In particular, we see that $\g{h}'$ is also contained in $(\g{t}_0 \loplus \g{r}^n) \oplus (\g{t}_- \loplus (\g{a} \loplus \g{n}))$. The projection of $\g{h}'$ to $\g{r}^n \oplus \g{a} \oplus \g{n}$ along $\g{t}_0 \oplus \g{t}_-$, which we denote by $\pi'$, is an isomorphism onto $\pi'(\g{h}')$, and the latter can be written as
    \[
    \pi'(\g{h}') = (\g{r}^n \ominus E) \oplus \R(E+X') \oplus ((\g{a} \loplus \g{n}) \ominus X')
    \]
    for some nonzero $X' \in \g{a} \loplus \g{n}$. Observe that $\Ad_g(\g{h}_-)$ is a subalgebra of $\g{h}'$, and we have $\pr_{\g{g}_-}(\Ad_g(\g{h}_-)) = \Ad_g(\pr_{\g{g}_-}(\g{h}_-)) = \Ad_g(\g{s})$. Thanks to \cite[Prop.\ 5.11(i)]{BDT:jdg}, the projection of $\Ad_g(\g{s})$ in $\g{a} \oplus \g{n}$ along $\g{t}_-$, which coincides precisely with $\pi'(\Ad_g(\g{h}_-))=(\g{a} \loplus \g{n}) \ominus X'$, is contained in $\Ad(g) ((\g{a} \oplus \g{n}) \ominus X) = (\g{a} \oplus \g{n}) \ominus X_\alpha$. For dimension reasons, it must coincide with the whole $(\g{a} \oplus \g{n}) \ominus X_\alpha$, which means that $X'$ is proportional to $X_\alpha$.
    
    By replacing $\g{h}$ with $\g{h}'$, we may simply write $X \in \g{g}_\alpha$. This means that $\g{a} \subsetneq \pi(\g{h})$ and in particular $H_\alpha \in \pi(\g{h})$. We calculate:
    \begin{equation}\label{structure_of_h}
	[H_\alpha + T_{H_\alpha}, E + X + T_{E+X}] = \alpha(H_\alpha)X + [T_{H_\alpha},E] + [T_{H_\alpha},X] \in \g{h}.
	\end{equation}
    Due to the ad-invariance of the inner product, we have $[T_{H_\alpha},E] \in\g{r}^n\ominus E$ and $[T_{H_\alpha},X] \in\g{g}_\alpha\ominus X$. Since $\alpha(H_\alpha) \ne 0$, we must have $X\in\g{h}_{\mathrm{mid}}$, which is a contradiction. Therefore, we must have that $X \in \g{a}$.
	
    It remains to show that $\g{n} \subseteq \g{h}$. Let $\lambda \in \Sigma^+$ be any positive root and suppose that $\g{h} \cap \g{g}_\lambda \subsetneq \g{g}_\lambda$. Take a nonzero vector $Y_\lambda \in \g{g}_\lambda$ orthogonal to $\g{h} \cap \g{g}_\lambda$. Since $\g{n} \subseteq \pi(\g{h})$, we can consider the vector $Y_\lambda + T_{Y_\lambda} \in \g{h}$. Moreover, there exists $H \in \g{a} \ominus X$ such that $\lambda$ takes nonzero value at $X + H$. We have:
	\[
	[E + X + H + T_{E+X+H}, Y_\lambda + T_{Y_\lambda}] = [E, T_{Y_\lambda}] + \lambda(X+H) Y_\lambda + [T_{E+X+H}, Y_\lambda] \in \g{h}.
	\]
	Note that $[E, T_{Y_\lambda}] \in \g{r}^n \ominus E$ and $[T_{E+X+H}, Y_{\lambda}] \in \g{g}_\lambda \ominus Y_{\lambda}$. With respect to the decomposition $\g{h} = \g{h}_0 \oplus \g{h}_\mathrm{mid} \oplus \g{h}_-$, it is easy to see that $[E, T_{Y_\lambda}] \in \g{h}_0$ and $\lambda(X+H) Y_\lambda + [T_{E+X+H}, Y_\lambda] \in \g{h}_-$. Since $\lambda(X+H) \ne 0$, the latter vector lies outside of $\g{h} \cap \g{g}_\lambda$, which is a contradiction. We deduce that $\g{h}$ contains $\g{g}_\lambda$ and, as  $\lambda$ was arbitrary, $\g{n} \subseteq \g{h}$, which concludes the proof.
\end{proof}

Note that, in light of Lemma \ref{lem:structure_of_h}, the projection $\pi(\g{h})$ of $\g{h}$ in $\g{r}^n \oplus (\g{a} \loplus \g{n})$ is precisely the subalgebra $\g{h}_{E,X}$ described in Theorem~\ref{theorem:foliations}. The next theorem is the final step in our line of argument and it concludes the proof of Theorem~\ref{theorem:foliations}.

\begin{theorem}
    The connected Lie subgroup $\s{H}_{E,X}$ of $(\s{SO}_n \ltimes \R^n) \times \s{G}_-$ with Lie algebra $\g{h}_{E,X} = \pi(\g{h})$ has the same orbits in $M_0 \times M_-$ as $\s{H}$.
\end{theorem}

\begin{proof}
    Our goal is to construct a connected Lie subgroup $\widehat{\s{H}}$ of $(\s{SO}_n \ltimes \R^n) \times \s{G}_-$ containing both $\s{H}$ and $\s{H}_{E,X}$ that acts with cohomogeneity one and thus has the same orbits as either of these two subgroups. We are going to define $\widehat{\s{H}}$ as the connected Lie subgroup with Lie algebra $\widehat{\g{h}} = \g{t}_0 \oplus \g{t}_- \oplus \pi(\g{h})$, but first we need to prove that $\widehat{\g{h}}$ is indeed a subalgebra, i.e., that both $\g{t}_0$ and $\g{t}_-$ normalize $\pi(\g{h})$. Thanks to Lemma \ref{lem:structure_of_h}, we know that
    \[
    \pi(\g{h}) = \g{h}_{E,X} = (\g{r}^n \ominus E) \oplus \R(E+X) \oplus (\g{a} \ominus X) \oplus \g{n}.
    \]
    Since $\g{t}_-$ is contained in $Z_{\g{k}_-}(\g{a})$, it commutes with $\g{r}^n \oplus \g{a}$ and normalizes $\g{n}$, hence it must normalize $\pi(\g{h})$. 
    
    In view of the $\ad$-invariance of the inner product, in order to show that $\g{t}_0 = \pr_{\g{so}_n}(\g{h})$ normalizes $\pi(\g{h})$, it suffices to prove that it normalizes $\g{r}^n \ominus E$. We will do this separately for $\pr_{\g{so}_n}(\g{h}_0), \, \pr_{\g{so}_n}(\g{h}_\mathrm{mid})$, and $\pr_{\g{so}_n}(\g{h}_-)$. Consider the subalgebra $\widecheck{\g{h}}_0 = \pr_{\g{so}_n \loplus \g{r}^n}(\g{h}_0)$ and its corresponding connected Lie subgroup $\widecheck{\s{H}}_0 \subseteq \s{SO}_n \ltimes \R^n$. Write $\s{H}_0$ and $\s{T}_-$ for the connected Lie subgroups of $(\s{SO}_n \ltimes \R^n) \times \s{G}_-$ corresponding to $\g{h}_0$ and $\g{t}_-$, respectively. Note that the latter subgroup is compact. Moreover, we have $\g{h}_0 = \g{h} \cap (\g{so}_n \loplus \g{r}^n) \oplus \g{t}_-$ and hence $\s{H}_0$ is the identity component of $\s{H} \cap ((\s{SO}_n \ltimes \R^n) \times \s{T}_-)$. Since this intersection is a closed subgroup, so is $\s{H}_0$. Observe that $\widecheck{\s{H}}_0$ is the image of $\s{H}_0$ under the projection $(\s{SO}_n \ltimes \R^n) \times \s{T}_- \to \s{SO}_n \ltimes \R^n$. As $\s{T}_-$ is compact, this projection is a closed map. We deduce that $\widecheck{\s{H}}_0$ is a closed subgroup and thus its action on $M_0$ is proper. 
    
    We would now like to show that this action has no singular orbits. Observe that we can think of it as the effectivization of the projection action $\s{H}_0 \curvearrowright M_0$. Therefore, it suffices to show that the latter action has discrete isotropy at every point. To this end, we first notice that the vector space sum $\g{h}_0 \oplus \g{h}_\mathrm{mid}$ is easily seen to be a subalgebra of $\g{h}$. We write $\s{H}_{0, \mathrm{mid}} \subseteq \s{H}$ for the corresponding connected Lie subgroup. Since $\s{H}_0 \subseteq \s{H}_{0, \mathrm{mid}}$, it suffices to show that all isotropy subgroups of the projection action $\s{H}_{0, \mathrm{mid}} \curvearrowright M_0$ are discrete. Since the projection of $\g{h}_0 \oplus \g{h}_\mathrm{mid}$ to $\g{r}^n$ is surjective, the action $\s{H}_{0, \mathrm{mid}} \curvearrowright M_0$ has cohomogeneity 0 and is thus transitive. Consequently, we just need to show that its isotropy subgroup at $0$ is discrete. To see this, just notice that $\g{h}_0 \oplus \g{h}_\mathrm{mid} \subseteq (\g{t}_0 \loplus \g{r}^n) \oplus (\g{t}_- \oplus \g{a})$ has trivial intersection with $\g{t}_0 \oplus \g{t}_- \oplus \g{a}$ because its projection to $\g{r}^n$ is a linear isomorphism. Altogether, we conclude that $\widecheck{\s{H}}_0$ is a closed connected subgroup of $\s{SO}_n \ltimes \R^n$ that acts on $M_0$ without singular orbits. As $\pr_{\g{r}^n}(\widecheck{\g{h}}_0) = \pi(\g{h}_0) = \g{r}^n \ominus E$, this action has cohomogeneity one, so its orbits are a family of mutually parallel affine hyperplanes. This allows us to apply Remark \ref{lemma:foliationseuclidan}, which states that the projection $\pr_{\g{so}_n}(\widecheck{\g{h}}_0) = \pr_{\g{so}_n}(\g{h}_0)$ preserves $T_0(\widecheck{\s{H}}_0 \cdot 0)$, or in other words, normalizes $\g{r}^n \ominus E$.

    It remains to show that the same is true for $\pr_{\g{so}_n}(\g{h}_\mathrm{mid})$ and $\pr_{\g{so}_n}(\g{h}_-)$. Given $V \in \g{r}^n \ominus E$, a straightforward computation yields
	\[
	[E + X + T_{E+X}, V + T_V] = [E, T_V] + [T_{E+X}, V] \in \g{r}^n \cap \g{h} \subseteq \g{r}^n \ominus E.
	\]
	Since $[E,T_V]$ is orthogonal to $E$, so must be $[T_{E+X},V]$. Therefore, $\pr_{\g{so}_n}(\g{h}_\mathrm{mid}) = \R \hspace{1pt} T_{E+X}$ normalizes $\g{r}^n \ominus E$. Similarly, if $H \in \g{a} \ominus X$, one has
	\[
	[H + T_H, V + T_V] = [T_H, V] \in \g{r}^n \cap \g{h} \subseteq \g{r}^n \ominus E,
	\]
	so $\pr_{\g{so}_n}(\g{h}_-)$ normalizes $\g{r}^n \ominus E$. We conclude that $\g{t}_0 = \pr_{\g{so}_n}(\g{h})$ normalizes $\g{r}^n \ominus E$ and thus the whole $\pi(\g{h})$, so $\widehat{\g{h}} = \g{t}_0 \oplus \g{t}_- \oplus \pi(\g{h})$ is a subalgebra of $(\g{so}_n \loplus \g{r}^n) \oplus \g{g}_-$.

    Consider the connected Lie subgroup $\widehat{\s{H}}$ of $(\s{SO}_n \ltimes \R^n) \times \s{G}_-$ with Lie algebra $\widehat{\g{h}}$. It is clear from our construction that $\g{h} \subseteq \widehat{\g{h}}$, so the orbits of $\s{H}$ in $M_0 \times M_-$ are contained in those of $\widehat{\s{H}}$. Now, $T_o(\s{H} \cdot o) \cong \pi(\g{h}) = \pi(\widehat{\g{h}}) \cong T_o(\widehat{\s{H}} \cdot o)$, so $\s{H} \cdot o = \widehat{\s{H}} \cdot o$. This implies that $\s{H}$ and $\widehat{\s{H}}$ share the same orbits. Similarly, since $\g{h}_{E,X} = \pi(\g{h}) \subseteq \widehat{\g{h}}$, the orbits of $\s{H}_{E,X}$ are contained in the orbits of $\widehat{\s{H}}$. As $T_o(\s{H}_{E,X} \cdot o) \cong \g{h}_{E,X} = \pi(\widehat{\g{h}}) \cong T_o(\widehat{\s{H}} \cdot o)$, we deduce that $\s{H}_{E,X}$ has the same orbits as $\widehat{\s{H}}$ and thus as $\s{H}$. This completes the proof.
\end{proof}	 

\subsection{The geometry of the orbits of $\s{H}_{E,X}$}\label{section:HFSS:geometry_of_orbits}
Here we briefly outline the geometry of the orbits of the action of $\s{H}_{E,X}$ on $M_0 \times M_-$. To begin with, we show that we can reduce the investigation to a single orbit. Indeed, note that the derived subalgebra of $\g{r}^n \oplus (\g{a} \loplus \g{n})$ is precisely $\g{n}$. Since $\g{n} \subseteq \g{h}_{E,X}$, it follows that $\g{h}_{E,X}$ is an ideal in $\g{r}^n \oplus (\g{a} \loplus \g{n})$. Therefore, the orbits of $\s{H}_{E,X}$ are mutually congruent (see~\cite[Lem.\ 2.1]{KT:GeomDed}), so we can fully understand the geometry of the orbits of $\s{H}_{E,X}$ by only studying the geometry of the orbit through $o$. Note that, under the identification $M_0 \times M_- \cong \R^n \times (\s{A} \ltimes \s{N})$, the orbit $\s{H}_{E,X} \cdot o$ simply becomes $\s{H}_{E,X}$. The vector $\xi = \frac{ ||E||^2 X-||X||^2 E }{||E|| ||X|| \sqrt{||E||^2 + ||X||^2}}$, thought of as a left-invariant vector field on $\R^n \times (\s{A} \ltimes \s{N})$, is then a unit normal field to $\s{H}_{E,X}$. Denote by $\mathcal{A}_\xi$ the shape operator of $\s{H}_{E,X}$ with respect to $\xi$, and let $\nabla$ denote the Levi-Civita connection of $\R^n \times (\s{A} \ltimes \s{N})$. Note that $\ad_\xi$ is a self-adjoint endomorphism of $\g{r}^n \oplus (\g{a} \loplus \g{n})$. The Koszul formula, combined with the left-invariance of the metric on $\R^n \times (\s{A} \ltimes \s{N})$, yields
\[
2 \langle \mathcal{A}_\xi V, W \rangle = 2 \langle -\nabla_V \xi, W \rangle = - \langle [\xi, V], W \rangle + \langle [V, W], \xi \rangle + \langle [\xi, W], V \rangle = 2 \langle \ad_\xi V, W \rangle
\]
for every $V, W \in \g{h}_{E,X}$. Consequently, the shape operator $\mathcal{A}_\xi$ is simply given by $\restr{\ad_\xi}{\g{h}_{E,X}}$. It is straightforward to check that $\xi$ commutes with $V_0 = (\g{r}^n \ominus E) \oplus \R(E+X) \oplus (\g{a} \ominus X)$, which implies that $V_0$ is contained in the principal curvature space of $\s{H}_{E,X}$ with principal curvature $0$. Similarly, $\ad_\xi Y_\alpha = \frac{\alpha(X) ||E||}{||X|| \sqrt{||E||^2 + ||X||^2}} Y_{\alpha}$ for any positive root $\alpha \in \Sigma^+$ and $Y_\alpha \in \g{g}_\alpha$. Thus, for each such $\alpha$, $\g{g}_\alpha$ is contained in the principal curvature space with principal curvature $\frac{\alpha(X) ||E||}{||X||\sqrt{||E||^2 + ||X||^2}}$. We round up the above discussion with the following
\begin{proposition}\label{prop:shape orbits hex}
	The orbits of $\s{H}_{E,X}$ in $M_0 \times M_-$ are mutually congruent.
    The shape operator of $\s{H}_{E,X} \cdot o$ with respect to $\xi$ coincides with $\restr{\ad_\xi}{\g{h}_{E,X}}$. The mean curvature of $\s{H}_{E,X} \cdot o$ is given by
	\[
	\frac{||E||}{||X|| \sqrt{||E||^2 + ||X||^2}} \sum_{\alpha \in \Sigma^+} \dim(\g{g}_\alpha) \alpha(X).
	\]
	In particular, the orbits of $\s{H}_{E,X}$ are minimal if and only if $\sum_{\alpha \in \Sigma^+} \dim(\g{g}_\alpha) \alpha(X) = 0$.
\end{proposition}

\begin{remark}
	It should be noted that the geometry of the orbits of $\s{H}_{E,X}$ is analogous to that of the orbits of the action $\s{H}_X \curvearrowright M_-$ described in item (b) of Theorem~\ref{theorem:foliations}. Namely, \cite[Prop.\ 3.1]{BT:jdg} shows that the orbits of $\s{H}_X$ are mutually congruent and have mean curvature $\sum_{\alpha \in \Sigma^+} \dim(\g{g}_\alpha) \alpha(X)$. Note that if the rank of $M_-$ is $\geq 2$, we can always find a nonzero vector $X_0 \in \g{a}$ with $\sum_{\alpha \in \Sigma^+} \dim(\g{g}_\alpha) \alpha(X_0) = 0$. Recall that a foliation on a Riemannian manifold is said to be harmonic if all of its orbits are minimal. In the context of Theorem~\ref{theorem:foliations}, the actions given by $I(M_+) \times \R^{n-1} \times \s{G}_-, \, \s{G}_+ \times \R^n \times \s{H}_{X_0}$, and $\s{G}_+ \times \s{H}_{E,X_0}$ (where $X_0$ is as above) give rise to non-congruent harmonic foliations on $M_+ \times M_0 \times M_-$.
\end{remark}

%% file: 5-Actions_with_singular_orbits.tex
In this section, we prove Theorem~\ref{th:singorbdecomp}, which states that a cohomogeneity-one action $\s{H} \curvearrowright M_0\times M_-$ with a singular orbit must necessarily decompose. We are going to argue by contradiction. Recall that the assumption that the action of $\s{H}$ does not decompose implies that the projection actions of $\s{H}$ on both $M_0$ and $M_-$ are transitive. This naturally leads us to transitive isometric actions on Euclidean spaces, which we discuss in Subsection~\ref{subsection:singorb:transitivern}. For a symmetric space of noncompact type $M_- = \s{G}_-/\s{K}_-$, the maximal proper connected Lie subgroups of $\s{G}_-$ acting on $M_-$ transitively are precisely the (identity components of the) maximal proper parabolic subgroups. In Subsection~\ref{subsection:singorb:canonical}, we extend some results of \cite{BT:crelle} to spaces of the form $M_0 \times M_-$, showing that if $\pr_{\g{g}_-}(\g{h})$ is contained in some maximal parabolic subalgebra $\g{q}_j$ of $\g{g}_-$, then the action of $\s{H}$ can be obtained by extending some cohomogeneity-one action on a lower rank symmetric space. Finally, in Subsection~\ref{subsection:singorb:singorb}, we conclude with the proofs of Theorems \ref{th:singorbdecomp} and \ref{mainth}.

\subsection{Transitive actions on Euclidean spaces.}\label{subsection:singorb:transitivern}
We now take a look at transitive groups of isometries of Euclidean spaces: we recall some structural results and turn them into a complete and explicit description of such groups, which is going to be crucial for the proof of Theorem \ref{th:singorbdecomp}.  Such actions have been extensively studied in the literature in the much more general context of homogeneous Hadamard manifolds and Riemannian solvmanifolds (see \cite{Alekseevski,AW:homogeneous,GW:isomsolv}). For instance, Alekseevskii gave an explicit description of simply transitive groups of isometries of Euclidean spaces in  \cite[Prop.\ 3.1]{Alekseevski}. For any Hadamard manifold $N$ and a Lie subgroup $\s{H} \subseteq I(N)$ acting transitively on $N$, there exists a solvable Lie subgroup $\s{F} \subseteq \s{H}$ whose action on $N$ is simply transitive. In the special case of $N \cong \R^n$, one can actually take $\s{F}$ to be a normal subgroup of $\s{H}$ such that $\s{H} = \s{H}_p \cdot \s{F}$ for some point $p \in N$. Combining these results together, one can arrive at a description of all transitive groups of isometries of the Euclidean space. However, below we include a proof of the latter result concerning the existence of a simply transitive subgroup in the special case of Euclidean spaces in order to clarify some technical details: for instance, it is shown in \cite[Th.\ 1.1]{Alekseevski} that, for any homogeneous Riemannian manifold $N$ of nonpositive curvature, one can always find a ``normal transitive"\footnote{Note that the notion of ``normal transitive'' subgroup in~\cite{Alekseevski} is nonstandard: a ``normal transitive'' subgroup of $I^0(N)$ is not actually a normal subgroup unless the Levi factor of $I^0(N)$ is compact.} subgroup $\s{F}$ of $I^0(N)$. One could try to use the same approach to find a normal transitive subgroup of any transitive group of isometries $\s{H}$ of $N$, however, there is a caveat: the proof of \cite[Th.\ 1.1]{Alekseevski} uses the fact that the isotropy groups of points $p \in N$ are maximal compact subgroups, which may not be true for $\s{H} \curvearrowright N$ if the action of $\s{H}$ is non-proper (that is, if $\s{H}$ is not closed in  $I(N)$). This same closedness assumption seems to be implicitly made in the proof of~\cite[Prop.\ 2.5]{AW:homogeneous}. In the special case when $\s{H} \subseteq I^0(N)$ is solvable, one can apply \cite[Lem.\ 1.2]{GW:isomsolv}, whose proof avoids this technical problem.  By combining these approaches for the Euclidean space, we arrive at the following

\begin{lemma}\label{lem:solvable_simply_transitive_subgroup}
    Let $\s{H} \subseteq \s{SO}_n \ltimes \R^n = I^0(M_0)$ be a (not necessarily closed) connected Lie subgroup acting transitively on the Euclidean space $M_0$.
    Then, there exists a connected normal solvable Lie subgroup $\s{F}$ of $\s{H}$ that acts on $M_0$ simply transitively.
\end{lemma}

\begin{proof}
    Fix a Levi-Malcev decomposition $\g{h} = \g{l} \ltimes \g{rad(h)}$ and let $\s{L}$ and $\s{Rad(H)}$ be the corresponding connected Lie subgroups of $\s{H}$. Here $\s{Rad(H)}$ is a closed normal solvable Lie subgroup of $\s{H}$, $\s{L}$ is semisimple, and the intersection $\s{L} \cap \s{Rad(H)}$ is discrete. We claim that the Levi factor $\g{l}$ is compact. Indeed, assume that $\g{l}$ has a noncompact ideal $\g{l}_i$. The image of $\g{l}_i$ under the projection map $\pr_{\g{so}_n} \colon \g{so}_n \loplus \g{r}^n \to \g{so}_n$ (which is a homomorphism since $\g{r}^n$ is an ideal) is a noncompact semisimple subalgebra of $\g{so}_n$, so it must be trivial. In other words, $\g{l}_i$ must be contained in $\g{r}^n$ and is thus abelian, which is a contradiction. Therefore, $\g{l}$ is a compact semisimple Lie algebra and thus $\s{L}$ is a compact subgroup. By Cartan's fixed point theorem, $\s{L}$ must have a fixed point $p \in M_0$. We may assume, possibly after conjugating the Levi factor by an element of $\s{H}$, that $p = 0$ and thus $\s{L} \subseteq \s{SO}_n$. Moreover, $M_0 = \s{H} \cdot 0 = \s{Rad(H)} \cdot (\s{L} \cdot 0) = \s{Rad(H)} \cdot 0$, so $\s{Rad(H)}$ acts on $M_0$ transitively.

    Consider now the isotropy subgroup $\s{C} = \s{Rad(H)} \cap \s{SO}_n$ of $\s{Rad(H)}$ at $0$ and let $\s{Nil(H)}$ be the nilradical of $\s{H}$ (i.e., the maximal connected normal nilpotent Lie subgroup). Note that $\s{C}$ may fail to be a closed subgroup of $\s{SO}_n$ because we did not require $\s{H}$ to be closed either. We claim that the intersection $\s{C} \cap \s{Nil(H)}$ is trivial. Indeed, note that the Lie algebra of $\s{Nil(H)}$ is the nilradical $\g{nil(h)}$ of $\g{h}$. For every element $X \in \g{nil(h)}$, the operator $\ad_{\g{rad(h)}}(X)$ is nilpotent. By Engel's theorem, there exists a basis for $\g{rad(h)}$ such that $\ad_{\g{rad(h)}}(X)$ is strictly upper-triangular for every $X \in \g{nil(h)}$. In that same basis, $\Ad_{\s{Rad(H)}}(g)$ is upper-triangular with ones on the diagonal for every $g \in \s{Nil(H)}$. In particular, the only possible complex eigenvalue of $\Ad_{\s{Rad(H)}}(g)$ is 1. On the other hand, note that the $\Ad_{\s{SO}_n}$-invariant inner product on $\g{so}_n \ltimes \g{r}^n$ restricts to an $\Ad_{\s{C}}$-invariant inner product on $\s{Rad(H)}$, and thus $\Ad_{\s{Rad(H)}}(g)$ is complex diagonalizable for every $g \in \s{C}$. We conclude that $\Ad_{\s{Rad(H)}}(g)$ is the identity for every $\s{C} \cap \s{Nil(H)}$, hence $\s{C} \cap \s{Nil(H)} \subseteq \s{Z(Rad(H))}$. Since $\s{Rad(H)}$ acts on $M_0$ transitively, any element of $\s{Nil(H)} \cap \s{SO}_n$ has to fix every point of $M_0$ and thus be trivial. We conclude that $\s{Nil(H)} \cap \s{SO}_n = \{e\}$.

    Now, choose a subspace $\g{f}$ complementary to $\g{c}$ in $\g{rad(h)}$ such that $\g{nil(h)} \subseteq \g{f}$ (this is possible since $\g{c} \cap \g{nil(h)} = \{0\}$), and let $\s{F}$ the corresponding connected Lie subgroup of $\s{Rad(H)}$. Thanks to \cite[Prop.\ 1.41]{Knapp}, we have $[\g{f}, \g{h}] \subseteq [\g{rad(h)}, \g{h}] \subseteq \g{nil(h)} \subseteq \g{f}$, so $\g{f}$ is an ideal in $\g{h}$ and thus $\s{F}$ is a normal subgroup of $\s{H}$. In particular, we have $\g{rad(h)} = \g{c} \loplus \g{f}$ and thus $\s{Rad(H)} = \s{C} \cdot \s{F}$. By construction, $\s{C}$ fixes the origin, and so $M_0 = \s{Rad(H)} \cdot 0 = \s{F} \cdot (\s{C} \cdot 0) = \s{F} \cdot 0$. We deduce that $\s{F}$ is a normal solvable connected Lie subgroup of $\s{H}$ that acts on $M_0$ transitively. Since $\g{f} \cap \g{so}_n = \{0\}$, we actually have that $\s{F}$ is simply connected and closed in $I(M_0)$ and its action on $M_0$ is simply transitive (see \cite[Lem.\ 2.4.]{AW:homogeneous}).
\end{proof}

\begin{remark}\label{remark:isotropynormalizesscrewm}
    It follows from the proof that $\s{H}$ splits as a semidirect product $\s{H} = (\s{H} \cap \s{SO}_n) \ltimes \s{F}$.
    Thus, we can recover any connected Lie subgroup of $I(M_0)$ acting on $M_0$ transitively from its simply transitive part $\s{F}$ by taking the semidirect product of $\s{F}$ with some Lie subgroup of the normalizer $N_{\s{SO}_n}(\s{F})$.
\end{remark}

In \cite{Alekseevski}, Alekseevskii gives the following explicit description of simply transitive actions on Euclidean spaces:

\begin{lemma}[{\cite[Prop.\ 3.1]{Alekseevski}}]\label{lemma:simplytransitivern}
    Let $\g{r}^n = \g{v} \oplus \g{w}$ be an orthogonal decomposition of $\g{r}^n$, and let $\varphi\colon\g{v}\to\g{so}(\g{w})\subseteq\g{so}_n$ be an injective Lie algebra homomorphism. Consider the subalgebra $\g{f} = \g{v}^{\varphi} \loplus \g{w}$ of $\g{so}_n \ltimes \g{r^n}$, where $\g{v}^{\varphi} = \{v + \varphi(v) \mid v \in \g{v} \}$. Then the connected Lie subgroup $\s{F}$ of $\s{SO}_n \ltimes \R^n$ with Lie algebra $\g{f}$ acts on $M_0$ simply transitively. Conversely, a connected Lie subgroup $\s{F} \subseteq \s{SO}_n \ltimes \R^n$ acting on $M_0$ simply transitively has Lie algebra of the form $\g{f} = \g{v}^{\varphi} \loplus \g{w}$ for some injective homomorphism $\varphi \colon \g{v} \to \g{so}(\g{w})$, where $\g{w} = \g{f} \cap \g{r}^n$ and $\g{v} = \g{r}^n \ominus \g{w}$.
\end{lemma}

\begin{remark}
    Notice that the image of $\varphi$ is an abelian subalgebra of $\g{so(w)}$, hence we have $\dim(\g{v}) \le \operatorname{rank}(\g{so(w)}) = \lfloor \frac{n - \dim(\g{v})}{2} \rfloor$, or equivalently, $\dim(\g{v}) \le \frac{n}{3}$.
\end{remark}

We are now ready to give a complete description of all transitive actions on $M_0 \cong \R^n$. Let $\s{H}$ be a connected Lie subgroup of $I(M_0)$ whose action on $M_0$ is transitive. Thanks to Lemma \ref{lem:solvable_simply_transitive_subgroup} and Remark \ref{remark:isotropynormalizesscrewm}, we can write $\s{H} = (\s{H} \cap \s{SO}_n) \ltimes \s{F}$, where $\s{F}$ acts on $M_0$ simply transitively. By virtue of Lemma~\ref{lemma:simplytransitivern}, $\g{f} = \g{v}^{\varphi} \loplus \g{w}$, where $\g{w} = \g{f} \cap \g{r}^n$, $\g{v} = \g{r}^n \ominus \g{w}$, and $\varphi \colon \g{v} \to \g{so(w)}$ is an injective Lie algebra homomorphism. In the proof of Lemma \ref{lem:solvable_simply_transitive_subgroup}, we saw that $[\g{h},\g{f}] \subseteq \g{nil}(\g{h}) \subseteq \g{f}$. We must then have $\g{nil}(\g{h}) \subseteq \g{nil}(\g{f})$, and the latter nilradical is easily seen to coincide with $\g{w}$. In particular, $\g{h} \cap \g{so}_n$ preserves $\g{w}$. Since the inner product on $\g{so}_n \loplus \g{r}^n$ is $\g{so}_n$-invariant, we see that $\g{h} \cap \g{so}_n$ must then preserve $\g{v}^\varphi = \g{f} \ominus \g{w}$ as well. We deduce that $[\g{h} \cap \g{so}_n, \g{v}^\varphi] = \{0\}$. This is only possible when $\g{h} \cap \g{so}_n$ commutes with the projections of $\g{v}^\varphi$ in both $\g{so}(\g{w})$ and $\g{r}^n$, which are equal to $\varphi(\g{v})$ and $\g{v}$, respectively. But $\g{h} \cap \g{so}_n$ can only commute with $\g{v}$ when it is contained in $\g{so(w)}$. We conclude that $\g{h} \cap \g{so}_n$ is contained in $Z_{\g{so}(\g{w})}(\varphi(\mathfrak{v}))$. Finally, note that $\g{h} \cap \g{r}^n = \g{w}$ implies that $\g{h} \cap \g{so}_n$ intersects $\varphi(\g{v})$ trivially.  Rewriting $\g{h} = ((\g{h} \cap \g{so}_n) \oplus \g{v}^\varphi) \loplus \g{w}$, one can think of these three summands as the rotational, screw-motion, and translational parts of $\g{h}$, respectively.

Conversely, in order to construct a transitive group of isometries of $M_0$, start with an orthogonal splitting $\g{r}^n = \g{v} \oplus \g{w}$ and a subalgebra $\g{c} \subseteq \g{so(w)}$. Split $\g{c} = \g{z(c)} \oplus \g{c}_\mathrm{ss}$ into its center and semisimple part and split $\g{z(c)} = \g{a} \oplus \g{b}$ so that $\dim(\g{a}) = \dim(\g{v})$. Finally, pick any isomorphism $\varphi \colon \g{v} \isoto \g{a}$. Then the connected Lie subgroup of $\s{SO}_n \ltimes \R^n$ with Lie algebra $((\g{b} \oplus \g{c}_\mathrm{ss}) \oplus \g{v}^\varphi) \loplus \g{w}$ acts on $M_0$ transitively.
	
\subsection{The canonical extension procedure on $M_0\times M_-$}\label{subsection:singorb:canonical}

In \cite{BT:crelle}  Berndt, together with the third author, introduced two methods for constructing cohomogeneity-one actions on a given symmetric space $M_- =\s{G}_-/\s{K}_-$ of noncompact type: the \emph{canonical extension} and the \emph{nilpotent construction}. Moreover, they showed that if $\s{Q}$ is a proper parabolic subgroup of $\s{G}_-$ and $\s{H} \subseteq \s{Q}$ is a closed connected subgroup acting on $M_-$ with cohomogeneity one and a singular orbit, then the action of $\s{H}$ is always orbit-equivalent to an action constructed by one of these procedures.

For reasons that will become evident below (see Theorem \ref{thm:parab}), in this subsection we will focus on the canonical extension method. In its original formulation, this method allows one to extend any isometric action $\s{H}_\Phi \curvearrowright B_\Phi$ on a boundary component $B_\Phi$ of $M_-$ to an isometric action on $M_-$ of the same cohomogeneity. This procedure was later generalized in \cite{DV:IMRN} to a larger class of Riemannian manifolds endowed with a 
polar action of a specific type. We apply that generalized method in the special case when the ambient space is $M_0 \times M_-$. By adapting the arguments of \cite{BT:crelle}, we show that if $\s{Q}_j$ is a maximal proper parabolic subgroup of $\s{G}_-$ and $\s{H} \subseteq (\s{SO}_n \ltimes \R^n) \times \s{Q}_j$ acts on $M_0 \times M_-$ with cohomogeneity one and a singular orbit and its action does not decompose, then it is orbit-equivalent to some action obtained by canonical extension from $M_0 \times B_j$. This result will prove instrumental in the proof of Theorem \ref{th:singorbdecomp}. 

We begin by introducing the canonical extension method on $M_0 \times M_-$. Let $\g{q}_\Phi$ be the parabolic subalgebra of $\g{g}_-$ associated with a subset of simple roots $\Phi \subseteq \Lambda$, and consider the corresponding horospherical decomposition
\[
M_- \cong B_\Phi \times \s{A}_\Phi \times \s{N}_\Phi.
\]
Let $\s{H}$ be a connected Lie group acting on $M_0 \times B_\Phi$ properly and isometrically. Without loss of generality, we may assume that $\s{H}$ is a closed subgroup of $I^0(M_0 \times B_\Phi) = I^0(M_0) \times I^0(B_\Phi)$ and then lift it to $I^0(M_0) \times \s{G}_\Phi$ (recall from Subsection \ref{subsection:prelim:parabolic} that the subgroup $\s{G}_\Phi \subseteq \s{G}_-$ is a finite covering of $I^0(B_\Phi)$). Altogether, we assume that $\s{H}$ is a closed connected subgroup of $I^0(M_0) \times \s{G}_\Phi$.
Since $\g{g}_\Phi$ normalizes $\g{a}_\Phi \loplus \g{n}_\Phi$, we have that
\[
\g{h}^\Lambda = (\g{h} \oplus \g{a}_\Phi) \loplus \g{n}_\Phi
\]
is a subalgebra of $(\g{so}_n \loplus \g{r}^n) \oplus \g{g}_-$. We write $\s{H}^\Lambda$ for the corresponding connected Lie subgroup of $(\s{SO}_n \ltimes \R^n) \times \s{G}_-$. With respect to the Langlands decomposition, one has
\begin{align*}
\s{H}^\Lambda = (\s{H} \times \s{A}_\Phi) \ltimes \s{N}_\Phi &\subseteq (\s{SO}_n \ltimes \R^n) \times ((\s{G}_\Phi \times \s{A}_\Phi) \ltimes \s{N}_\Phi) \\
&\subseteq (\s{SO}_n \ltimes \R^n) \times ((\s{M}_\Phi \times \s{A}_\Phi) \ltimes \s{N}_\Phi) = I^0(M_0) \times \s{G}_-.
\end{align*}
In particular, this shows that $\s{H}^\Lambda$ is a closed subgroup and thus acts on $M_0 \times M_-$ properly. By construction, $\g{h}^\Lambda \cap (\g{so}_n \oplus \g{k}_-) = \g{h} \cap (\g{so}_n \oplus (\g{g}_\Phi \cap \g{k}_-))$. Moreover, if we identify $T_o(M_0 \times M_-) \cong \g{r}^n \oplus \g{b}_\Phi \oplus \g{a}_\Phi \oplus \g{n}_\Phi$ and $T_o(M_0 \times B_\Phi) \cong \g{r}^n \oplus \g{b}_\Phi$, then it is easy to see that the normal space to $\s{H}^\Lambda \cdot o$ at $o$ in $M_0 \times M_-$ coincides with that to $\s{H} \cdot o$ at $o$ in $M_0 \times B_\Phi$. Therefore, both actions have the same slice representation at $o$ (on the level of Lie algebras), and consequently, the same cohomogeneity. We refer to the action of $\s{H}^\Lambda$ on $M_0 \times M_-$ as the \emph{canonical extension} of $\s{H} \curvearrowright M_0 \times B_\Phi$.

\begin{remark}\label{rem:CE_proper}
    If $\s{H}$ is a connected Lie group acting on $M_0 \times B_\Phi$ isometrically but not properly, we can still define its canonical extension in the exact same way, the only difference being that the action of $\s{H}^\Lambda$ on $M_0 \times M_-$ may fail to be proper. Using the horospherical decomposition, one can readily see that the orbits of $\s{H}^\Lambda \curvearrowright M_0 \times M_-$ are properly embedded if and only if such are the orbits of $\s{H} \curvearrowright M_0 \times B_\Phi$. Consequently, if we know that a proper isometric action on $M_0 \times M_-$ is orbit-equivalent to the canonical extension of an isometric action $\s{H} \curvearrowright M_0 \times B_\Phi$, we may assume without loss of generality that the action of $\s{H}$ on $M_0 \times B_\Phi$ is also proper.
\end{remark}
		
\begin{lemma}\label{lemma:decompcanonical}
	Let $M_0$ and $M_-$ be as above, and let $B_\Phi$ be a boundary component of $M_-$. Let $\s{H}$ be a connected Lie group acting on $M_0 \times B_\Phi$ properly, isometrically, and with cohomogeneity one. Then, the canonical extension $\s{H}^\Lambda \curvearrowright M_0 \times M_-$ decomposes if and only if $\s{H} \curvearrowright M_0\times B_\Phi$ does.
\end{lemma}

\begin{proof}
	Suppose that the action of $\s{H}$ on $M_0 \times B_\Phi$ decomposes. Then, it has the same orbits as some product action $\s{H}_0 \times \s{H}_\Phi \curvearrowright M_0 \times B_{\Phi}$, where one of the actions $\s{H}_0 \curvearrowright M_0, \s{H}_\Phi \curvearrowright B_\Phi$ is of cohomogeneity one and the other one is transitive. Therefore, $\s{H}^\Lambda$ and $(\s{H}_0 \times \s{H}_\Phi)^\Lambda$ also have the same orbits (see \cite[Prop.\ 4.2]{BT:crelle}, or \cite[Th.\ 2.1(vii)]{DV:IMRN}). Note that the Lie algebra of $(\s{H}_0 \times \s{H}_\Phi)^\Lambda$ can be written as
	\begin{equation}\label{eq:CE_vs_decomposability}
		\g{h}_0 \oplus \left((\g{h}_\Phi \oplus \g{a}_\Phi) \loplus \g{n}_\Phi\right)= \g{h}_0 \oplus \g{h}_\Phi^\Lambda,
	\end{equation}
	where $\g{h}_\Phi^\Lambda = (\g{h}_\Phi \oplus \g{a}_\Phi) \loplus \g{n}_\Phi$ is precisely the Lie algebra of the Lie subgroup $\s{H}_\Phi^\Lambda$ of $\s{G}_-$ obtained by canonical extension from $\s{H}_\Phi$ (in the original sense of \cite{BT:crelle}). Observe also that the right-hand side of \eqref{eq:CE_vs_decomposability} is a direct sum of Lie algebras. It follows that the action of $(\s{H}_0 \times \s{H}_\Phi)^\Lambda$ on $M_0 \times M_-$ has the same orbits as the product action $\s{H}_0 \times \s{H}_\Phi^\Lambda$, so it decomposes.
			
	In order to prove the converse, suppose that the action of $\s{H}$ on $M_0 \times B_\Phi$ does not decompose. By Lemma~\ref{lemma:SplittingC1}, $\s{H}$ must act transitively on both $M_0$ and $B_\Phi$ (by means of the projection actions). Under the identification $T_{o_{M_-}} B_\Phi \cong \g{b}_\Phi$, this means that $\pr_{\g{r}^n}(\g{h}) = \g{r}^n$ and $\pr_{\g{b}_{\Phi}}(\g{h}) = \g{b}_\Phi$. Since $\g{h} \subseteq \g{h}^\Lambda$, we know that the action of $\s{H}^\Lambda$ on $M_0$ must be transitive as well. Since $\g{a}_\Phi \oplus \g{n}_\Phi \subseteq \g{h}^\Lambda$, we also have
	\[
	\pr_{\g{b}_\Phi \oplus \g{a}_\Phi \oplus \g{n}_\Phi}(\g{h}^\Lambda) = \g{b}_\Phi \oplus \g{a}_\Phi \oplus \g{n}_\Phi.
	\]
	Since $T_{o_{M_-}}M_- \cong \g{b}_\Phi \oplus \g{a}_\Phi \oplus \g{n}_\Phi$, the action of $\s{H}^\Lambda$ on $M_-$ is also transitive. By invoking Lemma~\ref{lemma:SplittingC1} once again, we see that the action of $\s{H}^\Lambda$ on $M_0 \times M_-$ does not decompose.
\end{proof}

The following is the main result of this subsection:

\begin{theorem}\label{thm:parab}
	Let $\s{H} \subset (\s{SO}_n \ltimes \R^n) \times \s{G}_-$ be a closed connected subgroup acting on $M_0 \times M_-$ with cohomogeneity one and a singular orbit, and suppose that its action does not decompose. Let $\alpha_j \in \Lambda$ be a simple root of $M_-$ and suppose that $\g{h}$ is contained in $(\g{so}_n \loplus \g{r}^n) \oplus \g{q}_j$, where $\g{q}_j$ is the maximal proper parabolic subalgebra of $\g{g}_-$ associated to $\Lambda \setminus \{\alpha_j\}$. Then the action of $\s{H}$ has the same orbits as the canonical extension of some proper cohomogeneity-one action $\s{H}_j \curvearrowright M_0 \times B_j$.
\end{theorem}

This theorem is largely an adaptation of \cite[Sect.\ 5]{BT:crelle}, so much so that the original proof mostly applies verbatim with only minor and obvious modifications. For this reason, we are not going to give a full proof: rather, we will formulate the key lemmas and intermediate results and comment on how and where the proof from \cite{BT:crelle} needs to be modified.

To begin with, note that we could have asked $\g{h}$ to be contained in $(\g{so}_n \loplus \g{r}^n) \oplus \g{q}$ for any maximal proper parabolic subalgebra $\g{q}$ of $\g{g}_-$ because any such $\g{q}$ is inner-conjugate to $\g{q}_j$ for some $\alpha_j \in \Lambda$. We may also assume that the singular orbit of $\s{H}$ passes through the base point $o = (0, o_{M_-})$ because $\s{Q}_j$ acts on $M_-$ transitively.

As already mentioned, the horospherical decomposition allows us to identify $M_0 \times M_- \cong M_0 \times B_j \times \s{A}_j \times \s{N}_j$ and thus $T_o(M_0 \times M_-) \cong \g{r}^n \oplus \g{b}_j \oplus \g{a}_j \oplus \g{n}_j$. If we decompose $\g{q}_j = \g{k}_j \oplus \g{b}_j \oplus \g{a}_j \oplus \g{n}_j$ and write $\tau \colon (\g{so}_n \loplus \g{r}^n) \oplus \g{q}_j \twoheadrightarrow \g{r}^n \oplus \g{b}_j \oplus \g{a}_j \oplus \g{n}_j$ for the natural projection, we have $T_o(\s{H} \cdot o) \cong \tau(\g{h})$. The slice representation of $\s{H}$ at $o$ takes the form
\[
\s{H} \cap (\s{SO}_n \times \s{K}_-) \to \s{O}(\nu_o(\s{H} \cdot o)), \quad k \mapsto \restr{\Ad(k)}{\nu_o(\s{H} \cdot o)},
\]
where we think of the normal space as a subspace of $\g{r}^n \oplus \g{b}_j \oplus \g{a}_j \oplus \g{n}_j$. We know that $\s{H} \cap (\s{SO}_n \times \s{K}_-) \subseteq \s{SO}_n \times \s{K}_j$ and the latter subgroup centralizes $\g{a}_j$. Since $\dim \nu_o(\s{H} \cdot o) \ge 2$ and the slice representation acts transitively on the unit sphere in $\nu_o(\s{H} \cdot o)$, we must have $\nu_o(\s{H} \cdot o) \subseteq \g{r}^n \oplus \g{b}_j \oplus \g{n}_j$ and thus $\g{a}_j \subseteq \tau(\g{h})$. Recall that we have a grading $\g{n}_j = \bigoplus_{k=1}^m \g{n}_j^k$. We define
\begin{align*}
    \g{w}_{0} &= (\g{r}^n \oplus \g{b}_j) \ominus (\tau(\g{h}) \cap (\g{r}^n \oplus \g{b}_j)), \\
    \g{w}_{k} &= \g{n}_j^k \ominus (\tau(\g{h}) \cap \g{n}_j^k), \; \text{for } k = 1, \ldots, m.
\end{align*}
By design, we have
\begin{equation}\label{eq:decomp_proof_canonical}
	\nu_o(H \cdot o) \subseteq \g{w}_{0} \oplus \g{w}_{1} \oplus \cdots \oplus \g{w}_{m}.
\end{equation}
The main distinction from \cite{BT:crelle} here is that the authors of that paper had no $M_0$ factor, so their definition of $\g{w}_0$ did not involve $\g{r}^n$. The following lemma can be proven in exactly the same way as in \cite{BT:crelle} by throwing in $\g{so}_n$ and $\s{SO}_n$ wherever necessary.
\begin{lemma}\label{lemmadecompwk}
	Let $X \in \nu_o(\s{H} \cdot o)$ be nonzero and write $\pi_l \colon \bigoplus_{k=0}^m \g{w}_k \rightarrow \g{w}_l$ for the natural projection. Then,
	\begin{enumerate}[\normalfont (a)]
		\item $\nu_o(\s{H} \cdot o) = \R X \oplus [\g{h} \cap (\g{so}_n \oplus \g{k}_j),\hspace{1pt} X]$.
		\item $\restr{\pi_k}{\nu_o(\s{H} \cdot o)} \colon \nu_o(\s{H} \cdot o) \to \g{w}_k$ is surjective and $\s{H} \cap (\s{SO}_n \times \s{K}_j)$-equivariant for each $k = 0, 1, \ldots, m$.
		\item If $\g{w}_k \ne 0$, then $\restr{\pi_k}{\nu_o(\s{H} \cdot o)} \colon \nu_o(\s{H} \cdot o) \to \g{w}_k$ is an isomorphism. In particular, $\g{w}_k$ is an irreducible $\s{H} \cap (\s{SO}_n \times \s{K}_j)$-module.
	\end{enumerate}
\end{lemma}
Recall that $H^j \in \g{a}_j \subseteq \tau(\g{h})$. This means that there exists $H \in \g{h}$ such that $\tau(H) = H^j$. We call such $H$ a \textit{lift} of $H^j$. We can always write $H = H_{\g{so}_n}^j + H_{\g{k}_j}^j + H^j$, where the first two summands are in $\g{so}_n$ and $\g{k}_j$, respectively. Notice that $\tau$ restricts to a linear isomorphism between $\g{h} \ominus [\g{h} \cap (\g{so}_n \oplus \g{k}_j)]$ and $\tau(\g{h})$. If a lift $H_{\g{so}_n}^j + H_{\g{k}_j}^j + H^j$ is contained in $\g{h} \ominus [\g{h} \cap (\g{so}_n \oplus \g{k}_j)]$, we call it a \textit{perfect lift}. Plainly, there exists a unique perfect lift of $H^j$. Again, due to the absence of the $M_0$ factor, lifts of $H^j$ looked simply like $H_{\g{k}_j}^j + H^j$ in \cite{BT:crelle}. It should be noted that only perfect lifts were considered in that paper. We will come back to this later (see the discussion of the proof of Lemma \ref{lemma:BT4.6}).
\begin{lemma}\label{lem:f_normalizing_w's}
    Let $H_{\g{so}_n}^j + H_{\g{k}_j}^j + H^j$ be the perfect lift of $H^j$, and denote $f = \ad(H_{\g{so}_n}^j + H_{\g{k}_j}^j)$. Then, for $k = 0, 1, \ldots, m$, we have:
    \begin{enumerate}[\normalfont (a)]
        \item $f$ normalizes $\g{w}_k$,
        \item $f^2 = -c_k^2 \Id$ on $\g{w}_k$ for some $c_k \in \R$.
	\end{enumerate}
\end{lemma}
The proof of this lemma also goes essentially identically to the one in \cite{BT:crelle}. For example, in (a), one reduces the problem to showing that $f$ preserves $\tau(\g{h}) \cap (\g{r}^n \oplus \g{b}_j)$ and $\tau(\g{h}) \cap \g{n}_j^k$ for each $k = 1, \ldots, m$. One takes a vector $X$ in either of these subspaces and picks $X_{\g{so}_n} \in \g{so}_n$ and $X_{\g{k}_j} \in \g{k}_j$ such that $X_{\g{so}_n} + X_{\g{k}_j} + X \in \g{h}$. One then brackets $X_{\g{so}_n} + X_{\g{k}_j} + X$ with $H_{\g{so}_n}^j + H_{\g{k}_j}^j + H^j$ and deduces that $f(X)$ must lie in $\tau(\g{h}) \cap (\g{r}^n \oplus \g{b}_j)$ or $\tau(\g{h}) \cap \g{n}_j^k$, respectively. The proof of (b) carries over with similar adjustments. By tracing the proof, one sees that the condition on the lift $H_{\g{so}_n}^j + H_{\g{k}_j}^j + H^j$ to be perfect is necessary in (b) but is not actually required in (a).

The following lemma sets strong restrictions on the normal space to the singular orbit $\nu_o(\s{H} \cdot o)$ and is the chief place where our line of argument differs substantially from \cite{BT:crelle}:
        
\begin{lemma}\label{lem:positioning_of_the_normal_space}
	We have $\nu_o(\s{H} \cdot o) \subseteq \g{r}^n \oplus \g{b}_j$.
\end{lemma}

In other words, this lemma asserts that $\g{n}_j \subseteq \tau(\g{h})$, or equivalently, $\g{w}_k = \{0\}$ for each $k = 1, \ldots, m$. The main difference with \cite{BT:crelle} is that there is a second possibility there: $\nu_o(\s{H} \cdot o) \subseteq \g{w}_1$, or equivalently, $\g{w}_k = \{0\}$ for $k = 0$ and $k > 1$. To see why this scenario does not arise in our case, let us first go through the proof. Surprisingly, most of the proof goes identically to the one in \cite{BT:crelle} with only minor modifications similar to those we did in the proof of Lemma \ref{lem:f_normalizing_w's}. One shows that either $\g{w}_1 = \{0\}$, in which case $\g{w}_k = \{0\}$ for each $k \ge 1$, or else $\g{w}_1 \ne \{0\}$, in which case $\g{w}_k = \{0\}$ for $k \ne 1$. To rule out the second option in our case, note that it is equivalent to $\nu_o(\s{H} \cdot o) \subseteq \g{w}_1$. If this was the case, we would have $T_o(\s{H} \cdot o) \cong \tau(\g{h}) = \g{r}^n \oplus \g{b}_j \oplus \g{a}_j \oplus (\g{n}_j \ominus \nu_o(\s{H} \cdot o))$. That would mean that the projection of $\tau(\g{h})$ to $T_{o_{M_-}}(M_-)$ is not surjective and thus the projection action of $\s{H}$ on $M_-$ is not transitive. By Lemma \ref{lemma:SplittingC1}, this would imply that the action $\s{H} \curvearrowright M_0 \times M_-$ decomposes, leading to a contradiction. This second possibility for $\nu_o(\s{H} \cdot o)$ can only be ruled out here because we assume that the action does not decompose. In the original proof in \cite{BT:crelle}, this assumption did not make sense (due to the lack of the $M_0$ factor). By tracing the proof in \cite{BT:crelle}, we see that the case $\nu_o(\s{H} \cdot o) \subseteq \g{b}_j$ led to the canonical extension, while $\nu_o(\s{H} \cdot o) \subseteq \g{w}_1$ led to the nilpotent construction. In our case, the second scenario does not arise. Instead, we only have $\nu_o(\s{H} \cdot o) \subseteq \g{r}^n \oplus \g{b}_j$, which is going to lead to (our version of) the canonical extension.

The next lemma sheds light on how the subalgebra $\g{h}$ interacts with the decomposition ${(\g{so}_n \loplus \g{r}^n)} \oplus ((\g{m}_j \oplus \g{a}_j) \loplus \g{n}_j)$: 
		
\begin{lemma}\label{lemma:BT4.6}
	Let $H^j + H_{\g{so}_n}^j + H_{\g{k}_j}^j$ be any lift of $H^j$. We have:
	\begin{enumerate}[\normalfont (a)]
		\item $\g{h} \cap \g{n}_j = \tau(\g{h}) \cap \g{n}_j$.
		\item $\g{h} = [\g{h} \cap ((\g{so}_n \loplus \g{r}^n) \oplus \g{m}_j)] \oplus \R(H^j + H_{\g{so}_n}^j + H_{\g{k}_j}^j) \oplus [\g{h} \cap \g{n}_j]$.
	\end{enumerate}
\end{lemma}

In \cite{BT:crelle}, part (b) of this lemma reads as $\g{h} = [\g{h} \cap \g{m}_j] \oplus \R(H^j + H_{\g{k}_j}^j) \oplus [\g{h} \cap \g{n}_j]$. The proof here is essentially the same as in \cite{BT:crelle}---with the same minor modifications as before. There is a small issue in the proof given in \cite{BT:crelle} that we would like to address and correct here. The proof starts by finding a connected solvable Lie subgroup $\s{S}$ of $\s{H}$ such that $\s{S} \cdot o = \s{H} \cdot o$. Then, the authors claim that the perfect lift of $H^j$ lies in the Lie algebra $\g{s}$ of $\s{S}$, which may not actually be the case (especially due to the freedom in the choice of $\s{S}$). Fortunately, the property of the lift being perfect is not actually needed in the proof, and one can always find at least some lift of $H^j$ in $\g{s}$: indeed, the condition $\s{S} \cdot o = \s{H} \cdot o$ implies $T_o(\s{S} \cdot o) = T_o(\s{H} \cdot o)$, which translates to $\tau(\g{s}) = \tau(\g{h})$. Since the latter contains $H^j$, we can find a lift of $H^j$ that lies not only in $\g{h}$ but in $\g{s}$. The rest of the proof goes unchanged. Finally, note that if part (b) of the lemma is true for some  lift of $H^j$, then it is true for any lift.

We can immediately strengthen Lemma \ref{lemma:BT4.6} by combining it with Lemma \ref{lem:positioning_of_the_normal_space}:

\begin{corollary}\label{lemma:BT4.6_beefed_up}
    Let $H^j + H_{\g{so}_n}^j + H_{\g{k}_j}^j$ be any lift of $H^j$. We have:
    \[
    \g{h} = [\g{h} \cap ((\g{so}_n \loplus \g{r}^n) \oplus \g{m}_j)] \oplus \R(H^j + H_{\g{so}_n}^j + H_{\g{k}_j}^j) \oplus \g{n}_j.
	\]
\end{corollary}

Note that this corollary is specific to our situation and its analog does note have to hold in general in \cite{BT:crelle} because they may not have $\g{n}_j \subseteq \tau(\g{h})$. The last lemma rounds up the argument and completes the proof of Theorem \ref{thm:parab}:
		
\begin{lemma}
	The action of $\s{H}$ has the same orbits as the canonical extension of the cohomo\-geneity-one action $\s{H}_j \curvearrowright M_0 \times B_j$, where $\s{H}_j$ is the connected Lie subgroup of $(\s{SO}_n \ltimes \R^n) \times \s{M}_j$ with Lie algebra  $\g{h}_j = \g{h} \cap ((\g{so}_n \loplus \g{r}^n) \oplus \g{m}_j)$.
\end{lemma}

The proof of this lemma carries over almost verbatim from the proofs of Lemma 5.5 and Proposition 5.6 in \cite{BT:crelle}---again, with the same standard modification as before.

\subsection{The proof of Theorem \ref{th:singorbdecomp}}\label{subsection:singorb:singorb}
We now prove that a cohomogeneity-one action $\s{H} \curvearrowright M_0 \times M_-$ that has a singular orbit necessarily decomposes. Our approach is going to be as follows. First, by employing Theorem \ref{thm:parab} and replacing $M_-$ with a suitable boundary component of $M_-$ if necessary, we argue that we may assume without loss of generality that $\s{H}$ projects surjectively onto $\s{G}_-$. After that, we construct a chain of connected Lie subgroups
\[
\cdots \subsetneq \s{E}_{i+1} \subsetneq \s{E}_i \subsetneq \cdots \subsetneq \s{E}_1 \subsetneq \s{SO}_n \ltimes \R^n
\]
such that $\s{E}_i\curvearrowright M_0$ is transitive and $\s{H} \subsetneq \s{E}_i \times \s{G}_-$ for every $i$. We then argue that this chain can be extended indefinitely, which leads to a contradiction for dimension reasons. We begin with the following technical lemma: 

\begin{lemma}\label{lemma:splitting1} 
    Let $\s{E}_i \subseteq \s{SO}_n \ltimes \R^n $ be any connected Lie subgroup acting on $M_0$ transitively, and let $\s{E}$ be a maximal proper connected Lie subgroup of $\s{E}_{i} \times \s{G}_-$. Suppose that the projection actions of $\s{E}$ on $M_0$ and $M_-$ are both transitive.
    Then, either $\s{E} = \s{E}_{i} \times \s{Q}^0$ for some maximal proper parabolic subgroup $\s{Q} \subsetneq \s{G}_-$, or else $\s{E} = \s{E}_{i+1} \times \s{G}_-$ for some maximal proper connected Lie subgroup $\s{E}_{i+1}$ of $\s{E}_{i}$ whose action on $M_0$ is transitive. In either case, $\s{E}$ acts on $M_0 \times M_-$ transitively.
\end{lemma}

\begin{proof}
	Consider the natural projections $\pr_{\s{E}_i} \colon \s{E}_i \times \s{G}_- \to \s{E}_i$ and $\pr_{\s{G}_-} \colon \s{E}_i \times \s{G}_- \to \s{G}_-$. If $\pr_{\s{E}_i}(\s{E}) \subsetneq \s{E}_i$, by the maximality of $\s{E}$, we must have that $\s{E}_{i+1} := \pr_{\s{E}_i}(\s{E})$ is a maximal proper connected Lie subgroup of $\s{E}_i$ and $\s{E} = \s{E}_{i+1} \times \s{G}_-$. Since by hypothesis the projection action of $\s{E}$ on $M_0$ is transitive, so must be the action of $\s{E}_{i+1}$.
			
	Similarly, if $\pr_{\s{G}_-}(\s{E}) \subsetneq \s{G}_-$, then $\pr_{\s{G}_-}(\s{E})$ is a maximal proper connected Lie subgroup of $\s{G}_-$ and $\s{E} = \s{E}_i \times \pr_{\s{G}_-}(\s{E})$. Such a subgroup of $\s{G}_-$ must be either reductive\footnote{This notion of reductive subgroup is stronger than just being reductive as a Lie group (see the discussion on page 6 in \cite{kollross_duality}).} or the identity component of a parabolic subgroup (see \cite{Mostow:AnnalsMaximal}). If $\pr_{\s{G}_-}(\s{E})$ is reductive, it cannot act on $M_-$ transitively (see \cite[Prop.\ 3.1]{BT:crelle}), implying that the projection action of $\s{E}$ on $M_-$ is not transitive. Thus, $\pr_{\s{G}_-}(\s{E})$ must be the identity component of some maximal proper parabolic subgroup $\s{Q}$ of $\s{G}_-$.

	From now on, we assume that $\pr_{\s{E}_i}$ and $\pr_{\s{G}_-}$ restrict to surjective homomorphisms from $\s{E}$ onto $\s{E}_i$ and $\s{G}_-$, respectively. Let $\pr_{\g{e}_i} \colon \g{e} \to \g{e}_i$ and $\pr_{\g{g}_-} \colon \g{e} \to \g{g}_-$ be the corresponding surjective homomorphisms at the level of Lie algebras. Observe that $\g{e} \cap \g{g}_-$ is an ideal of $\g{g}_-$: for any $X_{\g{g}_-} \in \g{g}_-$ and $Y \in \g{e} \cap \g{g}_-$, we can find some $X \in \g{e}$ such that $\pr_{\g{g}_-}(X) = X_{\g{g}_-}$, and so
    \[
	[X_{\g{g}_-}, Y] = [X,Y] \in \g{e} \cap \g{g}_-.
	\]
	Since $\g{g}_-$ is semisimple, we can split it into a direct sum $\g{g}_- = \g{g}_1 \oplus \cdots \oplus \g{g}_k$ of simple noncompact ideals. Without loss of generality, we may assume that
	\[
	\g{e} \cap \g{g}_- = \g{g}_1 \oplus \cdots \oplus \g{g}_s
	\]
	for some $s \in \{0, \ldots, k-1 \}$ (where by $s = 0$ we mean that $\g{e} \cap \g{g}_- = \{0\}$). Note that it is not possible to have $s = k$ because that would imply $\g{g}_- \subseteq \g{e}$ and, since $\pr_{\g{e}_i}(\g{e}) = \g{e}_i$, also $\g{e_i} \subseteq \g{e}$. This would contradict the fact that $\g{s}$ is a proper subalgebra of $\g{e}_i \oplus \g{g}_-$.

	The projection $\g{e}_i \oplus \g{g}_- \to \g{e}_i \oplus \g{g}_1 \oplus \cdots \oplus \g{g}_s$ along $\g{g}_{s+1} \oplus \cdots \oplus \g{g}_k$ (which is a Lie algebra homomorphism) restricts to an isomorphism
	\[
	\g{e} \xrightarrow{\Psi} \g{e}_i \oplus \g{g}_1 \oplus \cdots \oplus \g{g}_s,
	\] 
	since $\g{e}$ projects surjectively onto $\g{e}_i$ and the kernel of $\Psi$ is $\g{e} \cap (\g{g}_{s+1} \oplus \cdots \oplus \g{g}_{k}) = \{0\}$.

    It follows from our discussion in Subsection~\ref{subsection:singorb:transitivern} that we can write $\g{e}_i = ((\g{e}_i \cap \g{so}_n) \oplus \g{v}^\varphi) \loplus \g{w}$, where $\g{w} = \g{e}_i \cap \g{r}^n$, $\g{e}_i \cap \g{so}_n \subseteq \g{so}(\g{w})$, $\g{v} = \g{r}^n \ominus \g{w}$, and $\g{v}^\varphi \subseteq \g{so}(\g{w}) \oplus \g{v}$. Note that $\Psi^{-1}(\g{w})$ is an abelian ideal of $\g{e}$. Since the restriction of the projection $\Phi \colon \g{e}_i \oplus \g{g}_- \to \g{g}_{s+1} \oplus \cdots \oplus \g{g}_k$ to $\g{e}$ is a surjective homomorphism, $\Phi(\Psi^{-1}(\g{w}))$ is an abelian ideal of $\g{g}_{s+1} \oplus \cdots \oplus \g{g}_k$, hence it must be trivial. It follows that $\Psi^{-1}(\g{w}) = \g{w}$, or in other words, $\g{w} \subseteq \g{e}$. Consequently, the restriction of $\Phi$ to $\Psi^{-1}((\g{e}_i \cap \g{so}_n) \oplus \g{v}^\varphi)$ must be surjective. But the latter is a compact Lie algebra, whereas $\g{g}_{s+1} \oplus \cdots \oplus \g{g}_k$ is noncompact semisimple, which leads to a contradiction.
\end{proof}

We can now proceed to prove Theorem~\ref{th:singorbdecomp}:

\begin{proof}[Proof of Theorem~\ref{th:singorbdecomp}]
    Let $\s{H}$ be a closed connected subgroup of $(\s{SO}_n \ltimes \R^n) \times \s{G}_-$ whose action on $M_0 \times M_-$ is of cohomogeneity one and has a singular orbit. We assume that the action does not decompose and are seeking a contradiction. By Lemma \ref{lemma:SplittingC1}, this is equivalent to saying that the projection actions of $\s{H}$ on both $M_0$ and $M_-$ are transitive.
            
    As mentioned in the proof of Lemma \ref{lemma:splitting1}, there are three possibilities for the projection $\pr_{\s{G}_-}(\s{H})$: it can be the whole $\s{G}_-$, or it can be contained in a maximal proper reductive subgroup of $\s{G}_-$, or else it can be contained in a maximal proper parabolic subgroup of $\s{G}_-$. Since the projection action $\s{H} \curvearrowright M_-$ is transitive, the second option can once again be ruled out. Suppose that $\pr_{\s{G}_-}(\s{H})$ is contained in some maximal proper parabolic subgroup $\s{Q}$ of $\s{G}_-$. After conjugating by an element of $\s{G}_-$, we may assume that $\s{Q} = \s{Q}_j$ is the parabolic subgroup associated with the subset of simple roots $\Phi_j = \Lambda \setminus \{ \alpha_j \}$ for some $\alpha_j \in \Lambda$. Thanks to Theorem \ref{thm:parab} and Remark \ref{rem:CE_proper}, $\s{H}$ then has the same orbits as the subgroup $\s{H}_j^\Lambda$ obtained by canonical extension of some cohomogeneity-one action $\s{H}_j \curvearrowright M_0 \times B_j$, where $\s{H}_j$ is a closed connected subgroup of $(\s{SO}_n \ltimes \R^n) \times I^0(B_j)$. According to Lemma~\ref{lemma:decompcanonical}, the action of $\s{H}_j$ on $M_0 \times B_j$ does not decompose---and it still has a singular orbit. In case the projection of $\s{H}_j$ in $I^0(B_j)$ is not the whole $I^0(B_j)$ and is thus contained in a maximal proper parabolic subgroup of $I^0(B_j)$, we can run the same argument again. After several iterations, we are bound to arrive at the following situation: the action $\s{H} \curvearrowright M_0 \times M_-$ is orbit-equivalent to the canonical extension of a cohomogeneity-one action $\s{H}_\Phi \curvearrowright M_0 \times B_\Phi$, where:
    
    \begin{enumerate}{\setlength\leftmargin{2em}}
		\item[(a)] $\Phi \subseteq \Lambda$ is some subset of simple roots,
		\item[(b)] $\s{H}_\Phi$ is a closed connected subgroup of $(\s{SO}_n \ltimes \R^n) \times I^0(B_\Phi)$,
		\item[(c)] $\s{H}_\Phi \curvearrowright M_0 \times B_\Phi$ does not decompose.
		\item[(d)] The projection of $\s{H}_\Phi$ to $I^0(B_\Phi)$ is surjective.
	\end{enumerate}      
	
    Note that $\Phi$ cannot be empty, for otherwise the canonical extension $\s{H}_\Phi^\Lambda \curvearrowright M_0 \times M_-$ would decompose. Since $\s{H}_{\Phi}$ is a proper connected Lie subgroup of $(\s{SO}_n \ltimes \R^n) \times I^0(B_\Phi)$, it must be contained in a maximal proper connected Lie subgroup $\s{E} \subsetneq (\s{SO}_n \ltimes \R^n) \times I^0(B_\Phi)$. The projection actions of $\s{H}_\Phi$ on $M_0$ and $B_\Phi$ are transitive, so the same is true for $\s{E}$. Moreover, the projection of $\s{E}$ to $I^0(B_\Phi)$ has to be surjective in view of item (d) above. According to Lemma \ref{lemma:splitting1}, we must then have $\s{E} = \s{E}_1 \times I^0(B_\Phi)$ for some maximal proper connected Lie subgroup $\s{E}_1$ of $\s{SO}_n \ltimes \R^n$ whose action on $M_0$ is transitive. Notice that the action of $\s{E}_1 \times I^0(B_\Phi)$ on $M_0 \times M_-$ is still transitive, so $\s{H}_\Phi$ must be contained in some maximal proper connected Lie subgroup $\s{E}'$ of $\s{E}_1 \times I^0(B_\Phi)$. Arguing in the same vein, we see that $\s{E}' = \s{E}_2 \times I^0(B_\Phi)$, where $\s{E}_2 \subsetneq \s{E}_1$ is a maximal proper connected Lie subgroup that acts on $M_0$ transitively. By iterating this process, we can construct a chain of maximal proper connected Lie subgroups
	\[
	\dots \subsetneq \s{E}_{i+1} \subsetneq \s{E}_{i} \subsetneq \dots \subsetneq \s{E}_1 \subsetneq \s{SO}_n \ltimes \R^n
	\]
	such that $\s{E}_i$ acts on $M_0$ transitively and $\s{H}_\Phi \subsetneq \s{E}_i \times I^0(B_\Phi)$ for all $i$. On the one hand, by dimension reasons, we can only have a finite number of such iterations. On the other hand, as explained above, we should always be able to make one more step. This yields a contradiction and completes the proof of the theorem.
\end{proof}

As explained in the introduction, when combined together, Theorems \ref{th:main:decompcompactxhadamard}, \ref{theorem:foliations}, and \ref{th:singorbdecomp} imply Theorem \ref{mainth}.

%% file: 6-The_non-simply_connected_case.tex
In this section, we look at cohomogeneity-one actions on non-simply connected symmetric spaces. If $M$ is such a space, its universal cover $\widetilde{M}$ is also symmetric. Theorem \ref{mainth} asserts that every cohomogeneity-one action on $\widetilde{M} = M_+ \times M_0 \times M_-$ either decomposes or else is orbit-equivalent to the action of $I^0(M_+) \times \s{H}_{E,X}$. In each case, we study how the action behaves with respect to the covering $\pi \colon \widetilde{M} \to M$, whether it descends to a proper cohomogeneity-one action on $M$, and what happens to orbit equivalence. The section culminates with Theorem~\ref{th:nonsimplyconnected}, which gives a description of all cohomogeneity-one actions on $M$ via those on $\widetilde{M}$ and can be regarded as the extension of Theorem \ref{mainth} to the non-simply connected case.

Let $M$ be a connected Riemannian manifold and $\pi \colon \widetilde{M} \rightarrow M$ its universal Riemannian covering. We would like to investigate how (orbit equivalence classes of) isometric actions on $M$ are related to those on $\widetilde{M}$. Write $\Aut(\pi)$ for the group of deck transformations of $\pi$. This is a discrete subgroup of $I(\widetilde{M})$ that acts on $\widetilde{M}$ freely and properly discontinuously and acts simply transitively on each fiber of $\pi$. Moreover, each of its elements is fully determined by its value at any single point of $\widetilde{M}$. We also have $\Aut(\pi) \cong \pi_1(M)$. We say that an isometry $\widetilde{f} \in I(\widetilde{M})$ \textit{descends} to $M$ if it permutes the fibers of $\pi$---and thus uniquely determines a \textit{descent} $f\in I(M)$ such that $\pi \circ \widetilde{f} = f \circ \pi$; we will also say that $\widetilde{f}$ is a \textit{lift} of $f$. We write $I(\widetilde{M},\pi)$ for the subgroup of $I(\widetilde{M})$ consisting of elements that descend to $M$. It is straightforward to check that $I(\widetilde{M},\pi)$ coincides with the normalizer of $\Aut(\pi)$ in $I(\widetilde{M})$ and is thus a closed subgroup of $I(\widetilde{M})$. By design, we have a homomorphism $\Pi \colon I(\widetilde{M},\pi) \to I(M)$ sending $\widetilde{f}$ to $f$, and its kernel is precisely $\Aut(\pi)$. By the lifting criterion, any isometry of $M$ admits lifts to $\widetilde{M}$, so $\Pi$ is a surjective local isomorphism.

Let $\s{H}$ be a connected Lie group acting on $M$ isometrically. By quotienting out the ineffectiveness kernel, we may assume that $\s{H}$ is a Lie subgroup of $I^0(M)$. We define the \textit{lift} of $\s{H}$ (resp., of $\s{H} \curvearrowright M$) to be the subgroup $\widetilde{\s{H}} := \Pi^{-1}(\s{H})^0 \subseteq I^0(\widetilde{M},\pi)$ (resp., the action $\widetilde{\s{H}} \curvearrowright \widetilde{M}$). Conversely, suppose $\widetilde{\s{H}}$ is a connected Lie group acting on $\widetilde{M}$ isometrically such that its image in $I(\widetilde{M)}$ is contained in $I(\widetilde{M},\pi)$ (without loss of generality, we assume $\widetilde{\s{H}} \subseteq I^0(\widetilde{M},\pi)$). Note that $\s{H} := \Pi(\widetilde{\s{H}})$ is a connected Lie subgroup of $I(M)$ of the same dimensions as $\widetilde{\s{H}}$. As $\widetilde{\s{H}}$ is a connected Lie subgroup of $\Pi^{-1}(\s{H})$, it must be its identity component, which is exactly the lift of $\s{H}$. We call $\s{H}$ (resp., $\s{H} \curvearrowright M$) the \textit{descent} of $\widetilde{\s{H}}$ (resp., of $\widetilde{\s{H}} \curvearrowright \widetilde{M}$). We see that a connected Lie subgroup of $I(\widetilde{M})$ is a lift if and only if it is contained in $I(\widetilde{M},\pi)$ (we will also say that its action on $\widetilde{M}$ \textit{descends} to $M$ in this case).

Let $\s{H}$ be a connected Lie subgroup of $I(M)$ and $\widetilde{\s{H}}$ its lift. Notice that $\Pi$ restricts to a surjective local isomorphism from $\widetilde{\s{H}}$ to $\s{H}$. For any $\widetilde{p} \in \widetilde{M}$, we have $\pi(\widetilde{\s{H}} \cdot \widetilde{p}) = \s{H} \cdot p$, where $p = \pi(\widetilde{p})$, so $\pi$ maps the orbits of $\widetilde{\s{H}}$ in $\widetilde{M}$ onto the orbits of $\s{H}$ in $M$. Moreover, one readily sees that the isotropy subgroup $\widetilde{\s{H}}_{\widetilde{p}}$ is an open subgroup of $\Pi^{-1}(\s{H}_p) \cap \widetilde{\s{H}}$, so
\[
\widetilde{\s{H}}/\widetilde{\s{H}}_{\widetilde{p}} \cong \widetilde{\s{H}} \cdot \widetilde{p} \xrightarrow{\hspace{0.7em} \pi \hspace{0.7em}} \s{H} \cdot p \cong \s{H}/\s{H}_p \cong \widetilde{\s{H}}/(\Pi^{-1}(\s{H}_p) \cap \widetilde{\s{H}})
\]
is a covering map. We see that the lift procedure preserves the cohomogeneity of an action and induces a surjective mapping between the corresponding orbit spaces: $\widetilde{M}/\widetilde{\s{H}} \to M/\s{H}$.

The following proposition outlines some further properties of the lift procedure and gives a handy necessary criterion for an action on $\widetilde{M}$ to be a lift. All actions in the proposition are assumed to be by connected Lie groups by default.

\begin{proposition}\label{prop:lifting_actions}
Let $M$ be a connected Riemannian manifold and $\pi \colon \widetilde{M} \rightarrow M$ its universal Riemannian covering.
\begin{enumerate}[\normalfont (a)]
    \item If two isometric actions on $M$ are orbit-equivalent, then so are their lifts. Conversely, if two isometric actions on $\widetilde{M}$ descend to $M$ and are orbit-equivalent via an element of $I(\widetilde{M},\pi)$, then their descents are orbit-equivalent as well.
    \item If an isometric action in $\widetilde{M}$ descends to $M$, then $\Aut(\pi) \cap I^0(\widetilde{M},\pi)$ permutes its orbits.
    \item The lift of a proper isometric action on $M$ is proper. Conversely, if $\widetilde{\s{H}}$ is a closed connected subgroup of $I(\widetilde{M},\pi)$, then its descent determines a proper action on $M$ if and only if there exists a neighborhood $U$ of the identity in $I(\widetilde{M},\pi)$ (or $I(\widetilde{M})$) such that for every $g \in \Aut(\pi)$, either $g \in \widetilde{\s{H}}$ or $\widetilde{\s{H}} \cap gU = \varnothing$.
\end{enumerate}
\end{proposition}

\begin{remark}
It might happen that two isometric actions on $\widetilde{M}$ are orbit-equivalent and one of them descends to $M$ but the other does not.
\end{remark}

\begin{proof}
    We start with (a). Suppose two actions $\s{H}_1 \curvearrowright M$ and $\s{H}_2 \curvearrowright M$ are orbit-equivalent via $g \in I(M)$, i.e., $g(\s{H}_1 \cdot p) = \s{H}_2 \cdot g(p)$ for all $p \in M$. Take any lift $\widetilde{g}$ of $g$ and fix $\widetilde{p} \in \widetilde{M}$ lying over $p \in M$. We would like to show that $\widetilde{g}(\widetilde{\s{H}}_1 \cdot \widetilde{p}) = \widetilde{\s{H}}_2 \cdot \widetilde{g}(\widetilde{p})$. We observe:
    \begin{equation}\label{lifting_orbit_equivalence}
    \pi(\widetilde{g}(\widetilde{\s{H}}_1 \cdot \widetilde{p})) = g(\pi(\widetilde{\s{H}}_1 \cdot \widetilde{p})) = g(\s{H}_1 \cdot p) = \s{H}_2 \cdot g(p) = \pi(\widetilde{\s{H}}_2 \cdot \widetilde{g}(\widetilde{p})),
    \end{equation}
    which means that the restrictions of $\pi$ to both submanifolds in question are covering maps over $g(\s{H}_1 \cdot p) = \s{H}_2 \cdot g(p)$. Since the two submanifolds share a common point $\widetilde{g}(\widetilde{p})$, the uniqueness property of lifting paths ensures that they coincide. Conversely, if we have two connected Lie subgroups $\widetilde{\s{H}}_1, \widetilde{\s{H}}_2 \subseteq I^0(\widetilde{M},\pi)$ and their actions on $\widetilde{M}$ are orbit-equivalent via $\widetilde{g} \in I(\widetilde{M},\pi)$, a calculation similar to \eqref{lifting_orbit_equivalence} shows that their descents are orbit-equivalent via $\Pi(\widetilde{g})$.

    For (b), first note that $\Aut(\pi) \cap I^0(\widetilde{M},\pi)$ is the kernel of $\restr{\Pi}{I^0(\widetilde{M},\pi)}$, so it is a discrete normal subgroup of $I^0(\widetilde{M},\pi)$. Since the latter is connected, $\Aut(\pi) \cap I^0(\widetilde{M},\pi)$ is actually a central subgroup of it. If $\widetilde{\s{H}} \subseteq I^0(\widetilde{M},\pi)$ is a connected Lie subgroup and $\widetilde{g} \in \Aut(\pi) \cap I^0(\widetilde{M},\pi)$, we have $\widetilde{g}(\widetilde{\s{H}} \cdot \widetilde{p}) = \widetilde{\s{H}} \cdot \widetilde{g}(\widetilde{p})$ for any $\widetilde{p} \in \widetilde{M}$, which was to be shown.
    
    Finally, we prove (c). If $\s{H} \curvearrowright M$ is proper, then (the image of) $\s{H}$ is closed in $I(M)$ and hence so is its lift $\widetilde{\s{H}}$ in $I(\widetilde{M})$. This means that the action of $\widetilde{\s{H}}$ on $\widetilde{M}$ is proper. Let now $\widetilde{\s{H}}$ be a closed connected subgroup of $I(\widetilde{M},\pi)$. First of all, whether the neighborhood $U$ in (c) is chosen in $I(\widetilde{M},\pi)$ or $I(\widetilde{M})$ is irrelevant because $I(\widetilde{M},\pi)$ is closed in $I(\widetilde{M})$. Second, notice that as soon as $U$ is a connected neighborhood of the identity in $I(\widetilde{M},\pi)$, the intersection $\widetilde{\s{H}} \cap gU$ is empty for any $g \in \Aut(\pi) \setminus I^0(\widetilde{M},\pi)$. The statement then follows from a more general Lie-theoretical fact that we prove separately in Lemma \ref{lem:covering_closed_subgroup} below.
\end{proof}

\begin{lemma}\label{lem:covering_closed_subgroup}
    Let $\Pi \colon \widetilde{\s{G}} \to \s{G}$ be a local isomorphism between connected Lie groups with kernel $\widetilde{\s{Z}}$, and let $\widetilde{\s{H}} \subseteq \widetilde{\s{G}}$ be a closed subgroup. Then the image $\s{H} = \Pi(\widetilde{\s{H}})$ is a closed subgroup of $\s{G}$ if and only if the following condition is satisfied: there exists a neighborhood $U$ of the identity in $\widetilde{\s{G}}$ such that for every $g \in \widetilde{\s{Z}}$, either $g \in \widetilde{\s{H}}$ or $\widetilde{\s{H}} \cap gU = \varnothing$.
\end{lemma}

\begin{proof}
    Notice that $\Pi$ is a covering map and $\widetilde{\s{Z}}$ is a discrete cental subgroup of $\widetilde{\s{G}}$ that can be identified with the group of deck transformations. First, assume that the aforementioned condition is satisfied. We would like to prove that $\s{H}$ is closed in $\s{G}$. It suffices to show that it is locally closed (see \cite[Ch.\ III, \S 2, Prop.\ 4]{bourbaki_general_topology_1-4}), i.e., that there exists a neighborhood $V$ of the identity in $\s{G}$ such that $\s{H} \cap V$ is closed in $V$. Let $U$ be a neighborhood of the identity in $\widetilde{\s{G}}$ as in the statement of the lemma. By shrinking it if necessary, we may assume that $\Pi$ restricts to a diffeomorphism between $U$ and $V = \Pi(U)$ and that the latter is evenly covered. The connected components of $\Pi^{-1}(V)$ are given by $gU$ as $g$ runs through $\widetilde{\s{Z}}$. If $g \in \widetilde{\s{H}}$, then the restriction $L_g \colon U \isoto gU$ sends $\widetilde{\s{H}} \cap U$ diffeomorphically onto $\widetilde{\s{H}} \cap gU$. Since $g$ is a deck transformation, we see that $\Pi(\widetilde{\s{H}} \cap U) = \Pi(\widetilde{\s{H}} \cap gU)$ inside $V$. We then compute:
    \[
    \s{H} \cap V = \; \smashoperator{\bigcup_{g \hspace{1pt} \in \hspace{1pt} \widetilde{\s{Z}}}} \, \Pi(\widetilde{\s{H}} \cap gU) \, = \; \Big( \hspace{0.6em} \smashoperator{\bigcup_{g \hspace{1pt} \in \hspace{1pt} \widetilde{\s{Z}} \hspace{1pt} \cap \hspace{1pt} \widetilde{\s{H}}}} \, \Pi(\widetilde{\s{H}} \cap gU) \Big) \; \cup \; \Big( \hspace{0.6em} \smashoperator{\bigcup_{g \hspace{1pt} \in \hspace{1pt} \widetilde{\s{Z}} \setminus \widetilde{\s{H}}}} \, \Pi(\widetilde{\s{H}} \cap gU) \Big) \, = \Pi(\widetilde{\s{H}} \cap U).
    \]
    Here the sets in the first union all coincide, whereas the second union is empty by assumption. Since $\widetilde{\s{H}} \cap U$ is closed in $U$, its image under $\Pi$ is closed in $V$, which was to be shown.

    Conversely, suppose that $\s{H}$ is closed and thus embedded. In that case, we can find an evenly covered neighborhood $V$ of the identity in $\s{G}$ whose intersection with $\s{H}$ is connected. Let $U$ be the connected component of $\Pi^{-1}(V)$ containing the identity; the other connected components are then given by $gU$ as $g$ runs through $\widetilde{\s{Z}}$. As we argued above, if $g \in \widetilde{\s{Z}} \cap \widetilde{\s{H}}$, then the image of $\widetilde{\s{H}} \cap gU$ under $\Pi$ coincides with that of $\widetilde{\s{H}} \cap U$. That image is an open subset of $\s{H} \cap V$. Suppose now that $g \in \widetilde{\s{Z}} \setminus \widetilde{\s{H}}$. Observe that $g(\widetilde{\s{H}} \cap U)$ has to be disjoint from $\widetilde{\s{H}} \cap gU$. Since $g$ is a deck transformation, this is equivalent to saying that $\Pi(\widetilde{\s{H}} \cap U)$ has to be disjoint from $\Pi(\widetilde{\s{H}} \cap gU)$. Note that the latter is also an open subset of $\s{H} \cap V$. We can thus form an open subset $\bigcup_{g \in \widetilde{\s{Z}} \setminus \widetilde{\s{H}}} \Pi(\widetilde{\s{H}} \cap gU)$ of $\s{H} \cap V$. It is still disjoint from $\Pi(\widetilde{\s{H}} \cap U) = \bigcup_{g \in \widetilde{\s{Z}} \cap \widetilde{\s{H}}} \Pi(\widetilde{\s{H}} \cap gU)$, yet its union with the latter has to be the whole $\s{H} \cap V$. By connectedness of $\s{H} \cap V$, we deduce that $\widetilde{\s{H}}$ intersects $gU$ trivially whenever $g \in \widetilde{\s{Z}} \setminus \widetilde{\s{H}}$. This completes the proof.
\end{proof}

The lifting procedure allows one to shift the study of isometric actions on a Riemannian manifold to that on its universal cover. In order to properly understand which actions on $\widetilde{M}$ descend to $M$, one has to understand what the group of deck transformations looks like and how it is embedded into $I(\widetilde{M})$. In general, this can be complicated, but for symmetric spaces the picture is rather simple.

Suppose that $M$ is a symmetric space, and let $\widetilde{M} = M_+ \times M_0 \times M_-$ be the decomposition of its universal cover into its compact, Euclidean, and noncompact factors. Since $M$ is symmetric, $\Aut(\pi) \cong \pi_1(M)$ is abelian. In fact, \cite[Th.\ 8.3.12]{wolf_constant_curvature} provides a complete description of $\Aut(\pi)$ as a subgroup of $I(\widetilde{M})$. Let us write $\Delta$ for the centralizer of $I^0(M_+)$ in $I(M_+)$, which is a finite abelian subgroup. As before, we identify $M_0 \cong\R^n$ and thus have $I(M_0) = \s{O}_n \ltimes \R^n$.
Then, $\Aut(\pi)$ is a discrete subgroup\footnote{Conversely, for every discrete subgroup $\Gamma \subseteq \Delta \times \R^n$, $\widetilde{M}/\Gamma$ is symmetric.} of $\Delta \times \R^n$. This implies that the projection of $\Aut(\pi)$ in $\R^n$ is a lattice (possibly not of full rank), which we denote by $L$.
We write $V = \mathrm{span}(L) \simeq \R^k$ and $V^\perp = \R^n \ominus V \simeq \R^{n-k}$. We also denote by $L^* \subseteq V$ the dual lattice of $L$, i.e., $L^* = \{ v \in V \mid \bilin{v}{w} \in \Z \;\, \text{for every} \; w \in L \}$. Since $\Delta$ is finite, the intersection $L' = \Aut(\pi) \cap L$ is a sub-lattice of finite index in $L$. Moreover, $\Aut(\pi) \subseteq \Delta \times L \subseteq \Delta \times V$, and we have
\begin{equation}\label{non-sc_decomposition_space}
M \cong \widetilde{M}/\Aut(\pi) \cong M_{\mathrm{comp}} \times V^\perp \times M_-,
\end{equation}
where $M_{\mathrm{comp}} = (M_+ \times V)/\Aut(\pi)$ is a compact symmetric space that is finitely covered by $M_+ \times V/L' \simeq M_+ \times \mathbb{T}^k$. We can also easily obtain an explicit description of $I^0(\widetilde{M},\pi)$. Indeed, since $\Aut(\pi) \subseteq I(M_+) \times I(M_0)$, we have $I(M_-) \subseteq I(\widetilde{M},\pi)$. By definition, every element of $I^0(M_+)$ centralizes $\Delta$. Similarly, every translation in $\R^n$ centralizes $L$. This implies that $I^0(\widetilde{M},\pi) = I^0(M_+) \times (\s{N}^0_{\s{SO}_n}(\Aut(\pi)) \ltimes \R^n) \times I^0(M_-)$. Now, every element $g$ of $\s{N}^0_{\s{SO}_n}(\Aut(\pi))$ must normalize $L$. Since $\s{N}^0_{\s{SO}_n}(\Aut(\pi))$ is connected, $g$ must centralize $L$ and is thus the identity on $V$. We deduce that
\begin{equation}\label{non-sc_decomposition_group}
I^0(\widetilde{M},\pi) = I^0(M_+) \times (\s{SO}(V^\perp) \ltimes \R^n) \times I^0(M_-).
\end{equation}
With respect to decompositions \eqref{non-sc_decomposition_space} and \eqref{non-sc_decomposition_group}, it is not difficult to see that $\Pi$ sends $I^0(M_+) \times V$ onto $I^0(M_{\mathrm{comp}})$, $\s{SO}(V^\perp) \ltimes V^\perp = I^0(V^\perp)$ onto itself, and $I^0(M_-)$ onto itself. This shows a posteriori that $I^0(M) = I^0(M_{\mathrm{comp}}) \times I^0(V^\perp) \times I^0(M_-)$.

We are now ready to fully reduce the classification problem for cohomogeneity-one actions on non-simply connected symmetric spaces to that on simply connected ones. It will prove useful to introduce the following bit of notation: given $\xi \in \g{r}^n \oplus \g{a}$, we will write $\s{H}_\xi$ for the connected Lie subgroup of $I^0(M_0) \times I^0(M_-)$ with Lie algebra $((\g{r}^n \oplus \g{a}) \ominus \xi) \loplus \g{n}$.  This corresponds to taking $\g{v} = \R\xi$ in Remark~\ref{remark:hyperpolar} and thus generalizes the subgroups $\s{H}_{E,X}$ (with $\xi = ||E||^2 X - ||X||^2 E$), \, $\R^{n-1} \times (\s{A} \ltimes \s{N})$ (for $\xi \in \g{r}^n$), and $\R^n \times \s{H}_X$ (for $\xi \in \g{a}$). Here and below, it will be convenient for us to identify $\g{r}^n \oplus \g{a}$ with $\R^n \times \s{A}$ by means of the exponential map. In particular, we can think of the normal vector $\xi$ as an element of $\R^n \times \s{A}$, which allows us to decompose it as $\xi = E_V + E_{V^\perp} + X$, where $E_V \in V$, $E_{V^\perp} \in V^\perp$, and $X \in \s{A} \simeq \g{a}$.

\begin{theorem}\label{th:nonsimplyconnected}
 	Let $M$ be a symmetric space and $\widetilde{M} = M_+ \times M_0 \times M_-$ its universal cover. Write $M = M_{\mathrm{comp}} \times V^\perp \times M_-$ as above. Let $\s{H}$ be a connected Lie group acting on $M$ properly, isometrically, and with cohomogeneity one. Then the action of $\s{H}$ is orbit-equivalent to one of the following:
 	\begin{enumerate}[\normalfont (a)]
 		\item The action of $\Pi(\s{H}_+ \times \R^n \times I^0(M_-))$, where $\s{H}_+$ is a closed connected subgroup of $I^0(M_+)$ acting on $M_+$ with cohomogeneity one.

        \item The action of $I^0(M_{\mathrm{comp}}) \times \s{H}_{V^\perp} \times I^0(M_-)$, where $\s{H}_{V^\perp}$ is a closed connected subgroup of $I^0(V^\perp)$ acting on $V^\perp$ with cohomogeneity one.
 		
 		\item The action of $I^0(M_{\mathrm{comp}}) \times V^\perp \times \s{H}_-$, where $\s{H}_-$ is a closed connected subgroup of $I^0(M_-)$ acting on $M_-$ with cohomogeneity one.

        \item The action of $\Pi(I^0(M_+) \times \s{H}_\xi)$ with $\xi = E_V + E_{V^\perp} + X \in \g{r}^n \oplus\g{a}$, where either $E_V\neq 0$ or $E\neq 0\neq X$, and there exists a nonzero $c \in \R$ such that $cE_V \in L^*$.
 	\end{enumerate}
 	Conversely, each of the above actions is proper and isometric and has cohomogeneity one.
\end{theorem}

\begin{remark}
    A few comments are in order. First of all,  implicit in the theorem is the fact that the actions in (a) and (d) are well-defined, i.e., that the corresponding descents actually exist. Notice that, by design, the actions in (b) and (c) decompose with respect to the decomposition $M = M_{\mathrm{comp}} \times V^\perp \times M_-$. In fact, even though it is written as a descent, the action in (a) also decomposes. To see this, note that $\s{H}_+ \times \R^n \times I^0(M_-)$ can be written as $(\s{H}_+ \times  V) \times V^\perp \times I^0(M_-)$. The descent of this subgroup is $\Pi(\s{H}_+ \times  V) \times V^\perp \times I^0(M_-)$, where $\Pi(\s{H}_+ \times  V) \subseteq I^0(M_{\mathrm{comp}})$. On the contrary, as can be traced in the proof below, the actions in (d) do not decompose. The assumption that either $E_V\neq 0$ or $E\neq 0\neq X$ is only added in (d) to avoid intersections with (b) and (c): if $E_V=X=0$, then $\s{H}_\xi = V \times (V^\perp \ominus E_{V^\perp}) \times (A \ltimes N)$, leading to (b). Similarly, for $E=0\neq X$, $\s{H}_\xi = \R^n \times ((A \ominus X) \ltimes N)$, which leads to (c). Finally, the other assumption on $E_V$ in (d), despite looking elaborate, is only there to ensure that the descended action is proper. It means that the line spanned by $E_V$ passes through (nonzero elements of) the dual lattice $L^*$ (unless $E_V = 0$, in which case the assumption is automatically satisfied). Loosely speaking, it says that $E_V$ is only allowed to run in directions that are ``rational'' with respect to $L$.
\end{remark}

We begin by dealing with arguably the most technical part of the proof: checking when the descent of $I^0(M_+) \times \s{H}_\xi$ is proper. Note that this is a closed connected subgroup of $I^0(M_+) \times \R^n \times (\s{A}\ltimes\s{N})$ acting on $\widetilde{M}$ with cohomogeneity one. In view of \eqref{non-sc_decomposition_group}, it is contained in $I^0(\widetilde{M},\pi)$ and thus descends to $M$. In Proposition \ref{prop:lifting_actions}(c), we worked out a general criterion for the descent of a proper action to be proper. We now refine that criterion in the special case of $I^0(M_+) \times \s{H}_\xi$, making it more user-friendly.

\begin{lemma}\label{lem:properness_of_descent_H_xi}
Given any $\xi = E_V + E_{V^\perp} + X \in \R^n \times \s{A}$, the following conditions are equivalent:

\begin{enumerate}[\normalfont (i)]
    \item There exists a nonzero $c \in \R$ such that $cE_V \in L^*$.
    \item The descent of $I^0(M_+) \times \s{H}_\xi$ acts on $M$ properly.
    \item The descent of $I^0(M_+) \times \s{H}_\xi$ acts on $M$ with closed orbits.
\end{enumerate}
\end{lemma}

\begin{proof}
According to Proposition \ref{prop:lifting_actions}(c), the descent of $I^0(M_+)\times\s{H}_\xi$ acts properly if and only if there exists a neighborhood $U$ of the identity in $I(\widetilde{M})$ such that for every $g \in \Aut(\pi)$, either $g \in I^0(M_+) \times \s{H}_\xi$ or $(I^0(M_+) \times \s{H}_\xi) \cap gU = \varnothing$. We now show that this condition is equivalent to (i). First of all, recall that the multiplication map
\begin{equation}\label{big_diffeomorphism_product}
    I^0(M_+) \times (\R^n \times \s{SO}_n) \times (\s{K}_- \times \s{A} \times \s{N}) \rightarrow I^0(M_+) \times I^0(M_0) \times I^0(M_-) \simeq I^0(\widetilde{M})
\end{equation}
is a diffeomorphism. In particular, for every element $g = (g_+, g_0, g_-) \in I^0(\widetilde{M})$, we have unique decompositions $g_0 = v \circ T$ with $v \in \R^n$ and $T \in \s{SO}_n$ and $g_- = g_{\s{K}_-} g_{\s{A}} g_{\s{N}}$ with $g_{\s{K}_-} \in \s{K}_-, \, g_{\s{A}} \in \s{A},$ and $g_{\s{N}} \in \s{N}$. This decomposition naturally depends on the order in which we multiply the factors. From now on, we fix the order as in \eqref{big_diffeomorphism_product}, since this will simplify our formulas.

Let $U$ be a neighborhood of the identity in $I(\widetilde{M})$. By shrinking it if necessary, we may assume without loss of generality that $U$ is connected and splits with respect to \eqref{big_diffeomorphism_product} as
\[
U = U_+ \times U_{\R^n} \times U_{\s{SO}_n} \times U_{\s{K}_-} \times U_{\s{A}} \times U_{\s{N}}.
\]
Given any $g = (g_+, v, e) \in \Aut(\pi) \subseteq \Delta \times L\subseteq I^0(M_+) \times I^0(M_0) \times I^0(M_-)$, we have
\begin{equation}\label{neighborhood_product_decomposition}
    g U = g_+ U_+ \times (v + U_{\R^n}) \times U_{\s{SO}_n} \times U_{\s{K}_-} \times U_{\s{A}} \times U_{\s{N}}.
\end{equation}
Note that an element $h = (h_+, h_0, h_-)$ of $I(\widetilde{M})$ lies in $I^0(M_+) \times \s{H}_\xi$ if and only if $h_+ \in I^0(M_+), \, h_0 \in \R^n, \, h_- = h_{\s{A}} h_{\s{N}} \in \s{A} \ltimes \s{N}$, and $(h_0, h_{\s{A}}) \perp \xi$, where the inner product on $\R^n \times A$ is simply carried over from $\g{r}^n \oplus \g{a}$. Therefore, the set $g U$ intersects $I^0(M_+) \times \s{H}_\xi$ nontrivially if and only if $g_+ \in I^0(M_+)$ and $(v + U_{\R^n}) \times U_{\s{A}}$ contains elements orthogonal to $\xi$. For small $U$, the latter just means that $\bilin{v}{\xi}$ is sufficiently small. Therefore, the condition in Proposition \ref{prop:lifting_actions}(c) reads here as follows: there exists $\delta > 0$ such that for every $(g_+, v) \in \Aut(\pi)$ with $g_+ \in I^0(M_+)$, either $\bilin{v}{\xi} = 0$ (that is, $v \in \s{H}_\xi$) or $|\bilin{v}{\xi}| \ge \delta$. The set of such vectors $v$ forms a sub-lattice of $L$, which we temporarily denote by $L''$. Note that $L''$ clearly contains $L' = \Aut(\pi) \cap L$ and is thus a subgroup of $L$ of finite index.
Since $\bilin{v}{\xi} = \bilin{v}{E_V}$, the above condition simply becomes the question of how $E_V$ is situated in $V$ with respect to the lattice $L''$. The condition that $\bilin{v}{E_V}$ is either zero or separated from zero by at least $\delta$ for every $v \in L''$ is equivalent to asking that the image of $L''$ under the functional $\bilin{-}{E_V}$ is a discrete subgroup of $\R$. Clearly, the latter is equivalent to asking that the image of the whole $L$ under $\bilin{-}{E_V}$ is discrete in $\R$.
If there exists $c \ne 0$ such that $cE_V \in L^*$ (that is, $\bilin{-}{cE_V}$ only takes integral values on $L$), then this is of course the case. Conversely, if $\bilin{L}{E_V}$ is a discrete subgroup of $\R$, then there exists a $\Z$-basis $e_1, \ldots, e_k$ of $L$ such that $\bilin{e_i}{E_V} = 0$ for $i = 1, \ldots, k-1$. If $\bilin{e_k}{E_V} = 0$, any $c$ works; otherwise, we take $c = \frac{1}{\bilin{e_k}{E_V}}$. Altogether, we conclude that (i) is equivalent to (ii).

Since (ii) readily implies (iii), to finish the proof of the lemma, we are left to show that (iii) implies (ii). Arguing by contradiction, suppose that the descent of $I^0(M_+) \times \s{H}_\xi \curvearrowright \widetilde{M}$ is not proper. Since the orbits of the descent are the images of the orbits of $I^0(M_+) \times \s{H}_\xi$ in $\widetilde{M}$ under $\pi$, we need to show that those images are nonclosed. We can regard the covering $\pi$ as a composition 
\[
\widetilde{M} \xrightarrow{\hspace{5pt} \pi' \,} \widetilde{M}/L' \xrightarrow{\hspace{5pt} \pi'' \,} (\widetilde{M}/L')/(\Aut(\pi)/L') = \widetilde{M}/\Aut(\pi) = M.
\]
We first prove that the images of the orbits of $I^0(M_+) \times \s{H}_\xi$ under $\pi'$ are nonclosed in $\widetilde{M}/L'$. Notice that $\pi'$ only affects $M_0$:
\[
\pi' \colon M_+ \times M_0 \times M_- \to M_+ \times M_0/L' \times M_-.
\]
Every orbit of $I^0(M_+) \times \s{H}_\xi$ in $\widetilde{M}$ is the product of $M_+$ with an orbit of $\s{H}_\xi$ in $M_0 \times M_-$. Therefore, it suffices to show that the images of the orbits of $\s{H}_\xi$ under $M_0 \times M_- \to M_0/L' \times M_-$ are nonclosed. If we replace $M_0$ with $\R^n$ and $M_-$ with its solvable model $\s{A} \ltimes \s{N}$, the latter map becomes a local isomorphism of Lie groups:
\[
\Psi \colon \R^n \times (\s{A} \ltimes \s{N}) \to \R^n/L' \times (\s{A}\ltimes \s{N}).
\]
The orbits of $\s{H}_\xi$ in $\R^n \times (\s{A} \ltimes \s{N})$ are simply its right cosets, and the images of those orbits under $\Psi$ are the right cosets of $\Psi(\s{H}_\xi)$. So we need to show that $\Psi(\s{H}_\xi)$ is a nonclosed subgroup of  $\R^n/L' \times (\s{A} \ltimes \s{N})$. By assumption, condition (ii) (and thus condition (i)) fails to hold. As we showed earlier, this is equivalent to saying that $\bilin{L}{\xi} = \bilin{L}{E_V}$ is a nondiscrete subgroup of $\R$. Since $L'$ has finite index in $L$, this is in turn equivalent to $\bilin{L'}{\xi}$ being nondiscrete in $\R$. But that means that for every neighborhood $U$ of the identity in $\R^n \times (\s{A} \ltimes \s{N})$, there exists $g \in L' \setminus \s{H}_\xi$ such that $g \hspace{1pt} U$ intersects $\s{H}_\xi$ nontrivially. Lemma \ref{lem:covering_closed_subgroup} then implies that the image $\Psi(\s{H}_\xi)$ is a nonclosed Lie subgroup of $\R^n/L' \times (\s{A} \ltimes \s{N})$, which was to be shown.

Now, let $S \subseteq \widetilde{M}$ by any orbit of $I^0(M_+) \times \s{H}_\xi$. We can write $S = M_+ \times C$, where $C \subseteq M_0 \times M_- \cong \R^n \times (\s{A} \ltimes \s{N})$ is a right coset of $\s{H}_\xi$. We know that $\pi'(S) = M_+ \times \Phi(C)$ is nonclosed in $M_+ \times M_0/L' \times M_-$, and we are left to show that its image under $\pi''$ is nonclosed in $M$. Since $\Phi(C)$ is nonclosed in $M_0/L' \times M_-$, its closure is a submanifold of strictly larger dimension, and the same is thus true about the closure of $\pi'(S)$. Note that $\pi''(\pi'(S))$ is closed if and only if $\pi''^{-1}(\pi''(\pi'(S))) = \bigcup_{g \in \Aut(\pi)/L'} g \pi'(S)$ is closed. As the group $\Aut(\pi)/L'$ is finite, we have
\[
\overline{\pi''^{-1}(\pi''(\pi'(S)))} \; = \; \overline{\bigcup_{g \in \Aut(\pi)/L'} g \pi'(S)} \; = \bigcup_{g \in \Aut(\pi)/L'} g \hspace{1pt} \overline{\pi'(S)},
\]
and the latter set is the union of submanifolds of strictly larger dimension. This implies that $\pi''(\pi'(S)) = \pi(S)$ is nonclosed, which completes the proof.
\end{proof}

With the lemma out of the way, we can now proceed to the proof of the theorem.

\begin{proof}[Proof of Theorem \ref{th:nonsimplyconnected}]
    We begin by checking that the actions listed in (a)-(d) are indeed proper, isometric, and of cohomogeneity one. For (b) and (c), this is trivial. It follows from our discussion earlier in the section that if $\widetilde{\s{H}} \subseteq I^0(\widetilde{M},\pi)$ is a closed connected subgroup acting on $\widetilde{M}$ with cohomogeneity one and satisfying the condition in Proposition \ref{prop:lifting_actions}(c), then it descends to a proper isometric cohomogeneity-one action $\Pi(\widetilde{\s{H}}) \curvearrowright M$. We use this approach with the actions in (a) and (d), as those are expressed as descents. For every subgroup $\s{H}_+ \subseteq I^0(M_+)$ as in (a), $\s{H}_+ \times \R^n \times I^0(M_-)$ is a closed connected subgroup of $I^0(\widetilde{M},\pi)$ thanks to \eqref{non-sc_decomposition_group}, and it clearly acts on $\widetilde{M}$ with cohomogeneity one. Thus, we need only verify that it satisfies the criterion from Proposition \ref{prop:lifting_actions}(c). Take a neighborhood $U_+$ of the identity in $I(M_+)$ such that $gU_+ \cap \s{H}_+ = \varnothing$ for every $g \in \Delta \setminus \s{H}_+$ (which exists because $\Delta$ is finite and $\s{H}_+$ is closed). Then $U = U_+ \times I(M_0) \times I(M_-)$ is a neighborhood of the identity in $I(\widetilde{M})$ that satisfies the condition. Similarly, for every $\xi$ as in (d), $I^0(M_+) \times \s{H}_\xi$ is a closed connected subgroup of $I^0(\widetilde{M},\pi)$ acting on $\widetilde{M}$ with cohomogeneity one. It then descends to $M$, and its descent acts properly thanks to Lemma \ref{lem:properness_of_descent_H_xi}.

    Now we prove that any proper isometric cohomogeneity-one action $\s{H} \curvearrowright M$ with $\s{H}$ connected is orbit-equivalent to one of the actions in (a)-(d). Without loss of generality, we may assume that $\s{H} \subseteq I^0(M)$ and write $\s{\widetilde{H}}$ for its lift. It follows from our discussion earlier in the section that $\s{\widetilde{H}}$ is a closed connected subgroup of $I^0(\widetilde{M}, \pi)$ acting on $\widetilde{M}$ with cohomogeneity one, and its descent is $\s{H}$. According to Theorem \ref{mainth}, the action of $\s{\widetilde{H}}$ on $\widetilde{M} = M_+ \times M_0 \times M_-$ either decomposes or is orbit-equivalent to the action of $I^0(M_+) \times \s{H}_{E,X}$ for some nonzero $E \in \g{r}^n$ and $X \in \g{a}$. We consider these cases separately.

    To begin with, suppose that the projection action $\s{\widetilde{H}} \curvearrowright M_+$ is of cohomogeneity one. In this case, the action of $\s{\widetilde{H}}$ on $\widetilde{M}$ has the same orbits as $\s{H}_+ \times \R^n \times I^0(M_-) \subseteq I(\widetilde{M}, \pi)$, where $\s{H}_+$ is some closed connected subgroup of $I^0(M_+)$ acting on $M_+$ with cohomogeneity one. In view of Proposition~\ref{prop:lifting_actions}(a), the descent of $\s{H}_+ \times \R^n \times I^0(M_-) \curvearrowright \widetilde{M}$ has the same orbits as the action of $\s{H}$ on $M$. This corresponds to item (a) on the list.

    Similarly, if the projection action $\s{\widetilde{H}} \curvearrowright M_-$ has cohomogeneity one, then the action of $\s{\widetilde{H}}$ on $\widetilde{M}$ has the same orbits as $I^0(M_+) \times \R^n \times \s{H}_- \subseteq I(\widetilde{M}, \pi)$, where $\s{H}_-$ is some closed connected subgroup of $I^0(M_-)$ acting on $M_-$ with cohomogeneity one. In the same vein, the descent of $I^0(M_+) \times \R^n \times \s{H}_-$ has the same orbits in $M$ as $\s{H}$. Note that, with respect to the decomposition $M = M_{\mathrm{comp}} \times V^\perp \times M_-$, this descent also has the same orbits as $I^0(M_{\mathrm{comp}}) \times V^\perp \times \s{H}_-$. This corresponds to item (c).

    Going further, suppose that the projection action $\s{\widetilde{H}} \curvearrowright M_0$ is of cohomogeneity one. In this case, the action of $\s{\widetilde{H}}$ on $\widetilde{M}$ has the same orbits as $I^0(M_+) \times \s{H}_0 \times I^0(M_-)$, where $\s{H}_0$ is a closed connected subgroup of $I^0(M_0)$ acting on $M_0$ with cohomogeneity one. Owing to the classification of cohomogeneity-one actions on Euclidean spaces, we know that the action of $\s{H}_0$ is orbit-equivalent via an element of $\R^n$ to the action of the group $\s{SO}(W^\perp)\times W$, where $W\subsetneq \R^n$ is some proper subspace. If $\dim(W) = n-1$, the orbits are affine hyperplanes parallel to $W$. Otherwise, they are generalized cylinders around $W$ (which includes spheres when $W = \{0\}$). First, assume that $V\subseteq W$. In this case, we see that the action of $\s{\widetilde{H}}$ is orbit-equivalent to that of
    \begin{equation}\label{subgroup_V_in_W}
    I^0(M_+) \times V \times \big( \s{SO}(W^\perp) \times (W \ominus V) \big) \times I^0(M_-).
    \end{equation}
    Here $\s{SO}(W^\perp) \times (W \ominus V)$ is a closed connected subgroup of $I^0(V^\perp)$ acting on $V^\perp$ with cohomogeneity one. In view of \eqref{non-sc_decomposition_group} and the subsequent paragraph, the subgroup \eqref{subgroup_V_in_W} descends to $I^0(M_{\mathrm{comp}}) \times \left(\s{SO}(W^\perp) \times (W \ominus V)\right) \times I^0(M_-)$, whose action on $M$ is orbit-equivalent to that of $\s{H}$ thanks to Proposition \ref{prop:lifting_actions}(a). This corresponds to item (b) on the list. We claim that the assumption $V \subseteq W$ automatically holds true whenever $\dim(W) < n-1$. Indeed, according to Proposition \ref{prop:lifting_actions}(b), the orbits of $\s{\widetilde{H}}$ (and thus of $I^0(M_+) \times (\s{SO}(W^\perp) \times W) \times I^0(M_-)$) must get permuted by the action of $L' \subseteq \Aut(\pi) \cap I^0(\widetilde{M},\pi)$. Since $L'$ spans $V$, this easily implies that $V$ must be contained in $W$. 
    
     We are thus only left to consider the case when $\dim(W) = n-1$ and $V \not\subseteq W$. Under these assumptions, we can write $W = \R^n \ominus E$ for some $E$ with $E_V \neq 0$. We see that $\s{\widetilde{H}}$ has the same orbits in $\widetilde{M}$ as $I^0(M_+) \times \s{H}_E$. We already know that the latter subgroup descends to $M$ and thus has the same orbits as $\s{H}$ thanks to Proposition \ref{prop:lifting_actions}(a). Since those orbits are closed, Lemma \ref{lem:properness_of_descent_H_xi} implies that there exists a nonzero $c \in \R$ such that $c E_V \in L^*$. This corresponds to item (d) with $\xi=E$.
    
    Finally, we come to the case when the action of $\s{\widetilde{H}}$ on $\widetilde{M}$ is orbit-equivalent to that of $I^0(M_+) \times \s{H}_{E,X}$ for some nonzero $E$ and $X$. Note that the latter is a subgroup of $I^0(\widetilde{M},\pi)$ and thus it descends to $M$. As follows from our discussion in Section 4 and in particular in Lemma \ref{lemma:solvablesubgroupfoliations}, an orbit equivalence between the actions of $\s{\widetilde{H}}$ and $I^0(M_+) \times \s{H}_{E,X}$ can actually be achieved by an element of $I^0(M_-) \subseteq I^0(\widetilde{M},\pi)$. Proposition \ref{prop:lifting_actions}(a) then implies that the action of $\Pi(I^0(M_+) \times \s{H}_{E,X})$ on $M$ is orbit-equivalent to that of $\s{H}$. Now, as we mentioned earlier, $\s{H}_{E,X} = \s{H}_\xi$, where $\xi = E' + X'$ with $E' = - ||X||^2 E$ and $X' = ||E||^2 X$. As the orbits of the descent are closed, Lemma \ref{lem:properness_of_descent_H_xi} implies that there exists a nonzero $c \in \R$ such that $cE'_{V} \in L^*$. This also corresponds to item (d).
\end {proof}

%% file: Main.bbl
\begin{thebibliography}{MSSV25}
	\bibitem[AB15]{AleBet} \textsc{M. M. Alexandrino} and \textsc{R. G. Bettiol}, \emph{Lie groups and geometric aspects of isometric actions},
	Springer, Cham, 2015.  \href{https://doi.org/10.1007/978-3-319-16613-1}{doi:10.1007/978-3-319-16613-1}.
	
	\bibitem[Ale75]{Alekseevski} \textsc{D. V. Alekseevskii}, Homogeneous Riemannian spaces of negative curvature, \emph{Mat. Sb. (N.S.)},
	{\bf 96} (1975), no. 138, 93--117.
	\href{https://doi.org/10.1070/SM1975v025n01ABEH002200}{doi:10.1070/SM1975v025n01ABEH002200}
	
	%\bibitem[Ale11]{Alexandrino:RIMA} \textsc{M. M. Alexandrino}, On polar foliations and the fundamental group, \emph{Results Math.}, {\bf 60} (2011), no.1--4, 213--223. \href{https://doi.org/10.1007/s00025-011-0171-4}{doi:10.1007/s00025-011-0171-4}.
	
	\bibitem[AR04]{AR:Acta} \textsc{U. Abresch} and \textsc{H. Rosenberg}, A hopf differential for constant mean curvature surfaces in $\mathbb{S}^2\times\R$ and $\mathbb{H}^2\times \R$, \emph{Acta Math.}, {\bf 193} (2004), no. 2, 141--174.
	\href{https://doi.org/10.1007/BF02392562}{doi:10.1007/BF02392562}. 
	
	\bibitem[AW76]{AW:homogeneous} \textsc{R. Azencott} and \textsc{E. N. Wilson}, Homogeneous manifolds with negative curvature I, \emph{Trans. Amer.
	Math. Soc.}, {\bf 215} (1976), 323--362. 
	\href{https://doi.org/10.2307/1999731}{doi:10.2307/1999731}.
	
	\bibitem[BB01]{BB:crelle} \textsc{J. Berndt} and \textsc{M. Br\"uck}, Cohomogeneity one actions on hyperbolic spaces, \emph{J. Reine Angew. Math.}, {\bf 541} (2001), 209--235. 
	\href{https://doi.org/10.1515/crll.2001.093}{doi:10.1515/crll.2001.093}.
	
	\bibitem[BCO16]{BCO:book} \textsc{J. Berndt}, \textsc{S. Console}, and \textsc{C. E. Olmos}. \emph{Submanifolds and holonomy}, second edition, Monographs and Research
	Notes in Mathematics, CRC Press, Boca Raton, 2016. 
	\href{https://doi.org/10.1201/b19615}{doi:10.1201/b19615}.
	
	\bibitem[BD15]{BD:tg} \textsc{J. Berndt} and \textsc{M. Dom\'inguez-V\'azquez}. Cohomogeneity one actions on some noncompact symmetric
	spaces of rank two, \emph{Transform. Groups}, {\bf 20} (2015), no. 4, 921--938.
	\href{https://doi.org/10.1007/s00031-015-9299-8}{doi:10.1007/s00031-015-9299-8}.
	
	\bibitem[BDT10]{BDT:jdg} \textsc{J. Berndt}, \textsc{J. C. D\'iaz-Ramos}, and \textsc{H. Tamaru}, Hyperpolar homogeneous foliations on symmetric spaces of noncompact type, \emph{J. Differential Geom.}, {\bf 86} (2010), no. 2, 191--235.
	\href{https://doi.org/10.4310/jdg/1299766787}{doi:10.4310/jdg/1299766787}.
	
	\bibitem[B\"oh98]{Boehm} \textsc{C. B\"ohm}, Inhomogeneous Einstein metrics on low-dimensional spheres and other low-dimensional
	spaces, \emph{Invent. Math.}, {\bf 134} (1998),  no. 1, 145--176.
	\href{https://doi.org/10.1007/s002220050261}{doi:10.1007/s002220050261}.
	
	\bibitem[Bou98]{bourbaki_general_topology_1-4} \textsc{N. Bourbaki}, \emph{General topology. Chapters 1--4}, Elements of Mathematics, Springer-Verlag,	Berlin, 1995.
	\href{https://doi.org/10.1007/978-3-642-61701-0}{doi:10.1007/978-3-642-61701-0}.
	
	\bibitem[BS89]{BrySal} \textsc{R.~L.~Bryant} and {S.~M.~Salamon}, On the construction of some complete metrics with exceptional holonomy, \emph{Duke Math. J.}, {\bf 58} (1989), no.~3, 829--850.
	\href{https://doi.org/10.1215/S0012-7094-89-05839-0}{doi:10.1215/S0012-7094-89-05839-0}.
	
	\bibitem[BT03]{BT:jdg} \textsc{J.~Berndt} and \textsc{H.~Tamaru}, Homogeneous codimension one foliations on noncompact symmetric spaces, \emph{J. Differential Geom.}, {\bf 63} (2003), no.~1, 1--40.
	\href{https://doi.org/10.4310/jdg/1080835656}{doi:10.4310/jdg/1080835656}.
	
	\bibitem[BT04]{BT:tohoku} \bysame, Cohomogeneity one actions on noncompact symmetric spaces with a totally geodesic singular orbit, \emph{Tohoku Math. J. (2)}, {\bf 56} (2004), no.~2, 163--177. \href{https://doi.org/10.2748/tmj/1113246549}{doi:10.2748/tmj/1113246549}.
	
	\bibitem[BT07]{BT:tams} \bysame, Cohomogeneity one actions on noncompact symmetric spaces of rank one, \emph{Trans. Amer. Math. Soc.}, {\bf 359} (2007), no.~7, 3425--3438.
	\href{https://doi.org/10.1090/S0002-9947-07-04305-X}{doi:10.1090/S0002-9947-07-04305-X}.
	
	\bibitem[BT13]{BT:crelle} \bysame, Cohomogeneity one actions on symmetric spaces of noncompact type, \emph{J. Reine Angew. Math.}, {\bf 683} (2013), 129--159. \href{https://doi.org/10.1515/crelle-2012-0002}{doi:10.1515/crelle-2012-0002}.
	
	\bibitem[Car38]{Cartan} \textsc{\'E. Cartan}, Familles de surfaces isoparam\`etriques dans les espaces \`a courbure constante, \emph{Ann. Mat. Pura Appl.}, {\bf 17} (1938), no.~1, 177--191.
	\href{https://doi.org/10.1007/BF02410700}{doi:10.1007/BF02410700}.
	
	\bibitem[CS19]{ChavesSantos} \textsc{R. M. B. Chaves} and \textsc{E. Santos}, Hypersurfaces with constant principal curvatures in $\mathbb{S}^n\times \R$ and $\mathbb{H}^n\times \R $, \emph{Illinois J. Math.}, {\bf 63} (2019), no.~4, 551--574.
	\href{https://doi.org/10.1215/00192082-8018599}{doi:10.1215/00192082-8018599}.
	
	\bibitem[DDO23a]{DDO:adv} \textsc{J. C. D\'iaz-Ramos}, \textsc{M. Dom\'inguez-V\'azquez}, and \textsc{T. Otero}, Cohomogeneity one actions on symmetric spaces of noncompact type and higher rank, \emph{Adv. Math.}, {\bf 428} (2023), Paper no.~109165.
	\href{https://doi.org/10.1016/j.aim.2023.109165}{doi:10.1016/j.aim.2023.109165}.
	
	\bibitem[DDO23b]{DDO:reag} \bysame, Homogeneous hypersurfaces in symmetric spaces, in: \emph{New trends in geometric analysis. Spanish Network of Geometric Analysis 2007–2021}, RSME Springer Ser., vol 10., Springer, Cham, 2023, pp. 141--190.
	\href{https://doi.org/10.1007/978-3-031-39916-9_5}{doi:10.1007/978-3-031-39916-9$\_$5}.
	
	\bibitem[DDR21]{DDR:crelle} \textsc{J. C. D\'iaz-Ramos}, \textsc{M. Dom\'inguez-V\'azquez}, and \textsc{A. Rodr\'iguez-V\'azquez}, Homogeneous and inhomogeneous isoparametric hypersurfaces in rank one symmetric spaces, \emph{J. Reine Angew. Math.}, {\bf 779} (2021), 189--222.
	\href{https://doi.org/10.1515/crelle-2021-0043}{doi:10.1515/crelle-2021-0043}.
	
	\bibitem[DM21]{DVManzano} \textsc{M. Dom\'inguez-V\'azquez} and \textsc{J. M. Manzano}, Isoparametric surfaces in $\mathbb{E}(\kappa,\tau)$-spaces, \emph{Ann. Sc. Norm. Super. Pisa Cl. Sci. (5)}, {\bf 22} (2021), no.~1, 269--285.
	\href{https://doi.org/10.2422/2036-2145.201805_006}{doi:10.2422/2036-2145.201805$\_$006}.
	
	\bibitem[Dom15]{DV:IMRN} \textsc{M. Dom\'inguez-V\'azquez}, Canonical extension of submanifolds and foliations in noncompact symmetric spaces, \emph{Int. Math. Res. Not. IMRN}, {\bf 2015} (2015), no.~22, 12114--12125.
	\href{https://doi.org/10.1093/imrn/rnv072}{doi:10.1093/imrn/rnv072}.
	
	\bibitem[DST21]{DVSLT} \textsc{M. Dom\'inguez-V\'azquez}, \textsc{V. Sanmart\'in-L\'opez}, and \textsc{H. Tamaru}, Codimension one Ricci soliton subgroups of solvable Iwasawa groups, \emph{J. Math. Pures Appl. (9)}, {\bf 152} (2021), 69--93.
	\href{https://doi.org/10.1016/j.matpur.2021.05.008}{doi:10.1016/j.matpur.2021.05.008}.
	
	\bibitem[FH17]{FH} \textsc{L. Foscolo} and \textsc{M. Haskins}, New $G_2$-holonomy cones and exotic nearly K\"ahler structures on $\mathbb{S}^6$ and $\mathbb{S}^3\times \mathbb{S}^3$, \emph{Ann. of Math. (2)}, {\bf 185} (2017), no.~1, 59--130.
	\href{https://doi.org/10.4007/annals.2017.185.1.2}{doi:10.4007/annals.2017.185.1.2}.
	
	\bibitem[FKM81]{FKM} \textsc{D. Ferus}, \textsc{H. Karcher}, and \textsc{H. F. M\"unzner}, Cliffordalgebren und neue isoparametrische Hyperfl\"achen, \emph{Math. Z.}, {\bf 177} (1981), no~4, 479--502.
	\href{https://doi.org/10.1007/BF01219082}{doi:10.1007/BF01219082}.
	
	\bibitem[GW88]{GW:isomsolv} \textsc{C. S. Gordon} and \textsc{E. N. Wilson}, Isometry groups of Riemannian solvmanifolds, \emph{Trans. Amer. Math. Soc.}, {\bf 307} (1988), no.~1, 245--269.
	\href{https://doi.org/10.2307/2000761}{doi:10.2307/2000761}.
	
	\bibitem[GWZ08]{GWZ} \textsc{K. Grove}, \textsc{B. Wilking}, and \textsc{W. Ziller}, Positively curved cohomogeneity one manifolds and 3-Sasakian geometry, \emph{J. Differential Geom.}, {\bf 78} (2008), no.~1, 33--111. 
	\href{https://doi.org/10.4310/jdg/1197320603}{doi:10.4310/jdg/1197320603}.
	
	\bibitem[Hel78]{Helgason} \textsc{S. Helgason}, \emph{Differential geometry, Lie groups, and symmetric spaces}, Pure and Applied
	Mathematics, vol.~80, Academic Press, Inc., New York-London, 1978.
	
	\bibitem[HL71]{Hsiang-Lawson} \textsc{W.-Y. Hsiang} and \textsc{H. B. Lawson, Jr.}, Minimal submanifolds of low cohomogeneity, \emph{J. Differential Geometry}, {\bf 5} (1971), 1--38. \href{https://doi.org/10.4310/jdg/1214429775}{doi:10.4310/jdg/1214429775}.
		
	\bibitem[Kna02]{Knapp} \textsc{A. W. Knapp}, \emph{Lie groups beyond an introduction}, second edition, Progress in Mathematics, vol.~40, Birkh\"auser Boston, Inc., Boston, 2002.
	
	\bibitem[Kol02]{Kollross:tams} \textsc{A. Kollross}, A classification of hyperpolar and cohomogeneity one actions, \emph{Trans. Amer. Math. Soc.}, {\bf  354} (2002), no.~2, 571--612.
	\href{https://doi.org/10.1090/S0002-9947-01-02803-3}{doi:10.1090/S0002-9947-01-02803-3}.
	
	\bibitem[Kol11]{kollross_duality} \bysame, Duality of symmetric spaces and polar actions, \emph{J. Lie Theory}, {\bf 21} (2011), no.~4, 961--986.
	\href{https://doi.org/10.5802/jolt.657}{doi:10.5802/jolt.657}.
	
	\bibitem[Kol17]{Kollross:reducible} \bysame, Hyperpolar actions on reducible symmetric spaces, \emph{Transform. Groups}, {\bf 22} (2017), no.~1, 207--228.
	\href{https://doi.org/10.1007/s00031-016-9384-7}{doi:10.1007/s00031-016-9384-7}.
	
	\bibitem[KT13]{KT:GeomDed} \textsc{A. Kubo} and \textsc{H. Tamaru}, A sufficient condition for congruency of orbits of Lie groups and some applications, \emph{Geom. Dedicata}, {\bf 167} (2013), 233--238.
	\href{https://doi.org/10.1007/s10711-012-9811-4}{doi:10.1007/s10711-012-9811-4}.
	
	\bibitem[Lor25]{LN:thesis} \textsc{J. M. Lorenzo-Naveiro}, Submanifolds and actions on homogeneous manifolds, Ph.D. thesis, Universidade de Santiago de Compostela, 2025.
	\href{https://minerva.usc.gal/entities/publication/d6d6eab8-fb84-4b6b-8ce1-d2d3ed9d6cb4}{Available at: https://minerva.usc.gal/entities/publication/d6d6eab8-fb84-4b6b-8ce1-d2d3ed9d6cb4}.
	
	\bibitem[LP24]{LimaPipoli} \textsc{R. F. de Lima} and \textsc{G. Pipoli}, Isoparametric hypersurfaces of $\mathbb{H}^n\times\R$ and $\mathbb{S}^n\times\R$, 2024.
	\href{https://arxiv.org/abs/2411.11506}{arXiv:2411.11506}.
	
	\bibitem[LP25]{LimaPipoli2} \bysame, Isoparametric hypersurfaces in products of simply connected space forms, 2025.
	\href{https://arxiv.org/abs/2511.12527}{arXiv:2511.12527}.
	
	%\bibitem[Lyt10]{Lyt:gded} \textsc{A. Lytchak}, Geometric resolution of singular Riemannian foliations, \emph{Geom. Dedicata}, {\bf 149} (2010), 379--395.
	%\href{https://doi.org/10.1007/s10711-010-9488-5}{doi:10.1007/s10711-010-9488-5}.

    \bibitem[Mic08]{Michor} \textsc{P. W. Michor}, \emph{Topics in differential geometry}, Graduate Studies in Mathematics, American Mathematical Society, Providence, RI, 2008. 
	
	\bibitem[Mon50]{Montgomery} \textsc{D. Montgomery}, Simply connected homogeneous spaces, {\emph Proc. Amer. Math. Soc.}, {\bf 1} (1950), 467--469.
	\href{https://doi.org/10.2307/2032314}{doi:10.2307/2032314}.
	
	\bibitem[Mos61]{Mostow:AnnalsMaximal} \textsc{G. D. Mostow}, On maximal subgroups of real Lie groups, \emph{Ann. of Math. (2)}, {\bf 74} (1961), 503--517.
	\href{https://doi.org/10.2307/1970295}{doi:10.2307/1970295}.
	
	\bibitem[MSSV25]{MSSV} \textsc{F. Manfio}, \textsc{J. B. M. dos Santos}, \textsc{J. P. dos Santos}, and \textsc{J. Van der Veken}, Hypersurfaces of $\mathbb{S}^3\times\R$ and $\mathbb{H}^3\times \R$ with constant principal curvatures, \emph{J. Geom. Phys.}, {\bf 213} (2025), Paper No. 105495.
	\href{https://doi.org/doi:10.1016/j.geomphys.2025.105495}{doi:10.1016/j.geomphys.2025.105495}.
	
	\bibitem[OV94]{OnishchikVinberg} \textsc{A. L. Onishchik} and \textsc{\`E.\ B. Vinberg} (eds.), \emph{Lie groups and Lie algebras, III}, Encyclopaedia of Mathematical Sciences, vol. 41, Springer-Verlag, Berlin, 1994.
	
	\bibitem[Seg38]{Segre} \textsc{B. Segre}, Famiglie di ipersuperficie isoparametriche negli spazi euclidei ad un qualunque numero di dimensioni, \emph{Atti Accad. Naz. Lincei. Rend. Cl. Sci. Fis. Mat. Natur.}, {\bf 27} (1938), 203--207.
	
	\bibitem[Sol21]{Solonenko:reducible} \textsc{I. Solonenko}, Homogeneous codimension-one foliations on reducible symmetric spaces of noncompact type, 2021.
	\href{https://arxiv.org/abs/2112.02189}{arXiv:2112.02189}.
	
	\bibitem[Sol23]{Solonenko} \bysame, Classification of homogeneous hypersurfaces in some noncompact symmetric spaces of rank two, \emph{Ann. Mat. Pura Appl. (4)}, {\bf 202} (2023), no. 6, 2915--2946.
	\href{https://doi.org/10.1007/s10231-023-01345-8}{doi:10.1007/s10231-023-01345-8}.
	
	\bibitem[Sol24]{solonenko_thesis} \bysame, Homogeneous hypersurfaces in Riemannian symmetric spaces, Ph.D. thesis, King’s College London, 2024.
	\href{https://kclpure.kcl.ac.uk/portal/en/studentTheses/homogeneous-hypersurfaces-in-riemannian-symmetric-spaces}{Available at: https://kclpure.kcl.ac.uk/portal/en/studentTheses/homogeneous-hypersurfaces-in-riemannian-symmetric-spaces}.
	
	\bibitem[SS23]{SantosSantos} \textsc{J. B. M. dos Santos} and \textsc{J. P. dos Santos}, Isoparametric hypersurfaces in product spaces, \emph{Differential Geom. Appl.}, {\bf 88} (2023), Paper No. 102005.
	\href{https://doi.org/10.1016/j.difgeo.2023.102005}{doi:10.1016/j.difgeo.2023.102005}.
	
	\bibitem[SS25]{SLS:arXiv} \textsc{I. Solonenko} and \textsc{V. Sanmart\'in-L\'opez}, Classification of cohomogeneity-one actions on symmetric spaces of noncompact type, 2025.
	\href{https://arxiv.org/abs/2501.05553}{arXiv:2501.05553}.
	
	\bibitem[TXY25]{TXY:arxiv} \textsc{H. Tan}, \textsc{Y. Xie}, and \textsc{W. Yan}, Isoparametric hypersurfaces in $\mathbb{S}^n\times\R^m$ and $\mathbb{H}^n\times\R^m$, 2025.
	\href{https://arxiv.org/abs/2511.07782}{arXiv:2511.07782}.
	
	\bibitem[Wol11]{wolf_constant_curvature} \textsc{J. A. Wolf}, \emph{Spaces of constant curvature}, sixth edition, AMS Chelsea Publishing, Providence, RI, 2011.
	\href{https://doi.org/10.1090/chel/372}{doi:10.1090/chel/372}.
\end{thebibliography}
